\documentclass[10pt,a4paper,twoside,reqno]{amsart}
\input prooftree.sty
\usepackage[utf8]{inputenc}
\usepackage[OT1]{fontenc}
\usepackage[english]{babel}
\usepackage{dsfont}
\usepackage{amsfonts}
\usepackage{amsmath}
\usepackage{wasysym}
\usepackage{amsthm}
\usepackage{amssymb}
\usepackage{float}
\usepackage{textcomp}
\usepackage{xcolor}
\usepackage{xspace}
\usepackage{graphicx} 
\usepackage{fancybox}
\usepackage{fancyhdr}
\usepackage{mathrsfs}
\usepackage{mathtools}
\usepackage{tikz}
\usetikzlibrary{arrows}
\usetikzlibrary{arrows.meta}
\usetikzlibrary{calc}
\usepackage{centernot}
\usepackage{listings}
\usepackage{faktor}
\usepackage{bm}
\usepackage{setspace}
\usepackage{hyperref}
\usepackage[capitalize]{cleveref}
\usepackage{newlfont}
\usepackage{geometry}
\usepackage{ccaption}
\usepackage{tipa}
\usepackage{units}
\usepackage[autostyle,italian=guillemets]{csquotes}
\usepackage{comment}

\usepackage{stmaryrd}
\usepackage{bbm}
\usepackage{guit}
\usepackage{tikz-cd}
\usepackage{enumerate}

\usepackage[shortlabels]{enumitem}
\everymath{\displaystyle}

\theoremstyle{definition}
\newtheorem{defi}{Definition}[subsection]
\newtheorem{nota}[defi]{Notation}
\theoremstyle{plain}
\newtheorem{theo}[defi]{Theorem}
\newtheorem{prop}[defi]{Proposition}
\newtheorem{cor}[defi]{Corollary}
\newtheorem{lemma}[defi]{Lemma}
\theoremstyle{remark}
\newtheorem{es}[defi]{Example}
\newtheorem{ese}[defi]{Examples}
\newtheorem{running}[defi]{Running example}
\newtheorem{rmk}[defi]{Remark}

\newtheoremstyle{note}% ⟨name⟩
{3pt}% ⟨Space above⟩1
{3pt}% ⟨Space below⟩1
{}% ⟨Body font⟩
{}% ⟨Indent amount⟩2
{\textup}% ⟨Theorem head font⟩
{.}% ⟨Punctuation after theorem head⟩
{.5em}% ⟨Space after theorem head⟩3
{}% ⟨Theorem head spec (can be left empty, meaning ‘normal’)⟩

\theoremstyle{note}
\newtheorem{drule}{Rule}

\newtheorem{trule}{Rule}

\newtheorem{frule}{Rule}

\newcommand{\B}{{\mathcal B}}
\newcommand{\C}{{\mathcal C}}
\newcommand{\D}{{\mathcal D}}
\newcommand{\E}{{\mathcal E}}
\newcommand{\F}{{\mathcal F}}
\newcommand{\G}{{\mathcal G}}
\renewcommand{\H}{{\mathcal H}}

\newcommand{\J}{{\mathcal J}}
\newcommand{\K}{{\mathcal K}}
\renewcommand{\L}{{\mathcal L}}
\newcommand{\M}{M}
\newcommand{\N}{N}

\newcommand{\R}{{\mathcal R}}
\renewcommand{\S}{{\mathcal S}}

\newcommand{\catfont}{\mathsf}
\newcommand\Cat{\catfont{Cat}}

\newcommand\FinCat{\catfont{Cat}_{\textup{fp}}}
\newcommand{\Gpd}{\catfont{Gpd}}

\newcommand{\blim}{{\bb b}$-$\mathsf{lim}}

\newcommand{\Lex}{\catfont{Lex}}
\newcommand{\Fc}{\catfont{Fc}}
\newcommand{\Pb}{\catfont{Pb}}
\newcommand{\Rex}{\catfont{Rex}}
\newcommand{\RLex}{\catfont{RLex}}
\newcommand{\Reg}{\catfont{Reg}}
\newcommand{\Radj}{\catfont{Radj}}
\newcommand{\Ladj}{\catfont{Ladj}}
\newcommand{\Mult}{\catfont{Mult}}
\newcommand{\RMult}{\catfont{RMult}}
\newcommand{\MonCat}{\catfont{MonCat}}

\newcommand{\Ex}{\catfont{Ex}}
\newcommand{\Mal}{\catfont{Mal}}
\newcommand{\RMal}{\catfont{RMal}}
\newcommand{\SmAb}{\catfont{SmAb}}
\newcommand{\Ptm}{\catfont{Ptm}}
\newcommand{\IsoFib}{\catfont{IsoFib}}
\newcommand{\GFib}{\catfont{GFib}}
\newcommand{\Cmp}{\catfont{Cmp}}
\newcommand{\Clan}{\catfont{Clan}}

\newcommand\IsoReg{\catfont{IsoReg}}

\newcommand\Set{\catfont{Set}}

\newcommand{\LL}{{\mathbb L}}

\newcommand{\TT}{{\mathbb T}}

\newcommand\Syn{\catfont{Syn}}
\newcommand\Str{\catfont{Str}}
\newcommand\Mod{\catfont{Mod}}

\newcommand\dom{\mathsf{dom}}
\newcommand\cod{\mathsf{cod}}
\newcommand\comp{\mathsf{comp}}
\newcommand\fst{\mathsf{fst}}
\newcommand\snd{\mathsf{snd}}
\newcommand\id{\mathsf{id}}
\newcommand{\defeq}{\equiv_{\mathrm{def}}}
\newcommand{\br}[1]{\llbracket{#1}\rrbracket}
\newcommand{\bb}[1]{\,\mathbbm{#1}}

\newcommand{\tx}{\mathsf}

\newcommand{\dw}{\downarrow\kern-3pt}
\newcommand{\op}{^\textnormal{op}}

\makeatother

\makeatletter
\newcommand{\changeoperator}[1]{%
	\csletcs{#1@saved}{#1@}%
	\csdef{#1@}{\changed@operator{#1}}%
}
\newcommand{\changed@operator}[1]{%
	\mathop{%
		\mathchoice{\textstyle\csuse{#1@saved}}
		{\csuse{#1@saved}}
		{\csuse{#1@saved}}
		{\csuse{#1@saved}}%
	}%
}
\makeatother

\makeatletter
\def\@tocline#1#2#3#4#5#6#7{\relax
	\ifnum #1>\c@tocdepth % then omit
	\else
	\par \addpenalty\@secpenalty\addvspace{#2}%
	\begingroup \hyphenpenalty\@M
	\@ifempty{#4}{%
		\@tempdima\csname r@tocindent\number#1\endcsname\relax
	}{%
		\@tempdima#4\relax
	}%
	\parindent\z@ \leftskip#3\relax \advance\leftskip\@tempdima\relax
	\rightskip\@pnumwidth plus4em \parfillskip-\@pnumwidth
	#5\leavevmode\hskip-\@tempdima
	\ifcase #1
	\or\or \hskip 1em \or \hskip 2em \else \hskip 3em \fi%
	#6\nobreak\relax
	\hfill\hbox to\@pnumwidth{\@tocpagenum{#7}}\par% <---- \dotfill -> \hfill
	\nobreak
	\endgroup
	\fi}
\makeatother

\changeoperator{sum}
\changeoperator{prod}
\changeoperator{coprod}

\newcommand{\ie}{{i.e.}\,}
\newcommand{\eg}{{e.g.}\,}
\newcommand{\cf}{cf.\xspace}
\newcommand{\etc}{etc.\xspace}
\newcommand{\co}{\colon}
\newcommand{\pr}{\mathsf{pr}}
\newcommand{\pru}{\mathsf{pr}_u}
\newcommand{\prd}{\mathsf{pr}_d}
\newcommand{\prl}{\mathsf{pr}_l}
\newcommand{\prr}{\mathsf{pr}_r}
\newcommand{\prp}{\mathsf{prop}}

\newcommand{\ffk}{{\sc ffk}}
\newcommand{\var}{\mathsf{var}}
\setlist[enumerate]{label=(\roman*),itemsep=0ex}
\crefname{equation}{}{}

\crefname{lemma}{Lemma}{Lemmas}
\crefname{theo}{Theorem}{Theorems}
\crefname{rmk}{Remark}{Remarks}
\crefname{defi}{Definition}{Definitions}
\crefname{conj}{Conjecture}{Conjectures}
\crefname{ex}{Example}{Examples}
\crefname{ese}{Example}{Examples}
\crefname{sec}{Section}{Sections}
\crefname{prop}{Proposition}{Propositions}
\crefname{hyp}{Assumption}{Assumptions}
\crefname{running}{Example}{Examples} 
\crefname{drule}{Rule}{Rules}
\crefname{trule}{Rule}{Rules}
\crefname{frule}{Rule}{Rules}

\title[isoregular theories]{isoregular theories, accessible 2-categories, \\ and free constructions}
\author{N. Gambino and G. Tendas}
\address{Department of Mathematics, The University of Manchester, Alan Turing Building, Oxford
Road,  Manchester M13 9PL, United Kingdom}
\email{nicola.gambino@manchester.ac.uk}
\email{giacomo.tendas@manchester.ac.uk}
\date{\today}

\begin{document}
	
\begin{abstract} We introduce isoregular theories, in which it is possible to express 
existential quantification up to unique isomorphism, as typically used  to characterise
category-theoretic universal constructions, such as limits. We then develop a functorial semantics
for isoregular theories and prove that their 2-categories of models are accessible with
flexible limits. We apply these results by showing that a number of 2-categories
of interest in general category theory, categorical algebra, and categorical logic 
are models of isoregular theories, thereby establishing that they are accessible 
2-categories with flexible limits and obtaining a number of new free constructions. 
\end{abstract}


	
\maketitle
\begin{comment}
	{\small
		\noindent{\bf Keywords:} \\
		{\bf Mathematics Subject Classification:} \\
		{\bf Competing Interests:} the author declares none.
	}
\end{comment}

\setcounter{tocdepth}{1}
\tableofcontents

\section{Introduction}

\subsection{Context and motivation}
Fields, Hilbert spaces, and other set-based mathematical structures can be studied 
uniformly not only via category theory or mathematical logic, but also via a combination
thereof, known as categorical logic,  founded by Lawvere~\cite{Law63:articolo}. Categorical logic is based on three key ideas: first, logical theories can be identified with certain categories (via the construction of syntactic categories, for example); secondly, set-theoretic models of a theory can be regarded as structure-preserving functors into the category of sets;
and, finally, categories of models can often be characterised in a  purely categorical, syntax-free, way. 
For example, algebraic theories are identified with categories with finite products, their models are regarded as product-preserving functors into the category of sets, and their categories of models
are characterised as algebraic varieties~\cite{ARValgebraic}. This point of view offers a powerful interplay of ideas. For example, categorical embedding theorems (such as the embedding theorem for regular categories~\cite{Bar86:articolo}) are closely related to logical completeness theorems, reconstruction theorems 
(as in the theory of ultracategories~\cite{Mak87:articolo,Lur18:articolo}) allow us to recover a theory from its category of models, and adjoint functor theorems (as available for locally presentable categories~\cite{AR94:libro}) enable us to establish the existence of free models. In this context, accessible categories occupy a prominent role~\cite{MP89:libro} since
they are exactly the categories of set-based structures axiomatisable in (infinitary) first-order logic. For example, the categories of fields and of Hilbert spaces are accessible.

Motivation from several directions, including the research programme of categorification in algebra, the theory of stacks in algebraic geometry, and the semantics of programming language in theoretical computer science, leads naturally to the study of category-based, rather than set-based, structures, such as categories with additional structure or properties (\eg monoidal categories or categories with finite limits). 
Additional complications and subtleties arise immediately. First, these structures assemble themselves most naturally into 2-categories, in which one has structure-preserving functors as morphisms, and appropriate natural transformations as 2-cells. Secondly, the morphisms of primary interest are those that preserve the structure only up to coherent isomorphism, not strictly (as this occurs  rarely in practice). The theory of 2-categories developed over the last fifty years, most notably by the Australian school, provides an extensive 
analysis of categories with structure (see~\cite{BlackwellR:twodmt,KellyGM:mongcl,KellyGM:monccc} for example). Part of the challenge of the subject is to work as strictly as possible, so as to exploit $\Cat$-enriched
category theory and avoid the complex coherence considerations typical of the fully weak, bicategorical, setting, while retaining sufficient generality, which is often achieved via the proof of suitable strictification results~\cite{lack20102}. In this context, the class of 2-categorical limits know as \emph{flexible limits}~\cite{BKPS89:articolo} is of particular
importance, for example because they can be understood as the homotopically well-behaved ones~\cite{LackS:homtam}. In recent years, there has also been significant interest and progress in developing a counterpart of the theory of accessible categories in the 2-dimensional setting~\cite{Bou2021:articolo,DiLibertiI:biabp2c,LR12:articolo,BLV:articolo}.

Yet, some important questions remain open, for example regarding whether some 2-categories of interest in categorical logic and categorical algebra are accessible 2-categories with flexible limits or concerning the existence of free categories with various kinds of structure. The motivation for this paper was to solve some of these open questions. For this, we  develop some aspects of categorical logic in the 2-categorical setting, building on earlier work of the second-named author and Rosick\'y on an enriched version of categorical logic~\cite{RT25ERegular,RT23EUA}. 

The problems that we tackle can be readily understood in an example. Let us consider the 2-category~$\Lex$ of categories with finite limits, functors that preserve finite limits, and natural transformations. 
Here, by a category with finite limits we mean a category with the property that every finite diagram
has a limit. Accordingly, by a finite limit-preserving functor we mean a functor that sends a finite
limit diagram to a finite limit diagram. This is in contrast with categories with \emph{chosen} finite limits
and functors that preserve (up to canonical isomorphism) the chosen finite limits, which give rise to a 2-category, written
$\Lex_c$ here (where the subscript abbreviates `chosen'). The 2-category $\Lex_c$ can be understood
with 2-dimensional monad theory~\cite{BlackwellR:twodmt}. Indeed, there is a 2-monad $T \co
\Cat \to \Cat$ such that the associated 2-category~$T\text{-}\tx{Alg}$ of strict algebras, 
pseudomorphisms and algebra 2-cells is $\Lex_c$. The fact that limits exist up to unique isomorphism
is then characterised by saying that the 2-monad $T$ is a lax idempotent~\cite{KockA:monwsa}, or property-like~\cite{KellyG:prolt}. With this results in place, one can establish various 2-categorical completeness
and cocompleteness properties of~$\Lex_c$, which can be transferred to  $\Lex$ since
the 2-functor $U \co \Lex_c \to \Lex$ forgetting the choice of limits is a 2-equivalence.

Pushing further this line of research, and building on earlier work by Makkai~\cite{MakkaiM:gensfc-I,MakkaiM:gensfc-II}, Bourke has recently shown that~$\Lex$  and related 2-categories, such as the 2-category $\Reg$ of regular categories and the
2-category $\GFib$ of Grothendieck fibrations,  are accessible 2-categories with filtered colimits
and flexible limits,  finite flexible limits commuting with filtered colimits, and accesible retract equivalences~\cite{Bou2021:articolo}. Importantly, some of these
examples go beyond the theory covered by 2-dimensional monad theory, since the categories of interest are not necessarily 2-monadic (or pseudo-monadic).
Yet, the approach of~\cite{Bou2021:articolo} - which is somehow inspired by the theory of sketches - 
does not seem to be easily applicable to other 2-categories of interest, such as some of importance in categorical algebra,
such as the 2-categories of Mal'cev, protomodular, and semiabelian categories, 
 and categorical logic, such as the 2-categories of comprehension categories and of clans\footnote{Pseudo-monadicity of 2-categories 
 closely related to that of clans~\cite{UemuraT:genfst}, is subject of ongoing work of Bourke and Jel\'inek.}.
  In a different vein, taking a fully bicategorical approach, Di Liberti and Osmond have shown that $\Lex$  is locally finitely bipresentable in a sense defined in~\cite{DiLibertiI:biabp2c}. Apart from the complex coherence conditions involved
in the fully bicategorical point, this approach leads  to bicategorical accessibility and completeness.

\subsection{Main results} Here, we develop and apply an alternative approach. Our starting point is the observation
that  the definition
of categories with finite limits involves axioms asserting the existence of some structure that
 is unique, up to unique isomorphism, when it exists. This form of quantification is somehow between unique existence (which is part of so-called finite limit theories) and ordinary existence (which is part of so-called regular theories) and hence 
has no counterpart in the 1-categorical setting.

We therefore introduce what we call \emph{isoregular theories}, which are logical theories in which one can express this form of essentially unique existential quantification. We then introduce the notions of a model of an isoregular theory (now taking
values in $\Cat$ rather than $\Set)$ and prove 2-categorical counterparts of key results of categorical logic. 
First, we develop a functorial semantics for isoregular theories. This is done by isolating the fundamental properties of the syntactic 2-category  $\Syn(\TT)$ of an isoregular theory $\TT$, which we do introducing the notion of an isoregular 
2-category (\cref{thm:isoregular}). Then, we  show (\cref{theorem:models=functors}) that models of $\TT$ are the same thing as  2-functors $M \co \Syn(\TT) \to \Cat$ preserving the isoregular structure. With this in place, we prove that the 2-category
of models of an  isoregular theory is accessible with flexible limits and that the forgetful 2-functor to
$\Cat$ has a left biadjoint (\cref{theo.acc+flex}), therefore ensuring the existence of free models.

We then obtain several applications in a uniform manner. In particular,  we obtain new proofs that $\Lex$, $\Reg$, $\GFib$ are accessible with flexible limits, subsuming the core of Bourke's results in~\cite{Bou2021:articolo}.
Building on this, we show that the 2-categories
$\Mal$, $\Ptm$, $\SmAb$ of Malt'sev, protomodular and semiabelian categories
(\cref{thm:mal,thm:ptm,thm:smab}), as well as the 2-categories  $\Cmp$ and $\Clan$
of  comprehension categories and clans (\cref{thm:comp,thm:clan}), 
are accessible with flexible limits, which does not seem to be known.
Additionally, a variety of forgetful functors are shown to admit left biadjoints, therefore establishing
a number of new free constructions. In particular, we obtain a new proof of the existence of
Joyal's finite free bicompletions (\cref{thm:lex-rex-rlex}). 

These results are a mere sample of the range of possible applications: as discussed in~\cref{sec:applications},  the methods introduced
here apply to several other 2-categories of interest, including those of pretoposes, multicategories, left
adjoint functors, \etc

\subsection{Technical aspects}  We conclude this introduction with some comments
on the more technical aspects of the paper. 

First, when it comes to the formulation of the syntax of isoregular theories we combine ideas
of enriched categorical logic~\cite{RT25ERegular,RT23EUA} and of type theory~\cite{nordstrom-petersson-smith:ml}.
Isoregular theories have rules for forming terms, rules for forming propositions, and rules for logical entailment, corresponding to three forms of judgements:
\[
 s \co X \mathrlap{,} \qquad A \co \prp \mathrlap{,} \qquad A_1, \ldots, A_n \co A \mathrlap{.}
 \]
These express that $s$ is a well-formed term, that $A$ is a well-formed proposition, that that $A_1, \ldots, A_n$ entail $A$, respectively, as in logic-enriched
type theories~\cite{AczelP:gentti}. 
In this setting, we can write a deduction rule asserting that, in order to be able to form existential quantification of the form~$(\exists y \co Y) B(x,y)$ only when, for $x \co X$, there exists at most a unique, up to unique isomorphism, $y \co Y$ such that~$B(x,y)$ holds. Thus, the formation rules for propositions depend on the deduction rules for logical entailment. This is in contrast with first-order logic, where one first defines the
set of well-formed formulas and then, separately, defines the set of theorems of a theory by giving the deduction rules for logical entailment. It also differs from the approach in~\cite[D.1.3, Definition~1.3.4]{Joh02:libro}, as we avoid introducing 2-dimensional regular theories and carving out isoregular theories out of them. Building on~\cite{RT25ERegular,RT23EUA}, it will also be essential to have at our disposal  propositions of the form $ A^{[1]}$, where $A$ is a proposition, which enable us to express properties of morphisms, rather than just objects, in the intended models. 
We do not make use of dependent sorts, so as to be able to build directly on~\cite{RT25ERegular,RT23EUA} and remain close to the presentation of~\cite{Joh02:libro}, and leave this as a possible direction for future work.

When it comes to characterising the 2-categorical structure of the syntactic categories of isoregular theories, our notion of an isoregular 2-category (\cref{thm:isoregular})
appears rather natural. Just as a regular 1-category is a 1-category with finite limits in which every kernel pair has a coequaliser and regular epimorphisms are stable under pullback, an isoregular 2-category is a 2-category with finite 2-limits in which every \emph{fully faithful} kernel pair has a coequaliser and \emph{fully faithful} regular epimorphisms are stable under pullback. In particular, isoregular 2-categories are {\sc ff}-regular 2-categories in the sense of~\cite{BG14:articolo} (see \cref{thm:ff-regular} for details).

Just as in a regular category every morphism can be factored as a regular epimorphism followed by a monomorphism, in an isoregular 2-category
every morphism with a fully faithful kernel pair can be factored as a fully faithful regular epimorphism followed by a monomorphism, which gives us exactly what is needed in order to characterise existence up to unique isomorphism. Indeed, morphisms with a fully faithful kernel pair 
admit several equivalent characterisations, including that of being the morphisms that are faithful and whose fibers are subcontractible (\cref{thm:ff-kernel-char}), an aspect that brings out the connection with the notion of a 
 \emph{homotopy proposition} of Univalent Foundations and Homotopy Type Theory, as we discuss in \cref{thm:hott}. The factorisation of such morphisms then amounts to being able to existentially quantify over these fibers. Indeed, in the syntactic 2-category of an isoregular theory, the composite
 \[
 \begin{tikzcd} 
 \{  x \co X, y \co Y \ |  \ B(x,y) \} \ar[r, >->] & X \times Y \ar[r] & X 
 \end{tikzcd} 
 \]
 has a fully faithful kernel if and only if it is provable that it is faithful and, for $x \co X$, there exists at most a unique, up to unique isomorphism, $y \co Y$ such that $B(x,y)$ holds. Existential quantification then corresponds to the factorisation
  \[
 \begin{tikzcd}[column sep = large] 
&  \{ x \co X \ | \ (\exists y \co Y) B(x,y) \} \ar[dr, >->]   &   \\
 \{ x \co X, y \co Y \ | \ B(x,y) \} \ar[ur, ->>] \ar[r, >->] & X \times Y \ar[r] & X  \mathrlap{.}
  \end{tikzcd}
  \]

In order to prove that models of isoregular theories are accessible 2-categories with flexible
limits, we first provide a functorial semantics (\cref{theorem:models=functors}), describing models as what we call isoregular functors, \ie 2-functors preserving the structure of an isoregular 2-category:
\[
 \Mod(\TT)\ \cong  \IsoReg[\Syn(\TT),\Cat]  \mathrlap{,} 
\]
and then, separately, establish that 2-categories of isoregular 2-functors are  accessible with flexible
limits (\cref{theo.acc+flex}), a result that provides additional evidence for the usefulness of the notion of
an isoregular 2-category. Clearly, working in the 2-categorical (rather than bicategorical) setting offers
significant simplifications throughout the development of the theory and does not prevent the intended applications.

\subsection{Outline of the paper}

\cref{sec:isoregular-categories} introduces isoregular 2-categories and establishes some of their properties, including various equivalent characterisations of morphisms with fully faithful kernel. \cref{sec:isoregular-theories} introduces syntax and semantics of isoregular theories, including the soundness theorem.
\cref{sec:functorial-semantics} establishes the functorial semantics for isoregular theories, by introducing
syntactic 2-categories. Using the functorial semantics, \cref{sec:accessibility} establishes the accessibility of models of isoregular theories. We conclude the paper in \cref{sec:applications} with applications.

\subsubsection*{Acknowledgements.} Nicola Gambino is grateful to the
Department of Mathematics and the School of Natural Sciences of The University of Manchester
for granting sabbatical leave, during which this paper was written. Giacomo Tendas acknowledges with gratitude the support of the EPSRC postdoctoral fellowship via grant EP/X027139/1. We thank John Bourke, Richard Garner, V\'it Jel\'inek, Marino Gran, and Ji\v{r}\'{i} Rosick\'y for helpful conversations.

%
%
%\red Nicola: We need to be careful with the notation $\mathcal{J}_{\mathsf{ffk}}$, as later we use it for a map in the syntactic category. Here, the notation should match, expressing that
%the composite map 
%\[
%\{ x \co X, y \co Y \ | \ B(x,y) \} \to \{ x \co X, y \co Y \ | \ \top \} \to \{ x \co X \ | \ \top \}
%\] 
%is an \ffk-morphism. \black

\section{Isoregular 2-categories} 
\label{sec:isoregular-categories}

\subsection{Preliminaries} We assume that the readers are familiar with the theory of 
2-categories and confine ourselves to fix some notation and terminology. See~\cite{lack20102} for the basics, \cite{KellyGM:eleo2c}
for 2-categorical limits in general and~\cite{BKPS89:articolo} for flexible 2-limits.
Our development will focus on \emph{cartesian 2-categories}, \ie 2-categories with all
finite 2-categorical limits, including a terminal object, products, pullbacks, equalisers.

We write $\Cat$ for the 2-category of small categories, functors and natural transformations.
For a 2-category $\K$, we write $\K(A,B)$ for the hom-category of morphisms from $A$ to
$B$ and 2-cells between them. The composite of $f \co A \to B$ and $g \co B \to C$ will be
denoted $g \circ f$ or simply~$gf$. Similar conventions are adopted for composition of
morphisms and 2-cells. We write $\K_0$ for the underlying 1-category of $\K$.

Let $\K$ be a 2-category. The \emph{power} of an object $A \in K$ by a small category $\bb c$ is
an object $A^{\bb c} \in \K$ together with a functor, called \emph{evaluation}, $\mathrm{ev}_{\bb c} \co {\bb c} \to \K(A^{\bb c}, A)$
 which is universal, in the sense that for every $X  \in \K$, the induced functor 
\[
\K(X, A^{\bb c}) \to \K(X,A)^{\bb c}
\]
is an isomorphism of categories. When $\K$ has powers by a small category $\bb c$,
these determine a 2-functor $(-)^{\bb c} \co \K \to \K$. In the following, we shall be primarily interested in
powers by finitely presentable categories. For a natural number $n$, we write $[n]$ for the nerve
of the poset $0 < 1 \ldots < n$. In particular, $[1]$ is the category with two objects ($0$ and $1$)
and a single non-identity map (from~$0$ to~$1$). In this case, the evaluation functor determines
two maps $\pi_0 \co A^{[1]} \to A$, $\pi_1 \co A^{[1]} \to A$. When~$\K$ has binary products, we write
$\pi_A \co A^{[1]} \to A \times A$ for their pairing, \ie $\pi_A \defeq (\pi_0, \pi_1)$.

Since we will not generally assume the ambient 2-category $\K$ to have all powers and copowers, the 2-dimensional universal property of
a conical 2-limit in $\K$ will have to be checked explicitly, rather than inferred from its 1-dimensional
one in $\K_0$. However, we have the following  observation.

\begin{lemma} \label{thm:enhance-conical-with-powers}
Let $\K$ be a 2-category with powers by $[1]$. If a conical limit exists in the
underlying category $\K_0$ of $\K$ and it is preserved by the functor $(-)^{[1]} \co \K_0 \to
\K_0$, then this limit is a 2-limit in~$\K$.
\end{lemma}

For an object $X \in \K$, we write $\K(X, -) \co \K \to \Cat$ for the induced representable 2-functor. These 2-functors allow us to extend to a general 2-category $\K$ many concepts that can be defined in $\Cat$. For example,
a  morphism $f\colon A\to B$ in $\K$ is called \emph{full} if, for every $X\in\K$, the functor 
\begin{equation}
\label{equ:representable-applied-to-f} 
\K(X,f)\colon \K(X,A)\to\K(X,B)
\end{equation}
is full. When the fullness condition is restricted to
2-cells in the codomain that are isomorphisms or equalities, we say that $f$ is \emph{full on isomorphisms} or
\emph{full on identities}, respectively. Explicitly, for $f$ to be full on identities means that, 
given $a \co X \to A$ and $a' \co X \to A$, if $fa = fa'$ there exists
a 2-cell $\alpha \co a \Rightarrow a'$ such that $f \alpha = 1_{fa}$. 
The notions of a faithful and a fully faithful morphism are defined analogously. 
Note that, if $f$ is faithful and full on identities, then the 2-cell $\alpha$ above is unique and
an isomorphism. 

If $\K$ has powers by $[1]$, $f$  is fully faithful if and only if the following square is a pullback:
\[
\begin{tikzcd}
A^{[1]} \ar[r, "f^{[1]}"] \ar[d, "\pi_A"'] & B^{[1]} \ar[d, "\pi_B"] \\
A \times A \ar[r, "f \times f"'] & B \times B \mathrlap{.}
\end{tikzcd}
\]
In a 2-category $\K$ a morphism is said to be a \emph{monomorphism} if it is representably so.
Hence, every monomorphism is faithful. A morphism is
a \emph{regular epimorphism} if it is a coequalizer (in the 2-categorical sense) of a parallel pair of maps.
If $\K$ has finite limits, a morphism is a monomorphism (a regular epimorphism) if and only if it is 
a monomorphism (a regular epimorphism, respectively) in the underlying category $\K_0$.

%\begin{defi} Let $f\colon A\to B$ be a morphism  in $\K$. We say that 
% $f$ is a {\em fully faithful regular epimorphism} if it is a regular epimorphism
% and fully faithful.
%\end{defi}

\subsection{Isoregular 2-categories}

Let $\K$ be a  2-category with finite 2-limits, to remain fixed for the rest of this section. Let $f \co A \to B$ be a morphism in $\K$.
Recall that the kernel pair of $f$ is the pullback
\[
\begin{tikzcd}
\Delta_f \ar[d, "\pi_1"'] \ar[r, "\pi_2"] & A \ar[d, "f"] \\
A \ar[r, "f"'] & B \mathrlap{.}
\end{tikzcd}
\]

\begin{prop} \label{thm:ff-kernel} Let $f \co A \to B$ be a morphism in $\K$. The following conditions are equivalent:
\begin{enumerate}
\item the morphism $\pi_1 \co \Delta_f \to A$ is fully faithful,
\item the morphism $\pi_2 \co \Delta_f \to A$ is fully faithful.
\end{enumerate}
\end{prop}

When the equivalent conditions of \cref{thm:ff-kernel} hold, we say that $f$ 
has a \emph{fully faithful kernel pair}. 
In the 2-category $\Cat$ of small categories, if a  functor $f \co A \to B$ has a fully faithful kernel pair, then its coequaliser is again fully faithful and therefore is a fully faithful regular epimorphism. The essential properties of this situation are axiomatised in the notion of an isoregular 2-category,
which we introduce next.

\begin{defi} \label{thm:isoregular}
	We say that a 2-category $\K$ is {\em isoregular} if
	\begin{enumerate}
		\item it has all finite 2-limits;
		\item every fully faithful kernel pair has a fully faithful coequaliser;
		\item fully faithful regular epimorphisms are stable under pullback.
	\end{enumerate}
\end{defi}

\begin{ese} The following  2-categories are isoregular. 
\begin{itemize}
\item The 2-category $\Cat$ of small categories. In $\Cat$, a functor is a fully faithful regular epimorphism if and only if it is surjective on objects and fully faithful, which is the case if and only if it is a retract equivalence. 
\item  2-categories of prestacks, \ie 2-categories of the form $[\C \op,\Cat]$ for any small 2-category $\C$;
\item 2-categories of stacks, \ie left exact localisations of 2-categories of prestacks, \cf~\cite{StreetR:twodst} and~\cite[Section~6.1]{BG14:articolo}.
\end{itemize}
\end{ese}

%Note that if $\K$ is isoregular with flexible limits then is complete (since it has equalisers, products, and powers); thus, if it is moreover accessible, it is locally presentable. 

\begin{rmk} \label{thm:ff-regular} 
Every isoregular 2-category is an {\sc ff}-regular 2-category in the sense of \cite[Section~7]{BG14:articolo}.
Indeed, a 2-category {\sc ff}-regular if it satisfies conditions
(i) and (iii) of \cref{thm:isoregular}, and, in place of (ii), it is required that every kernel pair of a fully faithful map admits a fully faithful coequaliser. Since the kernel pair of a fully faithful map is fully faithful, isoregular 2-categories are {\sc ff}-regular. As we shall see in \cref{thm:gfib-ffk}, the possibility of taking coequalisers of morphisms with fully faithful kernel pairs,
rather than just of fully faithful morphisms, will be important for some applications.
\end{rmk}

Morphisms with a fully faithful kernel pair, or \ffk-morphisms for short, 
 are fundamental for the definition of  an isoregular 2-category and therefore it is useful to have some equivalent characterisations of them. We provide these in~\cref{fact-morph} below, but for this we need a definition and a preliminary lemma.
 In analogy with the definition of a subterminal object in a 1-category, we define a object $C$ of a 2-category $\K$
 to be \emph{subcontractible} if for every $X \in \K$ and $f, g \co X \to C$, there exists a unique (and hence necessarily invertible) $\alpha \co f \Rightarrow g$. Equivalently, for every $X \in \K$, the category $\K(X,C)$ is either empty or contractible.
 In $\Cat$, a subcontractible object is a category that is either empty or contractible. For example, consider a category $D$
 and define $C$ to be the full subcategory of $D$ spanned by the terminal objects of $D$. Then $C$ is subcontractible.
   
 \begin{lemma} \label{thm:ffk-subcontr}
 Assume that $\K$ has a terminal object.   An object $C \in \K$ is subcontractible if and only if the unique map $f \co C \to 1$ is an \ffk-morphism.
 \end{lemma}
 
 \begin{proof} By definition, $f \co C \to 1$ is an \ffk-morphism if and only if, for every $X \in \K$, the projection functor $\pi_1 \co
 \K(X,C) \times \K(X,C) \to \K(X,C)$ is fully faithful. But this means exactly that $C$ is subcontractible.
 \end{proof}

\begin{prop} \label{thm:ff-kernel-char} \label{fact-morph} Let $f \co A \to B$ be a morphism in $\K$. The following conditions are equivalent:
\begin{enumerate}
\item the morphism $f$ is an \ffk-morphism, \ie it has a fully faithful kernel pair,
\item  for every $X\in\K$, the functor
	\[
	\K(X,f) \co \K(X,A) \to \K(X, B)
	\] 
	has a fully faithful kernel pair,
\item the diagram
\[
\begin{tikzcd}
A \ar[r, "1_A"] \ar[d, "1_A"'] & A \ar[d, "f"] \\
A \ar[r, "f"'] & B
\end{tikzcd}
\]
is a semistrict pullback, \ie for every $a' \co X \to A$ and $a'' \co X \to A$ such that
$f a' = fa''$, there exists $a \co X \to A$, unique up to unique isomorphism, and isomorphisms 
$\alpha' \co a \Rightarrow a'$, $\alpha'' \co a \Rightarrow a''$ such that $f \alpha' = 1_{fa}$
and $f \alpha'' = 1_{fa}$.
\item the morphism $f$ is faithful and full on identities,
\item the unique morphism $r_f \co A \to \Delta_f$ such that $\pi_1 \circ r = 1_A$ and $\pi_2 \circ r = 1_A$ is an equivalence.
\item the morphism $f \co A \to B$ is a subcontractible object of the slice 2-category $\K_{/B}$.
\end{enumerate}
\end{prop}

\begin{proof} The equivalence between~(i) and~(ii) is immediate. Since
the characterisations in parts (iii)--(v), can all be checked representably, the claim follows once we prove their equivalence with~(i) in~$\Cat$. This is a straightforward calculation that we omit. The equivalence with (vi) follows by~\cref{thm:ffk-subcontr},
recalling that $1_B \co B \to B$ is a terminal object of $\K_{/B}$.
\end{proof}

\begin{rmk}[\ffk-morphisms in $\Cat$]  \label{thm:ffk-in-cat}
In the 2-category $\Cat$, a functor $f \co A\to B$ is an \ffk-morphism
if and only if it is faithful and full on identities, \ie for every $a,a'\in A$ such that $fa=fa'$ there exists a map $u \colon a\to a'$ in $A$ such that $f(u) =1_{fa}$. In these circumstances, such a map $u$ is unique and an isomorphism.
Equivalently, by part~(vi) of \cref{thm:ff-kernel-char}, $f$ is faithful and, for every $b \in B$, the (strict) fiber category
$f^{-1}(b)$, given by the subcategory of $A$ spanned by the objects $a \in A$ such that $f(a) = b$ and those morphisms $h\in A$ such that $f(h)=1_b$, 
is either empty or contractible. 
\end{rmk}

\begin{es}
Let $I$ and $B$ be categories and define $A$ as the full subcategory 
of $B^I \times B \times (B^{I})^{[1]} $ spanned by triples $(b, a, \alpha\co \Delta a\to b)$, where $b \in B^I$ and
$(a, \alpha)$ is a limit cone for the diagram $b$. 
Then, the projection functor $f \co A \to B^I$ is an \ffk-morphism, since it is faithful and its fiber over $b \in B^I$
is either empty (if $b$ does not have a limit in $B$) or contractible (if a limit exists, as the limit cones over $b$ are unique 
up to unique isomorphism).
\end{es}

\begin{cor} \leavevmode
\begin{enumerate}
\item Every monomorphism is an \ffk-morphism.
\item Every fully faithful morphism is an \ffk-morphism.
\end{enumerate}
\end{cor}

\begin{proof} Both part~(i) and part~(ii) can be checked representably and they 
hold in $\Cat$ since \ffk-morphisms there are faithful and full on identities (\cf \cref{thm:ffk-in-cat}).
\end{proof} 

\ffk-morphisms are not closed under composition. For example, let us work in $\Cat$ and consider the composite functor 
\[
\begin{tikzcd} 
\{0\} + \{1 \} \ar[r, "F"] & J \ar[r] & \{ \ast \} \mathrlap{,} 
\end{tikzcd}
\]
where $J$ is the category with two objects and an isomorphism between them and the functor $F$ is the 
inclusion of the endpoints. The composite is not an \ffk-morphism, even if each of the composites.
This fact could be understood recalling that, by part~(vi) of \cref{thm:ff-kernel-char}, being an \ffk-morphism
involves a property of the fibers, which is not necessarily preserved by composition. 
However, the following closure properties with respect to fully faithful morphisms and monomorphisms hold.

\begin{lemma} Let $f \co A \to B$ be an \ffk-morphism.
\begin{enumerate}
\item For every monomorphism $b \co B \rightarrowtail B'$, the composite $b \circ f \co A \to B'$ is an \ffk-morphism.
\item For every fully faithful morphism $a \co A' \to A$, the composite $f \circ a \co A' \to B$ is an
\ffk-morphism.
\item For every fully faithful morphism $a \co A' \to A$ and monomorphism $b \co B \rightarrowtail B'$,
the composite $b \circ f \circ a \co A' \to B'$ is an \ffk-morphism.
\end{enumerate}
\end{lemma} 

\begin{proof} Parts~(i) and~(ii) follow from, say, part~(iv) of \cref{thm:ff-kernel-char}. Part~(iii) is an immediate
consequence of parts~(i) and~(ii).
\end{proof} 

\begin{cor} \label{thm:comp-easy}
The composite of a fully faithful morphism followed by a monomorphism is an \ffk-morphism.  \qed
\end{cor} 

Next, we give an example of
an \ffk-morphism that is neither fully faithful nor a monomorphism.

\begin{es} \label{thm:gfib-ffk}
Let $p \co E \to B$ be a Grothendieck fibration. Consider the sub-category $C$ of
$E \times B^{[1]} \times E^{[1]}$ with objects $(y, u, v)$ such that $\cod(u) = p(y)$,
and $v$ is a Cartesian lift of $u$, and morphisms satisfying an analogous condition. We then have a functor
\[
\begin{tikzcd} 
C \ar[r, >->] & E \times B^{[1]} \times E^{[1]} \ar[r, "\pi"] & E \times B^{[1]} 
\end{tikzcd}
\]
where $\pi$ is the projection on the first two coordinates. This functor is  not full
since for a morphism $(t, (v, w)) \co (y, u) \to (y', u')$ in $E \times B^{[1]}$, it is
not necessarily the case that $w = p(t)$. Instead, it is an \ffk-morphism, as it is
evidently faithful and its fiber over $(y,u) \in E \times B^{[1]}$  is contractible when $\cod(u) = y$
(in which case it consists of all the cartesian lifts of $u$) and empty when $\cod(u) \neq y$.
\end{es}

In a regular 1-category, every morphism can be factored as a regular epimorphism followed by a
monomorphism. As \cref{s.eq+mono} shows, in an isoregular 2-category, every \ffk-morphism can
be factored as a fully faithful regular epimorphism followed by a monomorphism. The factorisation
can be understood as turning a \ffk-morphism, whose fibers are subcontractible, into a monomorphism,
whose fibers are subterminal (\cf \cref{thm:hott} for additional comments on this point). The proof  is essentially the same as the one of the corresponding statement for regular 1-categories.

\begin{prop}\label{s.eq+mono}
Let $\K$ be an isoregular 2-category. Every \ffk-morphism  factors, in an essentially unique way, as a fully faithful regular epimorphism followed by a monomorphism. 
\end{prop}

\begin{proof} Let $f\colon A\to B$ be an \ffk-morphism and consider the following diagram
	\begin{center}
		
		\begin{tikzpicture}[baseline=(current  bounding  box.south), scale=2, hookarrow/.style={{Hooks[right]}->}]
			
			\node (k) at (0.2,0) {$\Delta_f$};
			\node (a) at (1,0) {$A$};
			\node (b) at (2,0) {$B$};
			
			\node (l) at (0.6,-0.6) {$\Delta_m$};
			\node (c) at (1.5,-0.6) {$X$};
			
			\path[font=\scriptsize]

			([yshift=1.5pt]k.east) edge [right hook->] node [above] {$h$} ([yshift=1.5pt]a.west)
			([yshift=-1.5pt]k.east) edge [right hook->] node [below] {$k$} ([yshift=-1.5pt]a.west)
			
			(a) edge [->] node [above] {$f$} (b)
			
			([yshift=1.5pt]l.east) edge [->] node [above] {$r$} ([yshift=1.5pt]c.west)
			([yshift=-1.5pt]l.east) edge [->] node [below] {$s$} ([yshift=-1.5pt]c.west)
			
			(a) edge [right hook->>] node [left] {$q$} (c)
			(c) edge [->] node [right] {$m$} (b)
			(k) edge [->] node [left] {$\ell$} (l);
		\end{tikzpicture}
	\end{center}
	where $(h,k)$ is the kernel pair of $f$ (and they are both fully faithful since $f$ is an \ffk-morphism, $q$~is the coequaliser of $(h,k)$ and so is a fully faithful regular epimorphism by hypothesis that $\K$ is an isoregular 2-category, $m$ is induced by the universal property of the coequaliser, $(r,s)$ is the kernel pair of $m$, and $\ell$ is induced by the universal property of the kernel pair. 
	
	To show that the desired factorization exists it suffices to prove that $m$ is a monomorphism, which is equivalent to requiring that $r=s$. For this, note that $\Delta_f$ and $\Delta_m$ fit in the pullbacks below:
	\begin{center}
		\begin{tikzpicture}[baseline=(current  bounding  box.south), scale=2]
			
			\node (k) at (0,0.8) {$\Delta_f$};
			\node (r) at (0.9,0.8) {$\bullet$};
			\node (s) at (0,0) {$\bullet$};
			\node (l) at (0.9,0) {$\Delta_m$};
			\node (e0) at (0.15,0.65) {$\lrcorner$};
			
			\node (a1) at (1.8,0.8) {$A$};
			\node (c1) at (1.8,0) {$X$};
			\node (e0) at (1.05,0.65) {$\lrcorner$};
			
			\node (a2) at (0,-0.8) {$A$};
			\node (c2) at (0.9,-0.8) {$X$};
			\node (e0) at (0.15,-0.15) {$\lrcorner$};
			
			\node (b) at (1.8,-0.8) {$B\mathrlap{.}$} ;
			\node (e0) at (1.05,-0.15) {$\lrcorner$};
			
			\path[font=\scriptsize]
			
			(k) edge [right hook->>] node [above] {} (r)
			(k) edge [right hook->>] node [left] {} (s)
			(s) edge [right hook->>] node [right] {} (l)
			(r) edge [right hook->>] node [below] {} (l)
			
			(r) edge [->] node [above] {} (a1)
			(a1) edge [right hook->>] node [right] {$q$} (c1)
			(l) edge [->] node [above] {$r$} (c1)
			
			(s) edge [->] node [above] {} (a2)
			(a2) edge [right hook->>] node [below] {$q$} (c2)
			(l) edge [->] node [left] {$s$} (c2)
			
			(c2) edge [->] node [below] {$m$} (b)
			(c1) edge [->] node [right] {$m$} (b);
		\end{tikzpicture}	
	\end{center}
	By stability under pullback of fully faithful regular epimorphisms, all the arrows in the top-left square are fully faithful regular epimorphisms. Hence, the diagonal of the square, which coincides with $\ell \colon \Delta_f \to \Delta_m$, is an epimorphism. Since $r \ell=qh=qk=s \ell$, we obtain $r=s$ as desired.
	
	The factorization is unique up to isomorphism since a factorisation of a morphism as a
	strong epimorphism followed by a monomorphism is so, when it exist.
\end{proof}

\begin{prop} \label{thm:wse-ffse}
	Let $\K$ be an isoregular 2-category. Let $f \colon A\to B$ be a morphism in $\K$.	
	The following conditions are equivalent:
	\begin{enumerate}
		\item $f$ is a fully faithful regular epimorphism,
		\item $f$ is a fully faithful strong epimorphism.
	\end{enumerate}
\end{prop}
\begin{proof}
If $f$ is fully faithful regular epimorphism, then it is in particular a strong epimorphism.
 Conversely, if $f \colon A\to B$ is a fully faithful strong epimorphism then it is an \ffk-morphism by Proposition~\ref{fact-morph}. Then, by Proposition~\ref{s.eq+mono}, $f$ factors as a fully faithful regular epimorphism followed by a monomorphism. Since $f$ is a strong epimorphism, the monomorphism must be an isomorphism; thus $f$ is a fully faithful regular epimorphism.
\end{proof}

\begin{cor} Let $\K$ be an isoregular 2-category. Fully faithful regular epimorphism are stable in $\K$ under composition, finite products, and finite powers.
\end{cor}

\begin{proof} 
	The fact that fully faithful regular epimorphisms are closed under composition  follows from \cref{thm:wse-ffse} since strong epimorphisms and fully faithful epimorphisms are. For stability under finite products, consider two fully faithful regular epimorphisms $f_i\colon A_i\to B_i$, for $i=1,2$. Then 
	\[
	f_1\times f_2=(B_1\times f_2)\circ (f_1\times A_2)
	\]
and both components are fully faithful regular epimorphisms since they are obtained pulling back $f_1$ and $f_2$ along the product projections. Thus $f_1\times f_2$ is a fully faithful regular epimorphism by stability under composition.
	
	Finally, we need to show stability under finite powers. Let $\bb b$ be a finitely presentable category and let $S$ be its (finite) set of objects seen as a discrete category; thus we have a identity-on-objects inclusion $\iota^{\bb b}\colon S\to {\bb b}$. For any fully faithful morphism $f\colon A\to B$ the following square
	\begin{center}
		\begin{tikzpicture}[baseline=(current  bounding  box.south), scale=2]
			
			\node (k) at (0,0.8) {$A^{\bb b}$};
			\node (r) at (0.9,0.8) {$B^{\bb b}$};
			\node (s) at (0,0) {$A^S$};
			\node (l) at (0.9,0) {$B^S$};
			\node (e0) at (0.15,0.65) {$\lrcorner$};
			
			\path[font=\scriptsize]
			
			(k) edge [right hook->] node [above] {$f^{\bb b}$} (r)
			(k) edge [->] node [left] {$A^{\iota^{\bb b}}$} (s)
			(s) edge [right hook->] node [below] {$f^S$} (l)
			(r) edge [->] node [right] {$B^{\iota^{\bb b}}$} (l);
		\end{tikzpicture}	
	\end{center}
	is a pullback: this is true in $\Cat$  and pullbacks and fully-faithfulness can be checked representably. Thus, if $f$ is a fully faithful regular epimorphism, then so is $f^S$ (being a finite product of copies of $f$), and hence so is $f^{\bb b}$ (by stability under pullbacks).
\end{proof}

\begin{rmk} \label{thm:hott} We can relate our notion of an \ffk-morphisms with ideas in Homotopy Type Theory~\cite{HoTT-Book}, Univalent Foundations~\cite{VoevodskyV:explfm}, and Higher Topos Theory~\cite{Lur09:libro}. Let $\Gpd$ be the 2-category of small groupoids, functors, and natural transformations. Let $p \co B \to A$
be an isofibration. For $a \in A$, the \emph{homotopy fiber} of $p$ over $a$ is the subgroupoid of $B \times A^{[1]}$ with objects pairs $(b, \alpha)$ where $\alpha \co p(b) \to a$ and 
maps satisfying an evident commutativity condition. When all the homotopy fibers of $p$ are either empty or contractible, one says that $p$ is \emph{$(-2)$-truncated}. While $(-2)$-truncated isofibrations are a suitable semantical counterpart of `homotopy propositions', one may consider also `strict propositions', \ie isofibrations that are actually monomorphisms, as traditionally considered in categorical logic. These two classes of maps are related by a reflection, whose left adjoint sends a $(-2)$-truncated isofibration  $p \co B \to A$ to the monomorphism obtained by the factorisation 
\[
\begin{tikzcd}
B \ar[r, hook, ->>, "q"] & \lvert  B \rvert  \ar[r, >->, "m"] & A 
\end{tikzcd}
\]
of $p$ as a retract equivalence followed by a monomorphism. While the fibers of $p$ are either empty or contractible, the fibers of $m$ are either empty or singletons. The factorisation of \cref{s.eq+mono} achieves a similar reflection in our context.
\end{rmk}

Let $\K$ and $\L$ be isoregular 2-categories.
A 2-functor $F \co \K \to \L$  is said to be \emph{isoregular}
if it preserves finite 2-limits and coequalisers of fully faithful kernel pairs; equivalently, if it preserves finite limits and fully faithful regular epimorphisms.
We write $\IsoReg(\C,\K)$ for the full sub-2-category of the functor 2-category
$[\K, \L]$ spanned by  isoregular 2-functors.

\section{Isoregular theories: syntax and semantics} 
\label{sec:isoregular-theories}

\subsection{Languages}  We follow the approach of~\cite{RT23EUA,RT25ERegular}. Languages are multi-sorted with arities being objects in the full sub-2-category of $\Cat$ spanned by the finitely presentable categories, which we write $\FinCat$.

\begin{defi}
A finitary language $\LL$ is given by:
\begin{itemize}
	\item a set of \emph{basic sorts} $\S$, written $S, T, \ldots$;
	\item a set of \emph{function symbols} with sorts, written $f\colon S_1^{\bb c_1},\ldots, S_n^{\bb c_n} \to S^{\bb c}$, where $S_1, \ldots, S_n, S$ are basic sorts and $\bb c_1, \ldots, \bb c_n, \bb c$ are arities;
	\item a set of \emph{relation symbols} $R\rightarrowtail S_1^{\bb c_1},\ldots ,S_n^{\bb c_n}$, where $S_1, \ldots, S_n$ are basic sorts and and $\bb c_1, \ldots \bb c_n$ are arities.
\end{itemize}
\end{defi}

Let $\LL$ be a finitary language as above. A \emph{sort} is an expression of the form $S^{\bb c}$ where $S$ is a basic sort and $\bb c$ is an arity. When ${\bb c } = [0]$, we write $S$ instead of $S^{[0]}$.
If $X = S^{\bb c}$ is a sort and $\bb d$ is an arity, we define
\[
X^{\bb d} \defeq S^{\bb d \times \bb c} \mathrlap{.}
\]

A \emph{context} is a sequence of variable declarations of the form $(\bar{x} \co \bar{X}) = (x_1 \co X_1, \ldots, x_n \co X_n)$. For such a context, we define
\[
(\bar{x} \co \bar{X})^{\bb c} \defeq (u_1 \co X_1^{\bb c}, \ldots, u_n \co X_n^{\bb c}) \mathrlap{.} 
\]
Here, the relabelling of variables has been made to improve readability, rather than for any essential reason. The deductive calculus for isoregular theories has three forms of judgement:
\begin{gather} 
(\bar{x} \co \bar{X}) \quad  s \co X  \mathrlap{,}  \label{equ:term-in-context}  \\
(\bar{x} \co \bar{X}) \quad A \co \mathsf{prop} \mathrlap{,}  \label{equ:prop-in-context} \\
(\bar{x} \co \bar{X}) \quad A_1, \ldots, A_m \vdash A \mathrlap{.} \label{equ:entailment}
\end{gather}
The judgements in~\eqref{equ:term-in-context} and~\eqref{equ:prop-in-context} 
express that, relative to the variable declarations in the context, the expression $s$ is a term of sort~$X$ and the expression $A$ is a proposition, respectively.
The judgement in~\eqref{equ:entailment} expresses that,  relative to the variable declarations in the context and under the assumptions of the propositions $A_1, \ldots, A_n$, the proposition $A$ is true. 

The notation introduced above will be simplified in many cases. When $m = 0$, we write the
judgement in~\eqref{equ:entailment}  as 
\[
(\bar{x} \co \bar{X}) \  A  
\]
and we do not mention the context when it is empty.  The deduction rules of our calculus have the form
\[
\begin{prooftree}
\mathcal{J}_1 \qquad \ldots \qquad \mathcal{J}_n
\justifies
\mathcal{J} \mathrlap{,}
\end{prooftree}
\]
where $\mathcal{J}_1$, \ldots $\mathcal{J}_m$, $\mathcal{J}$ are judgements of one of the forms above. When $n = 0$, we do not include the horizontal line. When stating rules, we do not mention a context 
 that is common to all premisses and conclusions unless this causes confusion.

The deduction rules regarding terms are presented in \cref{tab:terms}, making use of the conventions just introduced, and follow essentially~\cite{RT23EUA}. Here, substitution is  a primitive operation, rather than being defined by induction on terms, as this aligns well with our semantics of terms (\cf \cref{sec:semantics}). As we shall see, a judgement of the form $(\bar{x} \co \bar{X}) \ t \co Y$ will be interpreted in an isoregular 2-category as a morphism $\br{t} \co \br{\bar{X}} \to \br{Y}$.

%
%\begin{nota}
%	We often write $X,Y,Z,\dots$ to indicate things of the form $S^{\bb c}$ where $S$ is a  sort and $\bb c$ an arity. And write 
%	$$  X\defeq (X_1,\cdots,X_n)= (S_1^{\bb c_1},\cdots ,S_n^{\bb c_n})$$
%	Thus, function and relation symbols can be rewritten as $f\colon   X\to Y$, and $R\rightarrowtail   X$.
%	Moreover, if $X=S^{\bb c}$ and $\bb d$ is an arity, $X^{\bb d}\defeqS^{\bb d\times \bb c}$, and $  X^{\bb d}\defeq(X_1^{\bb d},\cdots, X_n^{\bb d})$. 
%\end{nota}

\begin{table}[htb]
\begin{myframedbox}{
\begin{trule}[Variable]\label{rule:variable} For $1 \leq i \leq n$,  
\[
	(x_1 \co X_1, \ldots, x_n \co X_n) \quad \var_i(x_1,\ldots, x_n) \co X_i   \mathrlap{.} 
\]
% If $X = (x_1 \co X_1, \ldots, x_n \co X_n)$ is a context, then $x_i \co X_i$ for $1 \leq n$.
\end{trule}
\begin{trule}[Function application] For each function symbol $f \co Y_1, \ldots, Y_n \to Y$ of $\LL$,
\[
	\begin{prooftree} 
	(\bar{x} \co \bar{X}) \ s_1 \co Y_1 \quad \ldots \quad (\bar{x} \co \bar{X}) \  s_n \co Y_n
	\justifies
	(\bar{x} \co \bar{X}) \  f(s_1, \ldots, s_n)  \co Y \mathrlap{.} 
	\end{prooftree}  \bigskip
	\]
	\end{trule}
\begin{trule}[Restriction]\label{rule:restriction-def} For each functor $f \co \bb d \to \bb c$ between f.-p. categories,
\[
	\begin{prooftree} 
	(\bar{x} \co \bar{X}) \ s \co S^{\bb c}
	\justifies
	(\bar{x} \co \bar{X}) \ s \cdot f \co S^{\bb d} \mathrlap{.} 
	\end{prooftree}
\]
\end{trule}
\begin{trule}[Powers] 
\[
	\begin{prooftree} 
	(\bar{x} \co \bar{X}) \quad s  \co X
	\justifies
	(\bar{x} \co \bar{X})^{[1]} \quad s^{[1]}  \co X^{[1]}
	\end{prooftree} \bigskip
\]
\end{trule}
\begin{trule}[Substitution] For a context $(\bar{x} \co \bar{X})$ 
 and in $(y_1 \co Y_1, \ldots, y_n \co Y_n)$ 
are disjoint, 
\[
	\begin{prooftree} 
	(\bar{x} \co \bar{X}) \quad s_1  \co Y_1 \quad \ldots \qquad
	(\bar{x} \co \bar{X}) \quad s_n  \co Y_n \qquad
	(y_1 \co Y_1, \ldots, y_n \co Y_n) \quad t \co Y
	\justifies
	(\bar{x} \co \bar{X}) \quad t[s_1/y_1, \ldots, s_n/y_n] \co Y
	\end{prooftree} \bigskip
\]
\end{trule}}
\end{myframedbox}
\caption{Formation rules for terms} 
\label{tab:terms}
\end{table}

\begin{rmk}
	In \cref{rule:variable} above, $\var_i(x_1,\ldots, x_n)\co X_i$ denotes the $i$-th variable projection, which is usually denoted simply by $x_i \co X_i$. We use this approach as it makes it simpler to write some deduction rules (\eg \cref{equ:comp}). In practice we will use the more traditional notation $x_i \co X_i$.
	\end{rmk}

%We provide some explanations for these rules and spell our their side
%conditions. Rule~\eqref{equ:variable-rule}, called the variable projection rule, has the side condition that $1 \leq i \leq n$. Rule~\eqref{equ:function}, called the function application rule, has side condition that
%$f$ is a function symbol $f \co X_1, \ldots, X_n \to X$. Rule~\eqref{equ:restriction}, called the restriction rule, has the side condition that $f \co \bb d \to \bb c$ is a functor. In rule~\eqref{equ:power}, called the
%power rule, there are no side conditions, but we make use of the convention for power contexts set above. Finally, for the rule in~\eqref{equ:subst}, called the subsitution rule, there is the side condition that 
%the variables in the context $\bar{X}$ and $\Delta$ are disjoint. 
%

%Given terms $ (  x_i\colon   X^i)\ t_i(  x_i)\colon Y_i $, for $1\leq 1\leq m$, and a term $ (  y\colon Y)\ s(  y)\colon S^{\bb c} $, then
%		$$ (  x_1\colon   X_1,\cdots,  x_m\colon   X_m)\ s(t_1(  x_1),\cdots, t_m(  x_m))\colon S^{\bb c} $$
%		is a term.

In the example below we will apply the restriction rule, taken along a functor $f\co \bb d\to\bb c$, to terms of the form $(\bar{x} \co \bar{X}) \ s \co X^{\bb c}$ where $X$ may not be a basic sort as in \cref{rule:restriction-def}. This should be understood as taking the restriction along $f\times 1_{\bb b}$ where $\bb b\in\FinCat$ is such that $X=S^{\bb b}$.

\begin{ese} \label{thm:first-terms} We provide some simple consequences of our deduction rules for terms, introducing some terms that will be useful.
\begin{enumerate}
\item For $s \co X^{[1]}$, we have terms $\dom(s) \co X$ and $\cod(s) \co X$, 
called the \emph{domain} and \emph{codomain} of $s$, respectively. These are obtained by applying the restriction rule to the two inclusion functors $\sigma_0,\sigma_1\colon [0] \to [1]$, respectively.
\item For $s \co X$, we have a term $\id_s \co X^{[1]}$, 
called the \emph{identity} at $s$, induced by unique functor $[1]\to [0]$.
%\item For $s \colon S^{[n]}$, we can define a term $\pr_{k,m}(s)\colon S^{[m]}$,
%induced by the inclusion $j_{k,m}\colon [m]\to [n]$ picking the string of $m$ composable morphisms starting from the $k$-th morphism, for $1\leq k+m\leq n+1$.
%\item For $s \colon S^{[2]}$, we can define terms $\dom(s) \co S$ and $\cod(s)\colon S$, also
%called the \emph{domain} and \emph{codomain}, 
%induced by the morphisms $\sigma_0 \mathrlap{,} \sigma_2 \colon [0] \to[2]$, picking the initial and terminal objects respectively.
\item For $s \colon X^{[2]}$, we have terms $\fst(s)\colon X^{[1]}$, $\snd(s)\colon X^{[1]}$ and $\comp(s)\colon X^{[1]}$, called the 
\emph{first component}, \emph{second component}, and \emph{composite} of $s$, induced by the morphisms $[1]\to[2]$, picking the respective arrows in the free-living composable pair $[2]$.
\item Let $1\leq k\leq n$. For $s \colon X^{[n]}$, to be thought of as a chain of $n$ composable morphisms, we can have a term $\pr_k(s)\colon X^{[1]}$, called the \emph{$k$-th component} of~$s$, induced by the functor $j_k\colon [1]\to [n]$ picking the $k$-th morphism of the chain.
\item  The category $[1]\times [1]$  consists of a commutative square,
\[
\begin{tikzcd}
(0,1) \ar[r, "u"] \ar[d, "\ell"'] & (1,1) \ar[d, "r"]  \\
(0,0) \ar[r, "d"']  & (1, 0) \mathrlap{.}
\end{tikzcd}
\]
For $s\colon X^{[1]\times [1]}$, we have terms
$\pru(s) \co X^{[1]}$,
$\prd(s) \co X^{[1]}$,
$\prl(s) \co X^{[1]}$,
$\prr(s) \co X^{[1]}$, 
induced by the four inclusions $\tau_u, \tau_d, \tau_l, \tau_r\colon [1]\to [1]\times [1]$ selecting the four edges of the square.
%\item For $s\colon X^{[1]\times [1]}$, we have terms
%$\pr_{u,r}(s) \co X^{[2]}$ and $\pr_{\ell,d}(s)\colon X^{[2]}$, 
%induced by the functors $j_{u,r},j_{\ell,d}\colon [2]\to [1]\times [1]$ picking the two composable pairs of the square.
\end{enumerate}
\end{ese}

\begin{running} \label{running:language}
 The language $\LL_{\tx{Gfib}}$ for the theory of Grothendieck fibrations consists of two basic sorts $E,B$, a function symbol $p\co E\to B$, and a relation symbol $\tx{Cart}\rightarrowtail E^{[1]}$ that will collect all cartesian arrows in $E$. 
When it will come to state the axioms for the the isoregular theory of Grothendieck fibrations, we will make use of the finite categories
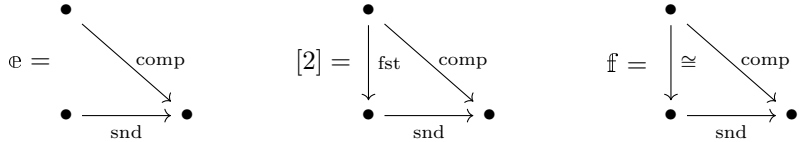
\begin{center}
	\begin{tikzpicture}[baseline=(current  bounding  box.south), scale=2, hookarrow/.style={{Hooks[right]}->}]
		
		\node (a') at (-1,0) {$\bullet$};
		\node (b') at (-1,-0.7) {$\bullet$};
		\node (e') at (-0.2,-0.7) {$\bullet$};
		\node (f') at (-1.25,-0.35) {$\bb e=$};

		\node (a) at (1,0) {$\bullet$};
		\node (b) at (1,-0.7) {$\bullet$};
		\node (e) at (1.8,-0.7) {$\bullet$};
		\node (f) at (0.7,-0.35) {$[2]=$};
		
		\node (a'') at (3,0) {$\bullet$};
		\node (b'') at (3,-0.7) {$\bullet$};
		\node (e'') at (3.8,-0.7) {$\bullet$};
		\node (f'') at (2.7,-0.35) {$\bb f=$};
		
		\path[font=\scriptsize]
		
		(a') edge [->] node [right] {comp} (e')
		(b') edge [->] node [below] {snd} (e')
		
		(a) edge [->] node [right] {fst} (b)
		(a) edge [->] node [right] {comp} (e)
		(b) edge [->] node [below] {snd} (e)
		
		(a'') edge [->] node [right] {$\cong$} (b'')
		(a'') edge [->] node [right] {comp} (e'')
		(b'') edge [->] node [below] {snd} (e'');
	\end{tikzpicture}
\end{center}
We write $\iota_{\bb e}\co \bb e \to [2]$ for the inclusion. With this notation, we have well-formed sorts $E^{\bb e}$,
$E^{\bb f}$, $E^{[2]}$ and, for example, propositions
\begin{align*} 
(z \co E^{[2]}) & \quad \tx{Cart}(\tx{snd}(z)) \co \mathsf{prop}  \mathrlap{,} \\
(u \co E^{[1] \times [1]}) &  \quad \tx{Cart}^{[1]}(u) \co \mathsf{prop}\mathrlap{,} 
\end{align*} 
asserting that the `second' map in a diagram of shape $[2]$ is cartesian and that a certain square is morphism of
cartesian maps, respectively.
\end{running}

%\begin{proof} \leavevmode
%\begin{enumerate}
%\item The terms induced by the two inclusions $\sigma_0,\sigma_1\colon [0] \to [1]$
%\item The term 
%\item Given the category $[n]$ with $n$ composable morphism, let 
%\item Given $[n]$ as above, the term corresponding to the inclusion $j_{k,m}\colon [m]\to [n]$ picking the string of $m$ composable morphisms starting from the $k$-th morphism, for $1\leq k+m\leq n+1$, is denoted by
%\end{enumerate}
%\end{proof}

\subsection{Isoregular theories}   Let $\LL$ be a fixed finitary language. The propositions of the calculus have the following six forms
\[
\top \mathrlap{,} \qquad 
R(s_1, \ldots, s_n) \mathrlap{,} \qquad
s = t \mathrlap{,} \qquad
A \land B \mathrlap{,} \qquad
(\exists x \co X) A \mathrlap{,}  \qquad
A^{[1]} \mathrlap{.} 
\]
The first five forms of propositions are familiar from first-order logic, while the sixth is specific to our  setting. For a proposition $A$, we refer to $A^{[1]}$ as the \emph{power} of $A$. Informally speaking, if a proposition
$A$ expresses properties of the objects of a category, the proposition $A^{[1]}$ expresses properties of its
morphisms. The deduction rules for \emph{isoregular logic} are presented in three groups:
\begin{itemize} 
\item rules for the formation of terms, in~\cref{tab:terms},
\item rules for the formation of propositions, in~\cref{tab:props},
\item rules for logical entailment, in  \cref{sec:deduction-rules}. 
\end{itemize}

One distinguishing aspect of this set of rules is that the deduction rules for forming new propositions depend on the deduction rules regarding logical entailment and viceversa.  This is in contrast with first-order logic, where one first defines the
set of well-formed formulas and then, separately, defines the set of theorems of a theory
by giving the deduction rules for logical entailment, but is quite a natural adaptation
of ideas from dependent type theory, where formation rules for types (which often can be seen as 
propositions, according to the propositions-as-types idea) can have premisses
involving inhabitation of other types (which can be seen as derivability of propositions).

Let us explain the rules for formation of existentially quantified propositions,  which allow us to form existential quantifiers only on propositions (provably) of a suitable kind. There are two such rules, 
given in \cref{equ:exists-formation-mono} and  \cref{equ:exists-formation-eqfib}, corresponding to
`there exists a unique' and `there exists a unique up to unique isomorphism', respectively.
The premiss $(x \co X, y \co Y) \ B(x,y) \co \mathsf{prop}$ of these rules can be thought informally as giving rise to a morphism
\begin{equation}
\label{equ:eqfib-informal} 
\{ x \co X, y \co Y \ | \ B(x,y) \} \rightarrowtail X \times Y \xrightarrow{ \pi_1} X \mathrlap{.}
\end{equation} 
In~\cref{equ:exists-formation-mono}, the judgement $\mathcal{J}_{\textsf{mono}}(B)$ expresses that, for $x \co X$, there exists at most one $y \co Y$ such that $B(x,y)$ holds, \ie asserting that the morphism in \eqref{equ:eqfib-informal} is a monomorphism. 
In \cref{equ:exists-formation-eqfib}, the judgements $\mathcal{J}_{\textsf{faithful}}(B)$ and 
$\mathcal{J}_{\textsf{full\text{-}id}}(B)$  express that, for $x \co X$, there exists at most one up to unique isomorphism $y \co Y$ such that~$B(x,y)$ holds, by asserting that the morphism in \eqref{equ:eqfib-informal} is faithful and
full on identities, and therefore an \ffk-morphism by \cref{thm:ff-kernel-char}.
Note that, using the rule in~\eqref{equ:exists-formation-mono}, the judgement $\mathcal{J}_{\textsf{faithful}}(B)$ implies that the existential quantifier in $\mathcal{J}_{\textsf{full\text{-}id}}(B)$ can actually be formed.  

\begin{table}[htb]
\begin{myframedbox}{
\begin{frule} 
\[
\begin{prooftree}
s \co X \qquad
t \co X 
\justifies
s = t \co \prp
\end{prooftree}
\]
\end{frule}
\medskip

\begin{frule}
\[
\begin{prooftree}
s_1 \co X_1 \quad
\ldots \quad
s_n \co X_n 
\justifies
R(s_1, \ldots, s_n) \co \prp
\end{prooftree}
\]
\end{frule}

\medskip
\begin{frule}
\[
\begin{prooftree}
A \co \prp \quad
B \co \prp 
\justifies
A \land B \co \prp
\end{prooftree}
\]
\end{frule}

\medskip
\begin{frule}
\[
\begin{prooftree}
A \co \prp 
\justifies
A^{[1]} \co \prp
\end{prooftree}
\]
\end{frule}

\begin{frule}
\label{equ:exists-formation-mono}
\[
\begin{prooftree}
(x \co X, y \co Y) \ B(x, y) \co \mathsf{prop} \qquad 
\mathcal{J}_{\mathsf{mono}}(B)
\justifies
(x \co X) \quad (\exists y \co Y) B(x,y) \co  \mathsf{prop}  \mathrlap{,} 
\end{prooftree} 
\]
where 
\[
\mathcal{J}_{\mathsf{mono}}(B) \defeq  (x \co X, y, y' \co Y) \ B(x, y), B(x, y') \vdash y = y' \mathrlap{.}
 \]
\end{frule}

\begin{frule}
\label{equ:exists-formation-eqfib}
\[
\begin{prooftree}
(x \co X, y \co Y) \ B(x, y) \co \prp \qquad 
\mathcal{J}_{\mathsf{faithful}}(B) \qquad
\mathcal{J}_{\mathsf{full\text{-}id}}(B) 
\justifies
(x \co X) \quad (\exists y \co Y) B(y) \co  \mathsf{prop}  \mathrlap{,} 
\end{prooftree}
\]
where  
\begin{align*} 
 & \mathcal{J}_{\mathsf{faithful}}(B) \defeq \\
 & \qquad \big( \bar{u} \co X^{[1]}, v, v' \co Y^{[1]}  \big) \quad
B^{[1]}(  \bar{u},  v),  \ B^{[1]}(  \bar{u},  v'),  \ \dom(  v)=\dom(  v'), \ \cod(  v)=\cod(  v')  \vdash   v=  v' \mathrlap{,} \\
& \mathcal{J}_{\mathsf{full\text{-}id}} (B) \defeq  \\
& \qquad 
\big( x \co X,  y, y' \co Y \big) \quad
B(  \bar{x},  y),  \ B(  \bar{x},  y')  \vdash (\exists   u\colon  Y^{[1]})  (B^{[1]}(\id_{\bar{x}},  u)\land \dom(  u)=  y\wedge\cod(  u)=  y' ) \mathrlap{.}
\end{align*} 
\end{frule}}
\end{myframedbox}
\caption{Formation rules for  propositions.}
\label{tab:props}
\end{table}

The deduction rule logical entailment for conjunction and existential quantification are the standard introduction and elimination rules, except that we need to include the so-called Frobenius rule (\cref{equ:frobenius}), which ensures a good behaviour of the existential quantifier in contexts where implication is not part of the logical calculus. Note that, whenever an existential statement occurs in the conclusion of a
deduction rule, the premisses ensure that, by the rules in~\eqref{equ:exists-formation-mono} and~\eqref{equ:exists-formation-eqfib} the existential quantifier can be formed. 
The other deduction rules for logical entailment express important properties of the 
intended models. For example, for a coequaliser of finite categories, 
\[
\begin{tikzcd}
\bb a \ar[r, shift left = 1, "f"] \ar[r, shift right = 1, "g"'] & \bb c \ar[r, "q"] & \bb c  \mathrlap{,}
\end{tikzcd}
\]
we have the deduction rule
\[
	\begin{prooftree}
		s \co S^{\bb b} \quad
		s \cdot f = s \cdot g 
		\justifies
		(\exists x \co S^{\bb c}) \big( x \cdot q = s \big) \mathrlap{.}
	\end{prooftree}
\]
Here, the existential quantifier can be formed according to our rules above since  the judgement
\[
(x, x' \co S^{\bb c}) \quad x \cdot q = s, x' \cdot q = s \vdash x = x' 
\]
is provable using \cref{equ:epimorphic-family}, as $q$ is an epimorphism. The rules for powers of terms and propositions express functoriality of the operation of raising terms and propositions to a power and the preservation of conjunction and existential quantifiers by forming the power of a proposition.

\begin{ese} \label{thm:easyprop} We give some examples of derivable judgements, 
establishing some properties of the terms introduced in~\cref{thm:first-terms}. 
\begin{enumerate}
%\item \red Delete? \black We can express that two chains of morphisms that are composable admit a composite.
%For example, the rule
%	\[
%	\begin{prooftree}
%	s \co X^{[2]} \qquad t \co X^{[1]} \qquad
%	\cod(s) = \dom(t) 
%	\justifies
%	(\exists u \co X^{[3]}) \pr_{1,2}(u)= s \wedge \pr_{3,1}(u)=t 
%	\end{prooftree}
%	\]
%	is derivable. Here, the conclusion expresses that the initial and terminal parts of the composite
%	are given by $s$ and $t$, respectively.
	\item \label{item:easyprop-ii} We can express that two triangles $t,t'\co X^{[2]}$ with common diagonal fit into a commutative square
	of the form
	\[
	\begin{tikzcd} 
	A \ar[r, "\fst(t')"]   \ar[d, "\fst(t)"']  & B \ar[d, "\snd(t')"] \\
	C \ar[r, "\snd(t)"'] & D
	\end{tikzcd}
	\]
	Indeed, the rule
	\[
	\begin{prooftree} 
	 t,t'\colon X^{[2]}  \qquad
	 \comp(t)=\comp(t')
	  \justifies
	 (\exists v \colon X^{[1]\times[1]}) \big( \pru(v)=\fst(t') \land \prd(v)=\snd(t) \land \prl(v)=\fst(t) \land \prr(v)=\snd(t') \big)
	 \end{prooftree}
	\]
	is derivable.
	
	\item Given a square $s\co X^{[1]\times [1]}$ we have terms $\dom(s)\co X^{[1]}$ and $\dom^{[1]}(s)\co X^{[1]}$ (similarly for $\cod$), where the first $\dom$ is relative to the context $X^{[1]}$ and the second is relative to $X$. Using the deduction rules it is easy to see that
	\[
	\begin{prooftree} 
		s\colon X^{[1]\times [1]}
		\justifies
		 \pru(s)=\dom^{[1]}(s) \land \prd(s)=\dom^{[1]}(t) \land \prl(s)=\dom(s) \land \prr(s)=\cod(s) 
	\end{prooftree}
	\]
	is derivable.

\end{enumerate}
\end{ese}

%
%\begin{lemma}
%	Let $L$ be an $\LL$-structure in $\K$. 
%\end{lemma}

%\begin{lemma}
%	Let $L$ be an $\LL$-structure in $\K$. If 
%	$A$ is interpretable over $L$ and $\mathsf{eqfib}(A)$ is valid in $L$, then 
%	the formula $(\exists x \colon X) A$ is interpretable over $L$.
%\end{lemma}
%

\begin{nota}[Substitution]
	Given a term $(\bar x\co \bar X)\  t(\bar x)\co Y$ and a formula in context $(y\co Y)\ A(y)$, we define
	$$(\bar x\co \bar X)\ A[t(\bar x)/y]\defeq (\bar x\co \bar X)\ (\exists y\co Y)\ A(y)\land t(\bar x)=y $$ 
	to denote the substitution of $y$ by $t(\bar x)$ in $A$. We prefer this approach to the classical recursive definition to avoid annoying technicalities generated by the power rule.
	Note that if $(y\co Y)\ A(y)\co \prp$ holds then so does $(\bar x\co \bar X)\ A[t(\bar x)/y]\co\prp$ since the uniqueness of the existential quantification is easily derivable.
\end{nota}

\begin{defi} \leavevmode
\begin{itemize}
\item  We define the class of \emph{isoregular theories} inductively as follows:
\begin{itemize}
\item the empty theory is isoregular;
\item if $\TT$ is an isoregular theory, $I$ is a set, and, for $i \in I$, 
\[
\mathcal{J}_i \defeq  (\bar{x}_i \co \bar{X}_i) \  A_{i,1}, \ldots, A_{i,n_i} \vdash A_{i} 
\]
is a judgement such that $(\bar{x}_i \co \bar{X}_i) \ A_{i, j} \co \tx{prop}$, for $1 \leq j \leq n_i$,
and $(\bar{x}_i \co \bar{X}_i) \ A_{i} \co \tx{prop}$ are derivable in $\TT$, then
$\TT \cup \{ \mathcal{J}_i  \ | \ i \in I \}$ is an isoregular theory.
\end{itemize} 
\item The \emph{derivable judgements} (or \emph{theorems}) of an isoregular theory are the judgements that are derivable from the axioms of the theory using the isoregular deduction rules. 
\end{itemize}
\end{defi} 

\begin{running} \label{running:theory} Building on our running example from \cref{running:language}, 
we define the theory of Grothendieck fibrations $\TT_{\tx{Gfib}}$ to consist of the following axioms:
\begin{enumerate}[itemsep=0.1cm]
	\item \hspace{16pt} $(z,w\co E^{[2]})\quad \tx{Cart}(\tx{snd}(z)),\ z\cdot \iota_{\bb e}=w\cdot \iota_{\bb e},\ p^{[1]}(\tx{fst}(z))=p^{[1]}(\tx{fst}(w))\vdash z=w$; 
	\item $(x\co E^{\bb e}, z\co B^{[2]})\quad \tx{Cart}(\tx{snd}(x)),\ p^{[1]}(\tx{snd}(x))=\tx{snd}(z),\ p^{[1]}(\tx{comp}(x))=\tx{comp}(z)\quad$ \\
	 \hspace*{75pt} \quad $\vdash  (\exists w\co E^{[2]})\quad w\cdot\iota_{\bb e}= x\wedge p^{[1]}(\tx{fst}(w))=\tx{fst}(z)\wedge \tx{Cart}(\tx{snd}(w))$;
	\item \hspace{31pt} $ (y\co E^{\bb f})\quad \tx{Cart}(\tx{snd}(y)) \vdash \tx{Cart}(\tx{comp}(y)) $;
	\item \hspace{12pt} $ (u\co E^{[1]\times [1]})\quad \tx{Cart}(\dom(u)), \tx{Cart}(\cod(u))\vdash \tx{Cart}^{[1]}(u) $;
	\item \hspace{3pt} $ (y\co E, u\co B^{[1]})\quad \cod(u)=p(y) \vdash (\exists v\co E^{[1]})\ \tx{Cart}(v)\wedge \cod(v)=y \wedge p^{[1]}(v)=u$.
\end{enumerate}
The first two axioms say that every map in $\tx{Cart}$ satisfies the unique lifting property, \ie is cartesian; axioms (iii) and (iv) say that $\tx{Cart}$ is closed under isomorphisms and is a full subcategory of $E^{[1]}$. Then, axiom (v) requires that every morphism in $B$ with codomain $p(y)$ has a cartesian lift.
\end{running}

\begin{rmk}
	If we remove \cref{equ:exists-formation-eqfib} from our deductive system we obtain a strict 2-dimensional version of cartesian theories, which only allows unique existential quantification. In this setting, it is still possible to prove that the syntactic 2-category $\Syn(\TT)$ of \cref{section:SynT} is finitely complete (rather than isoregular), and that models in a finitely complete 2-category $\K$ correspond to finite-limit-preserving 2-functors $\Syn(\TT)\to\K$. However, almost none of the examples of \cref{sec:applications} fit this framework. 
\end{rmk}

\subsection{Semantics}  \label{sec:semantics}

We shall be interested in interpreting formulas of these forms in  an isoregular 2-category with
a structure for $\LL$. Importantly, not every raw proposition will admit an interpretation. Instead, we define simultaneously when a proposition is interpretable and, in that case, what its interpretation is. More precisely, if $A$ is a proposition with free variables in the context  $(x_1 \co S_1^{\bb c_1},\ldots, x_n \co S_n^{\bb c_n})$, we define when $A$ is interpretable over $L$ and, in that case, its interpretation as a subobject
\[
\br{A}\rightarrowtail \br{S_1}^{\bb c_1}\times\cdots\times  \br{S_n}^{\bb c_n} 
\]
by recursion on the structure of the formula.

\begin{defi} \label{thm:structure} Let $\LL$ be a finitary language. Let $\K$ be a 2-category
with finite 2-limits. An $\LL$-structure $\M$ in $\K$  consists of:
\begin{enumerate}
	\item an object $\br{S}_{\M} \in \K$, for every basic sort $S\in\S$;
	\item a morphism 
	 $\br{f}_{\M} \colon \br{S_1}^{\bb c_1}_{\M} \times\cdots\times  \br{S_n}^{\bb c_n}_{\M} \to \br{S}^{\bb c}_{\M}$, 
	for every function symbol $f\colon S_1^{\bb c_1},\cdots, S_n^{\bb c_n}\to S^{\bb c}$ in $\LL$;
	\item a monomorphism
	$\br{R}_{\M} \rightarrowtail \br{S_1}^{\bb c_1}_{\M} \times\cdots\times  \br{S_n}^{\bb c_n}_{\M}$
	 for every relation symbol $R  \rightarrowtail S_1^{\bb c_1},\cdots ,S_n^{\bb c_n}$ in $\LL$.
\end{enumerate}  
\end{defi} 

Below, we shall drop the subscript in the interpretation of sorts, function symbols and relation symbols
whenever this does not cause confusion. 

\begin{rmk}\label{rmk:RT25}
	When $\K=\Cat$, this coincides with the notion of $\LL$-structure considered in~\cite[Definition~3.4]{RT25ERegular} specific to the case of 2-categories with chosen factorisation the (strong epi, mono). In fact, our isoregular theories and their models (in $\Cat$) can be seen as particular instances of the regular theories developed in~\cite{RT25ERegular}, where there is no restriction on the use existential quantification. Because of this, their 2-categories of models may in general not have flexible limits.
\end{rmk}

Assuming to have an $\LL$-structure $\M$ in $\K$ as in \cref{thm:structure}, we define the interpretation of
sorts and contexts. For a sort $X = S^{\bb c}$, we define
\[
\br{ S^{\bb c}} \defeq \br{S}^{\bb c } \mathrlap{.}
\]
For a context $(\bar{x} \co \bar{X})  = (x_1 \co X_1, \ldots, x_n \co X_n)$, we define
\[
\br{ \bar{X} } \defeq \br{ X_1} \times \ldots \times \br{X_n} \mathrlap{.}
\]
We then define the interpretation of the derivable judgements of the form~\eqref{equ:term-in-context},
so that a term $s \co Y$ in context $\bar{x} \co \bar{X}$ is interpreted as a morphism
\[
\br{s} \co \br{\bar{X}} \to \br{Y} \mathrlap{.}
\]
We do so by recursion on the derivation of the judgement, following the rules in \cref{tab:terms}.
First, for the variable rule, the interpretation of $(x_1 \co X_1, \ldots, x_n \co X_n)  \ \var_i(\bar x) \co X_i$, where
 $1 \leq i \leq n$, 
is defined to be the $i$-th projection: 
\[
\begin{tikzcd} 
\br{ X_1} \times \ldots \times \br{X_n} \ar[r, "\pr_i"] & \br{X_i}  \mathrlap{.} 
\end{tikzcd}
\]
For the function application rule, the interpretation of $f(s_1, \ldots, s_n) \co Y$ in context $\bar{x} \co \bar{X}$ is defined as the composite
\[
\begin{tikzcd}[column sep = huge]
\br{\bar{X}} \ar[r, "{(\br{s_1}, \ldots, \br{s_n})}"] & 
\br{Y_1} \times \ldots \times \br{Y_n} 
\ar[r, "\br{f}"] & 
\br{Y}  \mathrlap{,} 
\end{tikzcd}
\]
where, by the recursive hypothesis, we assumed that the interpretation of the judgements
$(\bar{x} \co \bar{X}) \ s_i \co Y_i$ is defined. For the restriction rule, we define the interpretation of $s \cdot f \co S^{\bb d}$ to be the composite
\[
\begin{tikzcd}
\br{\bar{X}} \ar[r, "\br{s}"] &
\br{S}^{\bb c} \ar[r, "\br{S}^f"] &
\br{S}^{\bb d} \mathrlap{.} 
\end{tikzcd}
\]
For the power rule, assuming we have defined $\br{s} \co \br{\bar{X}} \to \br{Y}$, we define the
interpretation of $s^{[1]} \co X^{[1]}$ in context $\bar{X}^{[1]}$  using the functoriality of powers, to be the
composite
\[
\begin{tikzcd} 
 \br{ \bar{X}^{[1]} } \ar[r, "\cong"] &  
\br{\bar{X}}^{[1]} \ar[r, "\br{s}^{[1]}"]  & 
 \br{Y}^{[1]} \ar[r, "\cong"] &
 \br{Y^{[1]}}   \mathrlap{,}  
 \end{tikzcd} 
\]
where, again, we assumed that the interpretation of $(\bar{x} \co \bar{X}) \ s \co Y$ is defined.
Finally, for the substitution rule, 
the interpretation of $t[s_1/y_i \ldots, s_n/y_n] \co Y$ in context $\bar{x} \co \bar{X}$ is
defined as the composite
\[
\begin{tikzcd}[column sep = huge]
\br{\bar{X}} \ar[r, "{(\br{s_1}, \ldots, \br{s_n})}"] & 
\br{Y_1} \times \ldots \times \br{Y_n} 
\ar[r, "\br{t}"] & 
\br{Y}  \mathrlap{.} 
\end{tikzcd}
\]

\begin{running} \label{running:structure} Let ${\bb L}_{\tx{GFib}}$ be the language for the theory of Grothendieck fibrations of
\cref{running:language}. A structure for it in the 2-category $\Cat$ consists of two
categories $\br{E}$ and $\br{B}$ together with a functor $\br{p} \co \br{E} \to \br{B}$ and
a subcategory $\br{\tx{Cart}} \subseteq \br{E}^{[1]}$. 
\end{running}

Let ${\M}$ be an $\LL$-structure in an isoregular 2-category $\K$. In order to extend the interpretation to the other forms of judgement,
some additional care is required, since the derivability of judgements of one form is related to the
derivability of the judgements of the other form, as familiar from dependent type theory. In particular, it is not possible to define the interpretation
of propositions first and define what it means for a logical entailment to be valid separately. Instead,
we start by defining a partial interpretation function mapping a judgement of the 
form $(\bar{x} \co \bar{X}) \ A \co \mathsf{prop}$ to a subobject $ \br{A} \rightarrowtail \br{\bar{X}} $
in $\K$ by recursion on the structure of the expression~$A$. This partial function will be proved to be total on the derivable judgements. To this end, we record when each clause is well-defined.

{\setlength{\leftmargini}{1.6em}
		\begin{enumerate}
		\item The interpretation of the judgement $(\bar{x} \co \bar{X}) \ R(s_1, \ldots, s_n) \co \mathsf{prop}$ is  the pullback
		\begin{center}
				\begin{tikzpicture}[baseline=(current  bounding  box.south), scale=2]
					\node (a0) at (0,0.8) {$\br{R(s_1, \ldots, s_n)}$};
					\node (a0') at (0.2,0.6) {$\lrcorner$};
					\node (b0) at (1.1,0.8) {$\br{\bar{X}}$};
					\node (c0) at (0,0) {$\br{R}$};
					\node (d0) at (1.1,0) {$\br{\bar{Y}}$};
					
					\path[font=\scriptsize]
					
					(a0) edge [>->] node [above] {} (b0)
					(a0) edge [->] node [left] {} (c0)
					(b0) edge [->] node [right] {$(\br{s_1}, \ldots, \br{s_n})$} (d0)
					(c0) edge [>->] node [below] {$m_R$} (d0);
				\end{tikzpicture}	
			\end{center}
			This is well-defined if we have $m_R \co \br{R} \rightarrowtail \br{\bar{Y}}$ and 
			$\br{s_i} \co \br{\bar{X}} \to \br{Y_i}$, for all $1 \leq i \leq n$.
			\item  The interpretation of the judgement $(\bar{x} \co \bar{X}) \ s=t \co \mathsf{prop}$ is the equaliser of $\br{s}$ and $\br{t}$:
			\[
			\begin{tikzcd}
			\br{s = t} \ar[r, >->] & 
			\br{\bar{X}} \ar[r, "\br{s}", shift left = 1] \ar[r, "\br{t}"', shift right = 1] &
			\br{Y} 
			\end{tikzcd}
			\]	
%			\begin{center}
%				\begin{tikzpicture}[baseline=(current  bounding  box.south), scale=2]
%					
%					\node (a0) at (0,0) {$\br{ s = t }$};
%					\node (c0) at (1.8,0) {$ \br{\bar{X}} $};
%					
%					\path[font=\scriptsize]
%					
%					(a0) edge [>->] node [below] {$\tx{eq}(\br{s},\br{t})$} (c0);
%				\end{tikzpicture}	
%			\end{center}
			This is well-defined if we have $\br{s} \co 
			\br{\bar{X}} \to \br{Y}$ and $\br{t} \co \br{\bar{X}} \to \br{Y}$.
%			\item The formula $A[s_1/y_1, \ldots s_n/y_n]$  is interpretable if $A$ is and its interpretation is the pullback:
%			\begin{center}
%				\begin{tikzpicture}[baseline=(current  bounding  box.south), scale=2]
%					
%					\node (a0) at (0,0.8) {$\br{A[s_1/y_1, \ldots s_n/y_n]}$};
%					\node (a0') at (0.2,0.6) {$\lrcorner$};
%					\node (b0) at (1,0.8) {$\br{X}$};
%					\node (c0) at (0,0) {$\br{A}$};
%					\node (d0) at (1,0) {$\br{Y}$};
%					
%					\path[font=\scriptsize]
%					
%					(a0) edge [>->] node [above] {} (b0)
%					(a0) edge [->] node [left] {} (c0)
%					(b0) edge [->] node [right] {$(\br{s_1}, \ldots, \br{s_n})$} (d0)
%					(c0) edge [>->] node [below] {$m_A$} (d0);
%				\end{tikzpicture}	
%			\end{center}
			\item The interpretation of the judgement $(\bar{x} \co \bar{X}) \ A \wedge B \co \mathsf{prop}$ is  the pullback
			\begin{center}
				\begin{tikzpicture}[baseline=(current  bounding  box.south), scale=2]
					
					\node (a0) at (0,0.8) {$\br{A \wedge B}$};
					\node (a0') at (0.2,0.6) {$\lrcorner$};
					\node (b0) at (1,0.8) {$\br{B}$};
					\node (c0) at (0,0) {$\br{A}$};
					\node (d0) at (1,0) {$\br{\bar{X}} \mathrlap{.}$};
					
					\path[font=\scriptsize]
					
					(a0) edge [>->] node [above] {} (b0)
					(a0) edge [>->] node [left] {} (c0)
					(b0) edge [>->] node [right] {$m_B$} (d0)
					(c0) edge [>->] node [below] {$m_A$} (d0);
				\end{tikzpicture}
			\end{center}
			This is well-defined if the interpretations of the judgements 
			$(\bar{x} \co \bar{X}) \ A \co \mathsf{prop}$ and $(\bar{x} \co \bar{X}) \ B \co \mathsf{prop}$
			are so.
			\item The interpretation of the judgement $(\bar{x} \co \bar{X}^{[1]}) \vdash A^{[1]} \co \mathsf{prop}$ is the pullback:
			\begin{center}
				\begin{tikzpicture}[baseline=(current  bounding  box.south), scale=2]
					
					\node (a0) at (0,0.8) {$\br{ A^{[1]} }$};
					\node (a0') at (0.2,0.6) {$\lrcorner$};
					\node (b0) at (1.1,0.8) {$\br{\bar{X}^{[1]}}$};
					\node (c0) at (0,0) {$\br{A}^{[1]}$};
					\node (d0) at (1.1,0) {$\br{\bar{X}}^{[1]}$};
					
					\path[font=\scriptsize]
					
					(a0) edge [>->] node [above] {} (b0)
					(a0) edge [>->] node [left] {$\cong$} (c0)
					(b0) edge [>->] node [right] {$\cong$} (d0)
					(c0) edge [>->] node [below] {$m_A^{[1]}$} (d0);
				\end{tikzpicture}
			\end{center}
			This is well-defined if the interpretation of the judgement
			$(\bar{x} \co \bar{X}) \ A \co \mathsf{prop}$ is so. 
			\item The interpretation of the judgement $(\bar{x} \co \bar{X}) \ (\exists y \colon Y) B(\bar{x}, y) \co \mathsf{prop}$  is the (fully faithful regular epimorphism, monomorphism)-factorisation 
			\begin{center}
				\begin{tikzpicture}[baseline=(current  bounding  box.south), scale=2]
					
					\node (a0) at (-0.2,0.8) {$\br{B}$};
					\node (c0) at (1,0.8) {$ \br{\bar{X}}\times\br{Y} $};
					\node (c1) at (2.2,0.8) {$ \br{\bar{X}}$};
					\node (d0) at (1,1.3) {$\br{ (\exists y \co Y) B(\bar{x}, y)}$};
					
					\path[font=\scriptsize]
					
					(a0) edge [>->] node [below] {$m_B$} (c0)
					(c0) edge [->] node [below] {$p_1$} (c1)
					(a0) edge [->>] node [above] {$\ \ \ $} (d0)
					(d0) edge [>->] node [above] {$\ \ \ $} (c1);
				\end{tikzpicture}	
			\end{center}	
			where $p_1$ is the projection on the first factor. This is well-defined
			if the interpretation of the judgement $(\bar{x} \co \bar{X}, y \co Y) \ B(\bar{x}, y) \co \mathsf{prop}$
			is well-defined and $m_B \circ p_1$  is an \ffk-morphism in $\K$. 
	\end{enumerate}}

\begin{defi} Let $\LL$ be a finitary language. Let $\K$ be an isoregular 2-category and $\M$ be an $\LL$-structure in $\K$.
\begin{enumerate}
\item The judgement $(\bar{x} \co \bar{X}) \ s \co Y$ is valid if $\br{s} \co \br{\bar{X}} \to \br{Y}$.
\item The judgement $(\bar{x} \co \bar{X}) \ A \co \mathsf{prop}$ is valid if its interpretation is well-defined.
\item The judgement $(\bar{x} \co \bar{X}) \ A_1, \ldots, A_n \vdash A$ is valid if
$\br{A_1} \cap \ldots \cap \br{A_n} \leq \br{A}$ as subobjects of $\br{\bar{X}}$.
\item A deduction rule
\[
\begin{prooftree}
\mathcal{J}_1 \quad \ldots \quad \mathcal{J}_n 
\justifies
\mathcal{J}
\end{prooftree}
\]
is valid if, whenever $\mathcal{J}_1 \ldots \mathcal{J}_n$ are valid, so is $\mathcal{J}$.
\end{enumerate}
\end{defi}

The next lemma can be understood as a form of definability. 

\begin{lemma} \label{thm:semantics-eqfib}
Let $(\bar{x} \co \bar{X}, y \co Y) \ B(\bar{x}, y) \co \mathsf{prop}$ be a derivable judgement and consider the morphism
\[
\begin{tikzcd}
\br{B} \ar[r, tail, "m_B"] & \br{X} \times \br{Y} \ar[r, "\pr_1"] & \br{X}
\end{tikzcd}
\]
\begin{enumerate}
\item The judgement $\mathcal{J}_{\mathsf{mono}}(B)$ is valid if and only if the morphism $\pr_1 \circ m_B$
is a monomorphism.
\item The judgements $\mathcal{J}_{\mathsf{faithful}}(B)$ and $\mathcal{J}_{\mathsf{full\text{-}id}}(B)$ are valid if and only if the morphism $\pr_1 \circ m_B$
is an \ffk-morphism.
\end{enumerate}
\end{lemma} 

\begin{proof} We only prove the `only if' statements. For part~(i), the first step is to note that the validity of the judgement
\[
(x \co X, y \co Y, y' \co Y) \ B(x,y) \land B(x,y') \vdash y = y' \mathsf{prop}
\]
expresses that we have a factorisation of the form
\[
\begin{tikzcd}
\br{B(x,y)} \cap \br{B(x,y')} \ar[rr, dotted] \ar[dr, >->] & & \br{X} \times \br{ y = y' } \ar[dl, >->] \\
 & \br{X} \times \br{Y} \times \br{Y} & 
 \end{tikzcd}
 \] 
 where the right-hand morphism is obtained by pulling back the diagonal of $\br{Y}$. Since the domain of the dotted arrow coincides with the kernel pair of  $\pr_1 \circ m_B$, this is easily seen to imply that $\pr_1 \circ m_B$ is a monomorphism.
 
 For part~(ii), the claim on faithfulness can be proved similarly to part~(i). For the claim on fullness on identities,
 for $x \co X, y, y' \co Y, u \co Y^{[1]}$, let 
 \[
 C(x,y,y',u) \defeq B(x,y) \land B(x,y') \land B^{[1]}(\id_x, y) \land \dom(u) = y \land \cod(u) = y'
 \]
 By assumption on faithfulness, the composite
 \[
 \br{ C(x,y,y',u) } \rightarrowtail \br{X} \times \br{Y} \times \br{Y} \times \br{Y}^{[1]} \rightarrow 
\br{X} \times \br{Y} \times \br{Y} 
  \]
  is a monomorphism. Therefore, the assumption means that we have an inclusion
  \[
  \br{B(x,y)} \cap \br{B(x,y')} \leq \br{C(x,y,y',u)} 
\]
which can easily be shown to imply the required fullness on identitites by unfolding the definition of the
intepretation of $C$. 
\end{proof}

\begin{defi}
	Let $\TT$ be an isoregular theory and $\K$ an isoregular 2-category. An $\LL$-structure $\M$ in $\K$ is called a {\em model} of $\TT$ if every axiom of $\TT$ is valid in $\K$.
\end{defi}

We establish the soundness of the deduction rules of isoregular theories.

\begin{theo} Let $\TT$ be an isoregular theory,  $\K$ an isoregular 2-category, and $\M$ a model
of $\TT$ in $\K$. Every theorem of $\TT$ is $\K$ is valid. 
\end{theo}

\begin{proof}  The proof proceeds by induction on the derivations, checking that all the deduction rules
of the isoregular deductive calculus in \cref{sec:deduction-rules} are valid. 

The deduction rules for atomic formulas, conjunction, and existential quantification are  treated essentially as in the 1-categorical case, but one should use \cref{thm:semantics-eqfib} to prove the validity of the formation rules for the existential quantifier.

The rule for the terminal object is valid since $\br{S^0}$ is terminal. The rules for restriction with the axioms for an action are clearly valid. For the one on jointly epimorphic family, observe that we have a diagram
\[
\begin{tikzcd}
\br{\bar{X}} 
	\ar[r, shift left = 1, "\br{s}"] 
	\ar[r, shift right = 1, "\br{t}"'] &
\br{S}^{\bb c} \ar[r, "\br{S}^{f_i}"] &
\br{S}^{\bb d}  \mathrlap{,}
\end{tikzcd}
\]
where the family $\big( \br{S}^{f_i} )_{1 \leq i \leq n}$ is jointly monomorphic. Similarly, for the rule on
coequalisers, the premisses give us a diagram of solid arrows
\[
\begin{tikzcd}
\br{\bar{X}} \ar[d, dotted] \ar[dr, "\br{s}"]  & & \\
\br{S}^{\bb c } \ar[r, "\br{S}^q"'] & 
\br{S}^{\bb d} \ar[r, shift left = 1, "\br{S}^f"] 
	\ar[r, shift right = 1, "\br{S}^g"'] & 
\br{S}^{\bb a}
\end{tikzcd}
\]
where $\br{S}^q$ is the equaliser of $\br{S}^f$ and $\br{S}^g$. The universal property of
equalisers provides the desired conclusion.  Finally, the rules about powers are valid 
either using that the 2-functor $(-)^{[1]} \co \K \to \K$ preserves all the structure
or via direct calculations. For example, the validity of the rule
\[
\begin{prooftree}
x \co S^{\bb c } \vdash A(x) \co \mathsf{prop} \qquad 
s \co (S^{\bb c })^{[1]} 
\justifies
A^{[1]}(s)  \vdash A( \dom(s) ) 
\end{prooftree}
\]
is obtained by pulling back the diagram
\[
\begin{tikzcd}
\br{A}^{[1]} \ar[d, dotted] \ar[r, tail] & \br{ S^{\bb c} }^{[1]} \ar[d, "\dom"] \\
\br{A} \ar[r, tail] & \br{S^{\bb c}}
\end{tikzcd}
\]
along the morphism $\br{s} \co \br{\bar{X}} \to \br{S^{\bb c}}$. 
\end{proof} 

We are now ready to define the 2-categories of structures over a language, and of models over an isoregular theory.

\begin{defi} Let $\TT$ be an isoregular theory,  $\K$ an isoregular theory, $\M$ and $\N$ be structures for $\TT$ in $\K$.
\begin{itemize}
\item A \emph{morphism of structures} $p \co \M \to \N$ consists of a morphism $p_S \co \br{S}_{\M} \to \br{S}_{\N}$, for every basic sort $S$, such that, for every function symbol $f$,
the following diagram commutes
\[
\begin{tikzcd} 
\br{\bar{X}}_{\M} \ar[d, "\br{f}_{\M}"'] \ar[r, "p_{\bar{X}}"] &  \br{\bar{X}}_{\N} \ar[d, "\br{f}_{\N}"] \\
\br{X}_{\M} \ar[r, "p_X"'] & \br{X}_{\N}  \mathrlap{,} 
\end{tikzcd} 
\]
 where $p_{\bar{X}}$ and $p_X$ are the evident morphisms, and for every  relation symbol $R$,
 there exists a (necessarily unique) dotted map making the diagram 
\[
\begin{tikzcd}[column sep = huge]
\br{R}_{\M} \ar[d, tail] \ar[r, dotted] &  \br{R}_{\N} \ar[d, tail] \\
\br{\bar{X}}_{\M} \ar[r, "p_{\bar{X}}"'] & \br{\bar{X}}_{\N} 
\end{tikzcd}
\]
commute.
\item A \emph{structure morphism 2-cell}  $\phi \co p \Rightarrow q$ of morphisms of structures consists of  a 2-cell $\phi_S \co p_S \Rightarrow q_S$, for every basic sort $S$, such that, for any function symbol $f$, the following commutes
\begin{center}
	\begin{tikzpicture}[baseline=(current  bounding  box.south), scale=2]
		
		\node (a0) at (0,0.9) {$\br{\bar{X}}_{\M}$};
		\node (ab0) at (0.8,0.9) {${\scriptsize \Downarrow {\phi_{\bar X}}}$};
		\node (b0) at (1.5,0.9) {$\br{\bar{X}}_{\N}$};
		\node (c0) at (0,0) {$\br{X}_{\M}$};
		\node (cd0) at (0.8,0) {$\Downarrow \phi_{X}$};
		\node (d0) at (1.5,0) {$\br{X}_{\N},$};
		
		\path[font=\scriptsize]
		
		(a0) edge [bend left=20, ->] node [above] {$p_{\bar X}$} (b0)
		(a0) edge [bend right=20, ->] node [below] {$q_{\bar X}$} (b0)
		(a0) edge [->] node [left] {$\br{f}_{\M}$} (c0)
		(b0) edge [->] node [right] {$\br{f}_{\N}$} (d0)
		(c0) edge [bend left=20, ->] node [above] {$p_X$} (d0)
		(c0) edge [bend right=20, ->] node [below] {$q_X$} (d0);
	\end{tikzpicture}
\end{center}
where $\phi_{\bar{X}}$ and $\phi_X$ are the evident morphisms induced by $\phi$, and for every  relation symbol $R$,
there exists a (necessarily unique) dotted 2-cell making the diagram 
\begin{center}
	\begin{tikzpicture}[baseline=(current  bounding  box.south), scale=2]
		
		\node (a0) at (0,0.9) {$\br{R}_{\M}$};
		\node (ab0) at (0.75,0.9) {$\Downarrow $};
		\node (b0) at (1.5,0.9) {$\br{\R}_{\N}$};
		\node (c0) at (0,0) {$\br{\bar X}_{\M}$};
		\node (cd0) at (0.8,0) {$\Downarrow \phi_{\bar X}$};
		\node (d0) at (1.5,0) {$\br{\bar X}_{\N},$};
		
		\path[font=\scriptsize]
		
		(a0) edge [dotted, bend left=20, ->] node [above] {} (b0)
		(a0) edge [dotted, bend right=20, ->] node [below] {} (b0)
		(a0) edge [>->] node [left] {} (c0)
		(b0) edge [>->] node [right] {} (d0)
		(c0) edge [bend left=20, ->] node [above] {$p_{\bar X}$} (d0)
		(c0) edge [bend right=20, ->] node [below] {$q_{\bar X}$} (d0);
	\end{tikzpicture}
\end{center}
commute.
\end{itemize}
\end{defi}

Structures, morphisms of structures, and structure morphism 2-cells form an evident 2-category that we denote
as $\Str(\LL,\K)$. We then define the \emph{$2$-category of models} of $\TT$ in $\K$, written $\Mod(\TT,\K)$,
as the full subcategory of $\Str(\LL,\K)$ spanned by the models of $\TT$.

\begin{running} \label{running:model} As we shall discuss in more detail in \cref{sec:applications}, a model of the
theory ${\bb T}_{\tx{GFib}}$ is a Grothendieck fibration and its 2-category of models 
$\Mod({\bb T}_{\tx{GFib}},\Cat)$ is isomorphic to the
2-category of Grothendieck fibrations and cartesian functors between them.
\end{running}

\begin{rmk}
	Following \cref{rmk:RT25}, $\Str(\LL,\Cat)$ coincides with the 2-category constructed in \cite[Definition~3.8]{RT25ERegular} which is locally presentable as a 2-category by \cite[Theorem~3.9]{RT25ERegular}. Adapting the same proof it is easy to see that $\Str(\LL,\K)$ is locally finitely presentable whenever $\K$ is so, and that the forgetful 2-functor
	$$ \Str(\LL,\K)\longrightarrow \prod_{S\in\S}\K\mathrlap{,} $$
	obtained by evaluating on sorts, is continuous and preserves filtered colimits. We will study the accessibility of $\Mod(\TT,\K)$ in \cref{sec:accessibility}.
\end{rmk}

We conclude this section with two lemmas on the 2-categories of models which will be useful later.

\begin{lemma}\label{lemma:morphisms.respect.formulas}
	Let $\TT$ be an isoregular theory, $\K$ be an isoregular 2-category, and $p \colon \M\to \N$ a morphism in $\Mod(\TT,\K)$. Then for any judgement $(\bar{x} \co \bar{X})\ A \co \mathsf{prop}$, there exists a unique morphism $p_A$ making the diagram
		\begin{center}
		\begin{tikzpicture}[baseline=(current  bounding  box.south), scale=1.5]
			
			\node (a0) at (0,0.8) {$\br{A}_{\M}$};
			\node (b0) at (1.5,0.8) {$\br{A}_{\N}$ };
			\node (c0) at (0,0) {$\br{\bar{X}}_{\M}$};
			\node (d0) at (1.5,0) {$\br{\bar{X}}_{\N}$};
			
			\path[font=\scriptsize]
			
			(a0) edge [dotted, ->] node [above] {$p_{A}$} (b0)
			(a0) edge [>->] node [left] {} (c0)
			(b0) edge [>->] node [right] {} (d0)
			(c0) edge [->] node [below] {$p_{\bar{X}}$} (d0);
		\end{tikzpicture}	
	\end{center}
	commute. 
\end{lemma}
\begin{proof}
	This is easily proved by induction on the complexity of the formula.
\end{proof}

\begin{lemma}\label{lemma:powers of models}
	Let $\TT$ be an isoregular theory and $\K$ an isoregular 2-category.  Then the 2-category $\Mod(\TT,\K)$ is closed in $\Str(\LL,\K)$ under powers by $[1]$.
\end{lemma}

\begin{proof} The claim follows since powers by $[1]$ in $\Str(\LL,\K)$ are computed componentwise, and for any model $\M$ and judgment $(\bar{x} \co \bar{X})\ A \co \mathsf{prop}$ one has $\br{A}_{\M^{[1]}}\cong \br{A}_\M^{[1]}$.
\end{proof}

\section{Functorial semantics for isoregular theories}
\label{sec:functorial-semantics}

\subsection{The syntactic 2-category of an isoregular theory}\label{section:SynT}

We construct the syntactic 2-category  of an isoregular theory. For this,
we introduce some definitions.

\begin{defi} Let $\TT$ be a isoregular theory. 
\begin{itemize}
\item A  \emph{formula-in-context}, written $\{ \bar{x} \co \bar{X}  \  | \ A(\bar{x}) \}$,  is a pair consisting of a context $\bar{X}$ and a raw formula $A$ such that the following judgement is derivable in $\TT$:
\[
(\bar{x} \co \bar{X}) \quad A \co \prp \mathrlap{.}
\]
\item  A \emph{morphism} $[f]\colon \{   \bar{x} \co \bar{X} | \ A( \bar{x})\}\to  \{  \bar{y} \co \bar{Y}  |\  B(\bar{y})\}$ between
formulas-in-context is an equivalence class of formulas-in-context  
		$\{  \bar{x} \co \bar{X}, \bar{y} \co \bar{Y}   |\ f( \bar{x}, \bar{y})\}$
		such that the following judgements are derivable in $\TT$:
		\begin{gather*} 
		 f(\bar{x},\bar{y})\vdash A(\bar{x})\wedge B(\bar{y}) \mathrlap{,} \\
		f(\bar{x},\bar{y}) \mathrlap{,} \  f(\bar{x},\bar{y}')\vdash \bar{y}=\bar{y}' \mathrlap{,} \\
		A(\bar{x})\vdash (\exists \bar{y} \co \bar{Y}) f(\bar{x},\bar{y}) \mathrlap{.} 
		\end{gather*}
		Two such formulas-in-context are equivalent if each is provable from the other in $\TT$.
\item  A \emph{$2$-cell} $[\alpha]\colon [f]\Rightarrow [f']$ between morphisms with domain 
$\{   \bar{x} \co \bar{X}   \  | \ A(\bar{x}) \}$ and codomain $\{ \bar{y} \co \bar{Y}   \  | \ B(\bar{y}) \}$ is an equivalence class of formulas-in-context  
		$\{ \bar{x} \co \bar{X}, \bar{v} \co \bar{Y}^{[1]}   \ | \ \alpha( \bar{x},  \bar{v}) \}$
		such that the following judgements are provable in $\TT$:
		\begin{gather*} 
			\alpha(\bar{x},\bar{v})\vdash A(\bar{x})\wedge B^{[1]}(\bar{v}) \mathrlap{,} \\
			\alpha(\bar{x},\bar{v})\wedge\alpha(\bar{x},\bar{v}')\vdash \bar{v}=\bar{v}' \mathrlap{,} \\
			A(\bar{x})\vdash (\exists \bar{v} \co \bar{Y}^{[1]}) \alpha (\bar{x},\bar{v}) \mathrlap{,}   \\
			\alpha(\bar{x},\bar{v})\vdash f(\bar{x},\mathsf{dom}(\bar{v}))\wedge f'(\bar{x},\mathsf{cod}(\bar{v})) \mathrlap{.}
			\end{gather*}
		 Two such formulas-in-context are equivalent if each is mutually derivable in $\TT$.
\end{itemize}
\end{defi}

The next results introduce the syntactic 2-category of $\TT$ (\cref{thm:synT-def}) and establish
that it is isoregular (\cref{thm:synTisoregular}), by first showing that it  it
has finite 2-limits (\cref{prop:limits-in-SynT}). For this, we put to work the deduction rules
for isoregular theories. In order to simplify notation and make our arguments more readable,
we  work only with contexts of the form  $(x \co X)$, $(y \co Y)$, $(x \co X, y \co Y)$. Of course, all arguments carry over to more general contexts.

\begin{prop} \label{thm:synT-def}
There is a 2-category $\Syn(\TT)$, called the syntactic 2-category of $\TT$, with propositions-in-context as objects, definable morphisms as morphisms,
and definable 2-cells as 2-cells.
\end{prop}

\begin{proof} Let $\{x \co   X \ |\ A(  x)\}$ and $\{y \co Y \ | \ B( y)\}$ be formulas-in-context. The category
\[
\Syn(\TT)\big( \{x \co X \ |\ A(  x)\}, \{y \co  Y \ |\ B(  y)\} \big)
\]
has morphisms from $ \{x \co X \ |\ A(  x)\}$ to $\{y \co  Y \ |\ B(  y)\}$ as objects and 2-cells between them
as maps. For 2-cells $[\alpha]\colon [f]\Rightarrow [f']$ and $[\beta]\colon [f']\Rightarrow [f'']$, their vertical composite $[\beta]\cdot [\alpha] \co [f] \Rightarrow [f''] $ is represented by the formula
\[
\{  x \co X, u \co Y^{[1]} \ |\ (\exists h \co Y^{[2]})\ \alpha(x,\fst(h))\wedge \beta(x,\snd(h))\wedge \tx{comp}(h)=u  \} \mathrlap{.}
\]
For a morphism $[f]\colon \{x \co  X|\ A(  x)\}\to  \{y \co  Y \ |\ B(  y)\}$, the identity 2-cell on it is represented by the formula-in-context
\[
\{  x \co X,  u: Y^{[1]} |\ (\exists y \co Y)  f(x,y)\wedge u=\id_y \} \mathrlap{.}
\]
Next, we consider the composition functors. Let $\{x \co   X  \ | \ A(  x)\}$ and $\{y \co  Y \ | \ B(y)\}$ and
$\{ z \co Z \ | \ C(z) \}$ be formulas-in-context. We define a functor
\begin{multline*} 
\Syn(\TT)\big( \{y \co Y \ |\ B(y)\}, \{z \co  Z \ |\ C(z)\} \big) \times 
\Syn(\TT)\big( \{x \co X \ |\ A(  x)\}, \{y \co  Y|\ B(  y)\} \big)   \\
\xrightarrow{(-) \circ (-)} \Syn(\TT)\big( \{x \co X \ |\ A(  x)\}, \{z \co  Z \ |\ C(z)\} \big)
\end{multline*}
as follows. For $[f]\colon \{x \co   X \ | \ A(x)\} \to   \{y \co   Y \ |\ B(  y)\}$ and $[g]\colon  \{y:   Y|\ B(  y)\}\to   \{z:   Z|\ C(  z)\}$, their composite $[g]\circ[f]$ is represented by the formula
\[
\{  x \co X, z \co  Z\ |\  (\exists y \co Y) f(x,y)\wedge g(y,z) \} \mathrlap{.}
\]
For a 2-cell $[\alpha]\colon [f]\Rightarrow [f']$ and a morphism $[g]\colon \{y:   Y\ |\ B(  y)\}\to  \{z:   Z \ |\ C(  z)\}$, their horizontal composite $[g]\circ [\alpha] \co [ g \circ f] \Rightarrow [g \circ f']$ is represented by the formula
\[
\{ x \co X, u \co Z^{[1]} \ |\ (\exists v:  Y^{[1]}) \alpha(x,v)\wedge  g^{[1]}(v,u)  \} \mathrlap{.}
\]
 Given a morphism $[f]\colon \{x:   X\ |\ A(  x)\}\to  \{y:   Y \ |\ B(  y)\}$ and a 2-cell $[\beta]\colon [g]\Rightarrow [g']$, their horizontal composite $[\beta]\circ [f] \co [g \circ f] \Rightarrow [g' \circ f]$ is represented by the formula
\[
\{ x \co X, w \co Z^{[1]}  \ |\ (\exists y \co Y) f(x,y)\wedge \beta(y,w) \} \mathrlap{.}
\]
Finally, for an object $ \{x:   X|\ A(  x)\}$, the identity morphism on it is represented by the formula-in-context
\[
\{ x \co X, x' \co  X |\  x= x'\} \mathrlap{.}
\]

We need to ensure that everything is well-defined, and that this is indeed a 2-category.
First, one checks the axioms hom-categories. To show that the composite of 2-cells is
a 2-cell and that the identity 2-cell is a 2-cell, one uses \cref{equ:epimorphic-family} for the required uniqueness property and \cref{equ:power-functoriality-terms}, respectively.
Here, associativity of composition can be shown
using the rules in \cref{thm:easyprop}, terms expressing that composition is associative in $\Cat$, and 
\cref{equ:action}.
Secondly, to prove that the horizontal composites of 2-cells with 1-cells are 2-cells, one uses \cref{equ:power-functoriality-formulas} and \cref{equ:power-dom-cod-prop}. For the interchange law, we need to consider
%\[
%\begin{tikzcd}
%\{ x \co X \ | \ A(x) \} \ar[r, bend left = 20, "{[f]}"]  \ar[r, bend right = 20, "{[f']}"']  &
%\{ y \co Y \ | \ B(y) \} \ar[r, bend left = 20, "{[g]}"]  \ar[r, bend right = 20, "{[g']}"']  &
%\{ z \co Z \ | \ C(z) \} 
%\end{tikzcd}
%\]
\begin{center}
	\begin{tikzpicture}[baseline=(current  bounding  box.south), scale=2]
		
		\node (a0) at (0,0.9) {$\{ x \co X \ | \ A(x) \}$};
		\node (ab0) at (1,0.9) {${\scriptsize \Downarrow {[\alpha]}}$};
		\node (b0) at (2,0.9) {$\{ y \co Y \ | \ B(y) \} $};
		\node (cd0) at (3,0.9) {${\scriptsize \Downarrow {[\beta]}}$};
		\node (d0) at (4,0.9) {$\{ z \co Z \ | \ C(z) \}$};
		
		\path[font=\scriptsize]
		
		(a0) edge [bend left=22, ->] node [above] {$[f]$} (b0)
		(a0) edge [bend right=22, ->] node [below] {$[f']$} (b0)
		(b0) edge [bend left=22, ->] node [above] {$[g]$} (d0)
		(b0) edge [bend right=22, ->] node [below] {$[g']$} (d0);
	\end{tikzpicture}
\end{center}
and show that $\big( [\beta] \circ [f'] \big) \circ \big( [g] \circ [\alpha] \big) = \big( [g'] \circ [\alpha] \big) 
\circ \big( [\beta] \circ [f] \big)$. This is a direct calculation, using the naturality of $\beta$, 
thinking about a situation of the form (but keeping in mind that this is not an actual diagram in
$\Syn(\TT)$):
\[
\begin{tikzcd}[column sep = large] 
 g f x  \ar[r, "g ( \alpha_x )"] \ar[d, "\beta_{f x}"'] & g f' x \ar[d, "\beta_{f'x}"] \\
g' f x  \ar[r, "g' ( \alpha_x )"'] &  g' f' x  \mathrlap{.} 
\end{tikzcd} 
\]
Finally, to see that the composition of morphism and the identity morphism are well-defined,  associative and unital, the proof is as in the 1-dimensional case. 
	\end{proof} 

% Alternative notation
%\item Given two 2-cells $[\alpha]\colon [\theta]\Rightarrow [\theta']\colon [\underline x:\underline A^{\underline X}|\ \varphi(\underline x)]\to  [\underline y:\underline B^{\underline Y}|\ \psi(\underline y)]$ and $[\beta]\colon [\rho]\Rightarrow [\rho']\colon [\underline y:\underline B^{\underline Y}|\ \psi(\underline y)]\to  [\underline z:\underline C^{\underline Z}|\ \chi(\underline z)]$, their horizontal composite is represented by the formula

\begin{lemma}\label{prop:limits-in-SynT}
	The $2$-category $\Syn(\TT)$ has all finite $2$-limits. 
\end{lemma}

\begin{proof}  It suffices to show that $\Syn(\TT)$ has a terminal object, binary products, equalisers, in the 2-categorical sense, and powers by $[1]$. In order to do this, we apply~\cref{thm:enhance-conical-with-powers}. When 
checking the required hypotheses below, we again restrict to formulas-in-context where the context has a single variable in order to improve readability.

Let us begin by constructing powers. Let $\{  x\colon   X|\ A(  x)\}$ be a formula-in-context and consider
%\[
%\begin{tikzcd}
% \{  u\colon   X^{[1]}|\ A^{[1]}(  u)\} \ar[r, bend left = 20, "\dom"] \ar[r, bend right = 20, "\cod"'] & 
% 	\{ x \co X \ |  \ A(x) \} 
%	\end{tikzcd}
% \]
 \begin{center}
 	\begin{tikzpicture}[baseline=(current  bounding  box.south), scale=2]
 		
 		\node (a0) at (0,0.9) {$\{  u\colon   X^{[1]}|\ A^{[1]}(  u)\}$};
 		\node (ab0) at (1,0.9) {${\scriptsize \Downarrow {[\iota]}}$};
 		\node (b0) at (2,0.9) {$\{ x \co X \ |  \ A(x) \}  $};
 		
 		\path[font=\scriptsize]
 		
 		(a0) edge [bend left=22, ->] node [above] {$[\dom]$} (b0)
 		(a0) edge [bend right=22, ->] node [below] {$[\cod]$} (b0);
 	\end{tikzpicture}
 \end{center}
where $\dom(u,x) \defeq A^{[1]}(u)\land\dom(u) = x$, $\cod(u,x) \defeq A^{[1]}(u)\land\cod(u) = x$ and 
\[
\iota(u,v) \defeq A^{[1]}(u)\land u = v \mathrlap{.}
\] 
It is immediate to check that these are morphisms and a 2-cell in $\Syn(\TT)$. In order show that we have a power, we need to show that the functor
\[
\Syn(\TT)\big( \{ y \co Y \ |  B(y) \}, \{ u \co X^{[1]} \ | \ A^{[1]}(u) \} \big) \xrightarrow{[\iota] \circ (-)} 
\Syn(\TT)\big( \{ y \co Y \ |  B(y) \}, \{ x \co X  \ | \ A(x) \}\big) ^{[1]} 
\]
is an isomorphism. Let us describe the domain category more explicitly. An object  is a morphism $[f] \co \{ y \co Y \ | \ B(y) \} \to \{ u \co X^{[1]} \ | \ 
B^{[1]}(u) \}$ and its image is the 2-cell $[\iota]\circ [f]$ given by
%\[
%\begin{tikzcd}
% \{  y \co Y \} | \ B(y) \} \ar[r, bend left = 10, "\dom \circ f"] \ar[r, bend right = 10, "\cod \circ f"'] & 
% 	\{ x \co X \ |  \ A(x) \} 
%	\end{tikzcd}
% \]
 \begin{center}
 	\begin{tikzpicture}[baseline=(current  bounding  box.south), scale=2]
 		
 		\node (a0) at (0,0.9) {$\{  y \co Y  | \ B(y) \}$};
 		\node (ab0) at (1,0.9) {${\scriptsize \Downarrow {[\iota\circ f]}}$};
 		\node (b0) at (2,0.9) {$\{ x \co X \ |  \ A(x) \}  $};
 		
 		\path[font=\scriptsize]
 		
 		(a0) edge [bend left=22, ->] node [above] {$[\dom\circ f]$} (b0)
 		(a0) edge [bend right=22, ->] node [below] {$[\cod\circ f]$} (b0);
 	\end{tikzpicture}
 \end{center}
 with $(\dom \circ f)(y,x) = (\exists u) f(y,u) \land \dom(u,x)$, $(\cod \circ f)(y,x) = (\exists u) f(y,u) \land \cod(u,x)$ and 
 \[
 (\iota \circ f)(y,u) \equiv f(y,u) \mathrlap{.} 
 \]
 A morphism in the domain category is a 2-cell $[\alpha] \co [f] \Rightarrow [g]$
 and $[\iota] \circ [\alpha] \co [\iota] \circ [f] \to [\iota] \circ [g]$ is the commutative square of 2-cells below.
 \begin{center}
 	\begin{tikzpicture}[baseline=(current  bounding  box.south), scale=2]
 		
 		\node (a0) at (0,0.8) {$[\dom \circ f]$};
 		\node (b0) at (1.4,0.8) {$[\dom \circ g]$};
 		\node (c0) at (0,0) {$[\cod \circ f]$};
 		\node (d0) at (1.4,0) {$[\cod \circ g ] \mathrlap{.}$};
 		
 		\path[font=\scriptsize]
 		
 		(a0) edge [->] node [above] {$[\dom \circ \alpha]$} (b0)
 		(a0) edge [->] node [left] {$[\iota \circ f]$} (c0)
 		(b0) edge [->] node [right] {$[\iota \circ g]$} (d0)
 		(c0) edge [->] node [below] {$[\cod \circ \alpha]$} (d0);
 	\end{tikzpicture}	
 \end{center}
 Unfolding the definitions and using the functionality of the various morphisms, one obtains that the four
 components of $\iota \circ \alpha$ are of the form
 \[
 (y \co Y, u \co X^{[1]} ) \quad (\exists v\co X^{[1]\times [1]} ) \alpha(y,v) \land \pr_{\delta}(v) = u
 \]
 where $\delta \in \{ u, d, \ell, r \}$. 
 
 With these definitions in place, it is not hard to prove that $\iota \circ (-)$ is an isomorphism. We leave the verification that it is bijective on objects to the readers and instead check that it is full and faithful.
 For faithfulness, fix $[f]$ and $[g]$ and let $[\alpha] \co [f] \Rightarrow[ g]$ and $[\beta] \co [f] \Rightarrow [g]$ be
 such that $[\iota \circ f ]= [\iota \circ g]$. Unfolding the definitions and using what we just observed above, we obtain 
 \[
 (\exists v, v'\co X^{[1]\times[1]}) \alpha(y,v) \land \beta(y,v') \land \bigwedge_{\delta} \pr_{\delta}(v) = \pr_{\delta}(v')
 \]
 which implies $(\exists v\co X^{[1]\times[1]}) \alpha(y,v) \land \beta(y,v)$, giving  $[\alpha] = [\beta]$ as required.
 
 For fullness, given a morphism $([\alpha_1], [\alpha_2])\co [\iota \circ f]\to [\iota \circ f]$ in the arrow category, as depicted in
 \begin{center}
 	\begin{tikzpicture}[baseline=(current  bounding  box.south), scale=2]
 		
 		\node (a0) at (0,0.8) {$[\dom \circ f]$};
 		\node (b0) at (1.4,0.8) {$[\dom \circ g]$};
 		\node (c0) at (0,0) {$[\cod \circ f]$};
 		\node (d0) at (1.4,0) {$[\cod \circ g ]$};
 		
 		\path[font=\scriptsize]
 		
 		(a0) edge [->] node [above] {$[\alpha_1]$} (b0)
 		(a0) edge [->] node [left] {$[\iota \circ f]$} (c0)
 		(b0) edge [->] node [right] {$[\iota \circ g]$} (d0)
 		(c0) edge [->] node [below] {$[\alpha_2]$} (d0);
 	\end{tikzpicture}	
 \end{center}
 we define $[\alpha] \co [f] \Rightarrow[ g]$ by letting
 \[
 \alpha(y,v) \defeq f(y, \prl(v)) \land g(y, \prr(v)) \land \alpha_1(y, \pru(v)) \land \alpha_2(y, \prd(v))
 \]
It is clear that $[\iota \circ \alpha] = [\alpha_1, \alpha_2]$, so it remains to show that $\alpha$ is a 2-cell
in $\Syn(\TT)$. These are straightforward calculations, using the definitions, \cref{equ:union}, and \cref{equ:power-squares}.
 
Next, we show that the underlying category $\Syn(\TT)_0$ has a
terminal object, binary products, equalisers in the 1-categorical sense.
The terminal object is  $\{x\colon S^{0} \ |\ \top \}$.  For binary products, the product of $\{x\colon   X|\ A(  x)\}$ and $\{y\colon   Y|\ B(  y)\}$ is
\[
\{x\colon   X,y\colon Y |\ A(  x)\wedge B(  y)\} \mathrlap{.} 
\]
For morphisms $[f],[g]\colon \{  x\colon   X|\ A(  x)\}\to  \{   y\colon   Y|\ B(  y)\}$, their equaliser
is 
\[
 \{  x\colon   X|\ (\exists y \co Y) f(x,y)\wedge g(x,y)\} \mathrlap{.} 
 \]
The proof that these have the required 1-categorical universal properties proceeds as in the case of finite limit theories (see \eg \cite[D1.4.2]{Joh02:libro}) and hence it is omitted. 

To prove that these also satisfy the 2-categorical universal property, it is enough to show that such 1-dimensional limits are preserved by powers by $[1]$, but this is a direct consequence of the rules in \cref{sec:deduction-powers} asserting stability of conjunctions and existential quantification under powers.
\end{proof}

\begin{rmk}  \label{thm:pullback-synT}
In view of~\cref{thm:synTisoregular}, it is useful to have an explicit definition of the pullbacks in $\Syn(\TT)$.
Given $[f] \co \{ x \co X \ | \ A(x) \} \to \{ z \co Z \ | \ C(z) \}$ and $[g] \co \{ y \co Y \ | \ B(y) \} \to \{ z \co Z \ | \ C(z) \}$, the pullback
\[
\begin{tikzcd}
P \ar[r, "{[p]}"] \ar[d, "{[q]}"'] & \{ x \co X \ | \ A(x) \} \ar[d, " {[f]}"] \\ 
 \{ y \co Y \ | \ B(y) \} \ar[r,  " {[g]}"']  & \{ z \co Z \ | \ C(z) \} 
 \end{tikzcd}
\]
is given by letting
\[
P \defeq \{ x \co X, y \co Y \ | \ A(x) \land B(y) \land (\exists z \co Z) f(x,y) \land g(y,z) \}
\]
with $p(x,y,x')$ being the proposition $P(x,y)\land x = x'$, and $q(x,y,y')$ the proposition $P(x,y)\land y = y'$.
\end{rmk} 

The following lemma collects the stability properties of all our logic constructs, under powers by $[1]$; these will be essential in the proof of \cref{theorem:models=functors}.

\begin{lemma}\label{lemma:powers-in-Syn(T)}
	There are the following isomorphisms in $\Syn(\TT)$:
	\begin{enumerate}[itemsep=0.07cm]
		\item $\{ x \co S \  | \ \top \}^{\bb c}\cong \{ x \co S^{\bb c} \  | \ \top \}$ for any $\bb c\in\FinCat$,
		\item $\{ x \co S \  | \ \top \}^{f}\cong [y\cdot f=x]\colon\{ y \co S^{\bb d} \  | \ \top \}\to \{ x \co S^{\bb c} \  | \ \top \} $ for any $f\co \bb c\to \bb d\in\FinCat$,
		\item $\{ \bar{x} \co \bar{X} \  | \ s(x) = t(x) \}^{[1]}\cong \{ \bar{u} \co \bar{X}^{[1]} \  | \  s^{[1]}(\bar{u} ) =t^{[1]}(\bar{u} )  \}$,
		\item $\{ \bar{x} \co \bar{X}  \  | \ R(s_1, \ldots, s_n) \}^{[1]}\cong \{ \bar{u} \co \bar{X}^{[1]}  \  | \ R^{[1]} \big( s_1^{[1](\bar{u} )}, \ldots, s^{[1]}_n \big) (\bar{u})\}$;
		\item $\{ \bar{x} \co \bar{X}   \  | \ A(\bar{x}) \wedge B(\bar{x} ) \}^{[1]}\cong \{ \bar{u} \co \bar{X}^{[1]}   \  | \ A^{[1]}(\bar{u} ) \wedge B^{[1]}(\bar{u} ) \}$;
		\item $\{ \bar{x} \co \bar{X}  \  | \ (\exists y \co Y) B(\bar{x},y) \}^{[1]}\cong \{ \bar{u} \co \bar{X}^{[1]}  \  | \ (\exists v \co Y^{[1]}) B^{[1]}(\bar{u} , v) \}$;
		\item $\{ \bar{x} \co \bar{X} \  | \ A(\bar{x}) \wedge(\exists y \co Y) B(\bar{x},y) \}\cong \{  \bar{x} \co \bar{X}   | \ (\exists y \co Y) A(\bar{x}) \wedge B(\bar{x}, y) \}$.
	\end{enumerate}
\end{lemma}

\begin{proof} Parts (iii)-(vii) follow easily by how we constructed the powers explicitly and from the deduction rules of \cref{sec:deduction-powers}; thus they are left to the readers. We shall focus on (i) and (ii) instead. By \cref{prop:limits-in-SynT} we already know that (i) holds for $\bb c=[1]$, for $\bb c=0$ the empty category, and (trivially) for $\bb c=[0]$ the terminal category. Similarly, by how powers by $[1]$ are constructed, (ii) holds for the two functors $\sigma_0,\sigma_1\co [0]\to[1]$ inducing the terms $\dom$ and $\cod$. Since the closure of $[1]$ under finite colimits in $\FinCat$ is the whole category, it is enough to prove that for any coequaliser $q\co\bb b\to \bb c$ of a pair $f,g\co \bb a\to \bb b$ in $\FinCat$ the following 
\begin{center}
	
	\begin{tikzpicture}[baseline=(current  bounding  box.south), scale=2, hookarrow/.style={{Hooks[right]}->}]
		
		\node (b) at (-1.8,0) {$\{ x \co S^{\bb c} \  | \ \top \}$};
		\node (k) at (0,0) {$\{ y \co S^{\bb b} \  | \ \top \} $};
		\node (a) at (1.8,0) {$\{ z \co S^{\bb a} \  | \ \top \}$};

		\path[font=\scriptsize]

		([yshift=1.7pt]k.east) edge [->] node [above] {$[y\cdot f=z]$} ([yshift=1.7pt]a.west)
		([yshift=-1.7pt]k.east) edge [->] node [below] {$[y\cdot g=z]$} ([yshift=-1.7pt]a.west)
		
		(b) edge [->] node [above] {$[x\cdot q=y]$} (k);
	\end{tikzpicture}
\end{center}	
is an equaliser in $\Syn(\TT)$, and that for any $\bb a,\bb b\in\FinCat$ we have
$$ \{ x \co S^{\bb a+\bb b} \  | \ \top \}\cong \{ x \co S^{\bb a} \  | \ \top \}\times \{ x \co S^{\bb b} \  | \ \top \} $$ 
with projections induced by restricting along the inclusions $\iota_{\bb a}\co\bb a \to \bb a+\bb b$ and $\iota_{\bb b}\co\bb b\to \bb a+\bb b$. This now follows easily from the explicit construction of equalisers and products in $\Syn(\TT)$ and, respectively, from Rule~\ref{equ:coeq} and Rule~\ref{equ:union} (applied to the case where $\bb b=0$).
\end{proof}

In order to state \cref{thm:faithful-definable,thm:ff-kernel-definable,prop:swe-synT} below, for a morphism $[f]\colon \{   x\colon   X|\ A(  x)\}\to  \{   y\colon   Y|\ B(  y)\}$  in $\Syn(\TT)$ we define the following judgements:
\begin{align*} 
& \mathsf{is\text{-}mono}(f) \defeq  \\
 & \qquad (x, x' \co X, y \co Y)\  f(x,y), f(x', y) \vdash x' = x''  \mathrlap{,}   \\
& \mathsf{is\text{-}faithful}(f)  \defeq  \\
& \qquad (u, u' \co X^{[1]}, v \co Y^{[1]})  \ \dom(u) = \dom(u'), \cod(u) = \cod(u'), f^{[1]}(u,v), f^{[1]}(u',v)  \vdash u = u'  \mathrlap{,}  \\
& \mathsf{is\text{-}full\text{-}identities}(f) \defeq   \\
& \qquad (x', x'' \co X, y \co Y)   \ f(x',y), f(x'', y)  \vdash (\exists u \co X^{[1]}) f^{[1]}(u,\id_y) \land \dom(u) = x' \land \cod(u) = x'' \mathrlap{.} 
\end{align*}
For a proposition $(x \co X, y \co Y) B(x,y) \co \mathsf{prop}$, we have a composite morphism 
\[
m_B \co \{x \co X, y \co Y \ | \ B(x,y) \} \rightarrowtail \{x \co X \ |  \ \top\}  \times \{y  \co Y  \ | \  \top\}  \to  \{y  \co Y  \ | \  \top\}
\] 
in $\Syn(\TT)$. Then, the judgement $\mathcal{J}_{\mathsf{mono}}(B)$ of \cref{equ:exists-formation-mono} is equivalent to
the judgement $\mathsf{is\text{-}mono}(m_B)$. A similar equivalence holds for the judgements in~\cref{equ:exists-formation-eqfib}.

\begin{lemma} \label{thm:mono-definable}
Let $[f]\colon \{   x\colon   X|\ A(  x)\}\to  \{   y\colon   Y|\ B(  y)\}$ be a morphism  in $\Syn(\TT)$. Then 
	the following conditions are equivalent:
\begin{enumerate}
		\item $[f]$ is a monomorphism,
		\item the judgement $\mathsf{is\text{-}mono}(f)$ is derivable.
		\end{enumerate}
\end{lemma} 

\begin{proof} 
	The map $[f]$ is a monomorphism if an only if in the two projections in the pullback of $[f]$ along itself are equal. By \cref{thm:pullback-synT} these two projections are  represented by
	$$ p(x',x'',x)\defeq A(x') \land A(x'') \land x'=x \land (\exists y \co Z) f(x',y) \land f(x'',y) $$
	$q(x',x'',x)$ defined as above but with $x''=x$ instead of $x'=x$. It is easy to see that these two maps are the same in $\Syn(\TT)$ if and only if $f(x',y), f(x'',y)$ entails $x'=x''$, giving the desired equivalence.
%For the implication (i) $\Rightarrow$ (ii), assume that $[f]$ is a monomorphism. Assume the hypotheses of $\mathsf{is\text{-}mono}(f)$, \ie let
%$x', x'' \co X, y \co Y$ and assume $f(x',y)$, $f(x'', y)$. We claim that $x' = x''$ is derivable. Consider the maps $g', g'' \co \{ \_ \co X \ | \ \top \} \to \{   x\colon   X|\ A(  x)\}$
%defined by letting them be the formulas $x = x'$ and $x = x''$. By the assumption, $f \circ g' = f \circ g''$ (since both $f(x',y)$ and $f(x'',y)$ are assumed to hold).
%Since $f$ is a monomorphism, $g' = g''$, which gives the required conclusion. The implication (ii) $\Rightarrow$ (i) is a simple calculation, which we leave to readers.
\end{proof}

\begin{lemma} \label{thm:faithful-definable}
Let $[f]\colon \{   x\colon   X|\ A(  x)\}\to  \{   y\colon   Y|\ B(  y)\}$ be a morphism  in $\Syn(\TT)$. Then 
	the following conditions are equivalent:
\begin{enumerate}
		\item $[f]$ is faithful,
		\item the judgement $\mathsf{is\text{-}faithful}(f)$ is derivable.
		\end{enumerate}
\end{lemma} 

\begin{proof} We know, as a general 2-categorical fact, that $[f]$ is faithful if and only if the morphism 
	$$(\dom, f^{[1]}, \cod) \co \{   u\colon   X|\ A^{[1]}(  u)\} \longrightarrow \{   x\colon   X|\ \top\} \times \{   v\colon   Y^{[1]}|\ B^{[1]}(  u)\} \times \{   x'\colon   X|\ \top\}$$
	 is a monomorphism. The claim now follows easily from \cref{thm:mono-definable}.
\end{proof}

\begin{lemma} \label{thm:ff-kernel-definable}
	Let $[f]\colon \{   x\colon   X|\ A(  x)\}\to  \{   y\colon   Y|\ B(  y)\}$ be a morphism  in $\Syn(\TT)$. Then 
	the following conditions are equivalent:
	\begin{enumerate}
		\item $[f]$ is an \ffk-morphism,
		\item the judgements $\mathsf{is\text{-}faithful}(f)$ and $\mathsf{is\text{-}full\text{-}identities}(f)$ 
		are derivable.
	\end{enumerate}
\end{lemma}

\begin{proof} The equivalence between $[f]$ being faithful and the judgement expressing 
faithfulness of $f$ in $\TT$ is \cref{thm:faithful-definable}, so we need to prove the equivalence
between fullness on identities for $[f]$ and the corresponding judgement in $\TT$.

For one implication, assume that
\begin{equation}
	\label{equ:variant-id-lift}
	(x', x'' \co X, y \co Y) \ f(x',y), f(x'', y) \vdash (\exists u \co X^{[1]}) f^{[1]}(u,\id_y) \land
	\dom(u) = x' \land \cod(u) = x'' 
\end{equation} 
 is derivable. Let $g', g'' \co \{ z \co Z \ | \ C(z) \} \to \{ x \co X \ | \ A(x) \}$ be such that
 $fg' = fg''$. We claim that there is a 2-cell $\alpha \co g' \Rightarrow g''$ such that
 $f \circ \alpha = \id_{fg'}$. For $z \co Z, u \co X^{[1]}$, we define
 \[
 \alpha(z,u) \defeq 
 g'(z, \dom(u)) \land g''(z, \cod(u)) \land 
 (\exists y \co Y) \big( (f \circ g')(z,y) \land f^{[1]}(u, \id_y) \big) \mathrlap{.}
  \]
  The existential quantifier can be formed since $f \circ g'$ is functional. The verification that
  $\alpha$ is a 2-cell is immediate. It remains to check that $f \circ \alpha = \id_{fg'}$. Thus,
  we need to show that, for $z \co Z$ and $v \co Y^{[1]}$, the propositions
  \[
  (f \circ \alpha)(z,v) \defeq (\exists u \co X^{[1]}) \alpha(z,u) \land f^{[1]}(u,v) 
  \]
  and
  \[
  (\id_{f \circ g'})(z,v) \defeq (\exists y \co Y) (f \circ g')(z,y) \land v = \id_y
  \]
  are equivalent. Both implications are easy, noting that, for 
  $u \co X^{[1]}$, $f^{[1]}(u,v)$ and $f^{[1]}(u, \id_y)$
  imply $v = \id_y$. 
  
  For the converse, we assume that $[f]$ is full on identities and show
  that the judgement~\eqref{equ:variant-id-lift} is derivable. Let $[p(x',x'',x)]$ and $[q(x',x'',x)]$ the two projections in the kernel pair of $f$ (see \eg the proof of \cref{thm:mono-definable}), then by definition $[f\circ p]=[f\circ q]$ and so, by fullness on identities there is $[\alpha]\co [p]\Rightarrow[q]$ such that $[f\circ \alpha]=\id_{[f\circ p]}$. It is now easy to see that the functionality of $\alpha$ says exactly that~\eqref{equ:variant-id-lift} holds.
%  For the converse, we assume that $[f]$ has the identity-lifting property and show
%  that the judgement
%  \begin{equation}
%  \label{equ:variant-id-lift}
%(x', x'' \co X, y \co Y) \ f(x',y'), f(x'', y''), y' = y'' \vdash (\exists u \co X^{[1]}) f^{[1]}(u,\id(y')) \land
% \dom(u) = x' \land \dom(u) = x'' 
% \end{equation} 
%is derivable. Let $x', x'' \co X$, $y', y'' \co Y$ and assume the hypotheses of the judgement.
%Consider the terminal object $1 \defeq \{ z \co 1 \ | \ \top \}$ of $\Syn(\TT)$ and the 
%morphisms $g', g'' \co 1 \to \{ x \co X \ | \ A(x) \}$ defined by
%\[
%(z \co 1, x \co X) \ g'(z, x) \defeq x = x' \mathrlap{,} \qquad
%(z \co 1, x \co X) \ g''(z, x) \defeq x = x'' \mathrlap{.} 
%\]
%We claim that $f \circ g' = f \circ g''$, \ie that, for $z \co 1$ and $y \co Y$,
%$(f \circ g')(z,y) \Leftrightarrow (f \circ g'')(z,y)$. By the definitions above,
%this amounts to showing that $f(x',y) \Leftrightarrow f(x'',y)$, which follows easily
%from the hypotheses of~\eqref{equ:variant-id-lift}. By the assumption that $f$ has the identity-lifting property, we obtain that there
%exists a 2-cell $\alpha \co g' \Rightarrow g''$ such that $f \circ \alpha = \id_{f \circ g'}$.
%This gives the required conclusion in~\eqref{equ:variant-id-lift}.
\end{proof}

\begin{lemma}\label{prop:swe-synT}
	Let $\TT$ be a isoregular theory. Let $[f]\colon \{   x\colon   X|\ A(  x)\}\to  \{   y\colon   Y|\ B(  y)\}$ be a morphism  in $\Syn(T)$.
	Then the following conditions are equivalent:
	\begin{enumerate}
		\item the morphism $[f]$ is a fully faithful regular epimorphism,
		\item the following judgements are derivable:
			\begin{gather*} 
			\mathsf{is\text{-}faithful}(f) \mathrlap{,}  \\
			 \mathsf{is\text{-}full\text{-}identities}(f) \mathrlap{,}  \\
			(y \co Y) \ B(  y) \vdash (\exists x \co X) f(x,y)\mathrlap{,} 
		\end{gather*}
		\item the following  judgements are derivable:
			\begin{gather*} 
	 		\mathsf{is\text{-}faithful}(f)   \mathrlap{,}  \\
			(y \co Y) \ B(  y) \vdash (\exists x \co X) f(  x,  y) \mathrlap{.}
		\end{gather*}
	\end{enumerate}
	In this case, there exists an isomorphism 
	\[
	\{   y\colon   Y|\ B(  y)\}\cong \{   y\colon   Y|\ (\exists x \co X)  f(  x,  y)\} \mathrlap{.}
	\]
\end{lemma}

\begin{proof} For implication (i) $\Rightarrow$ (ii), assume that $[f]$ is a fully faithful regular epimorphism. Being fully faithful, $[f]$ is an \ffk-morphism and therefore the first two judgements hold by
\cref{thm:ff-kernel-definable}. Being a regular epimorphism, $[f]$ is the coequaliser of its kernel pair. Such coequaliser (arguing as in the 1-dimensional case) is given by $\{   y\colon   Y|\ (\exists x \co X)  f(  x,  y)\}$, where this is a well-defined formula-in-context since  $[f]$ is an \ffk-morphism. Thus the final isomorphism in the statement follows, and as a consequence we obtain the judgement $B(y) \vdash (\exists x \co X) f(x,y)$.
The implication (ii) $\Rightarrow$ (iii) is trivial.
For the implication (iii) $\Rightarrow$ (i), the assumptions express that $[f]$ is
a regular epimorphism and is faithful. Using the  \cref{equ:power-squares}, we obtain $B^{[1]}(v) \vdash (\exists u \co X^{[1]}) f^{[1]}(u,v)$, which expresses fullness.
%The isomorphism follows from the equivalence $B(y) \Leftrightarrow (\exists x \co X) f(x,y)$, for $y \co Y$.
\end{proof} 

\begin{prop} \label{thm:synTisoregular}
	The $2$-category $\Syn(\TT)$ is isoregular.
\end{prop}

\begin{proof} By \cref{prop:limits-in-SynT}, we already know that $\Syn(\TT)$ has finite 2-limits. 
Next, we need to show that a morphism $[f]\colon \{   x\colon   X|\ A(  x)\}\to  \{   y\colon   Y|\ B(  y)\}$ with
a fully faithful kernel pair has a fully faithful coequaliser. 
For this, define 
\[
\{ y \co Y \ \ | (\exists x \co X) f(x,y) \}
\]
This is a well-defined formula-in-context since  $[f]$ is an \ffk-morphism by \cref{fact-morph}
and therefore we can apply \cref{thm:ff-kernel-definable}. The fact that the induced map into this object is a fully faithful regular epimorphism follows from \cref{prop:swe-synT}, and the verification that this is
indeed the coequaliser of the kernel pair of $[f]$ is done as in the ordinary setting (in fact, since $\Syn(\TT)$ has powers by $[1]$ it suffices to prove the 1-dimensional universal property).

Finally, we need to show that fully faithful regular epimorphisms are stable under pullback.
This follows from~\cref{prop:swe-synT} recalling the construction of pullbacks in~$\Syn(\TT)$
in \cref{thm:pullback-synT}: pullbacks are constructed using conjunction, and existential
quantification commutes with conjunction by the Frobenius Rule~\ref{equ:frobenius}.
\end{proof}

\subsection{Functorial semantics} The next result is fundamental for our development, as it provides a counterpart of the cornerstones of functorial semantics in the 1-categorical setting by establishing the connection between 2-categories of models
of isoregular theories and 2-categories of isoregular 2-functors.

\begin{theo}\label{theorem:models=functors}
	For any isoregular 2-category $\K$ and any isoregular theory $\TT$ we have an equivalence 
	$$ \Mod(\TT,\K)\simeq \IsoReg(\Syn(\TT),\K)$$
	of $2$-categories.
\end{theo}
\begin{proof} 
	We begin by constructing a 2-functor 
	$$\Sigma\colon \Mod(\TT,\K)\to \IsoReg(\Syn(\TT),\K)$$
	that we then show is an equivalence.
	Given $M\in \Mod(\TT,\K)$ we define $\Sigma M\colon \Syn(\TT)\to\K$ by:\begin{itemize}
		\item for any $\{\bar x\co\bar{X}|\ A(\bar x)\}\in \Syn(\TT)$ we set $$\Sigma M(\{\bar x\co\bar{X}|\ A(  \bar x)\})\defeq\br{A};$$
		\item for any $[f]\colon \{ \bar x\co\bar{X}|\ A(  \bar x)\}\to  \{  \bar y\co\bar Y |\ B(\bar y)\}$ in $\Syn(\TT)$ the map $\Sigma M([f])\colon\br{A}\to\br{B}$ is the composite $\pi_2\circ \pi_1^{-1}$ depicted below
			\begin{center}
			\begin{tikzpicture}[baseline=(current  bounding  box.south), scale=2]
				
				\node (00) at (0,0) {$\br{f}$};
				\node (10) at (0.9,0) {$\br{A\wedge B}$};
				\node (20) at (2,0) {$\br{\bar{X}}\!\times \br{\bar Y}$};
				
				\node (11) at (0.9,0.8) {$\br{A}$};
				\node (21) at (2,0.8) {$\br{\bar{X}}$};
				
				\node (1-1) at (0.9,-0.8) {$\br{B}$};
				\node (2-1) at (2,-0.8) {$\br{\bar Y}$};
				
				\path[font=\scriptsize]
				
				(00) edge [>->] node [above] {} (10)
				(00) edge [->] node [above] {$\pi_1$} (11)
				(00) edge [->] node [below] {$\pi_2$} (1-1)
				
				(10) edge [>->] node [below] {} (20)
				(10) edge [->] node [above] {} (11)
				(10) edge [->] node [right] {} (1-1)
				
				(11) edge [>->] node [above] {} (21)
				(1-1) edge [>->] node [above] {} (2-1)
				(20) edge [->] node [below] {} (21)
				(20) edge [->] node [left] {} (2-1);
			\end{tikzpicture}	
		\end{center}
		where $\pi_1$ is invertible since $f$ is functional.
		\item for any 2-cell $[\alpha]\colon [f]\Rightarrow [f']\colon  \{ \bar x\co \bar{X}|\ A(  \bar x)\}\to  \{  \bar y\co\bar Y|\ B(  \bar y)\}$ in $\Syn(\TT)$ consider its transpose $[\bar\alpha]\colon  \{ \bar x\co \bar{X}|\ A( \bar x)\}\to  \{ \bar u\colon   \bar Y^{[1]}|\ B^{[1]}( \bar u)\} $; then we have the following commutative diagram in $\K$.
		\begin{center}
			\begin{tikzpicture}[baseline=(current  bounding  box.south), scale=2]
				
				\node (00) at (-0.2,0) {$\br{A}$};
				\node (10) at (1.1,0) {$\br{B^{[1]}}$};
				\node (20) at (1.8,0) {$\br{B}^{[1]}$};
				
				\node (21) at (1.8,0.6) {$\br{B}$};
				
				\node (2-1) at (1.8,-0.6) {$\br{B}$};
				
				\path[font=\scriptsize]
				
				(00) edge [>->] node [above] {$\Sigma M([\bar\alpha])$} (10)
				(00) edge [bend left=25,->] node [above] {$\Sigma M([f])$} (21)
				(00) edge [bend right=25,->] node [below] {$\Sigma M([f'])$} (2-1)
				
				(10) edge [->] node [above] {$\cong$} (20)
				
				(20) edge [->] node [right] {$\dom$} (21)
				(20) edge [->] node [right] {$\cod$} (2-1);
			\end{tikzpicture}	
		\end{center}
		Commutativity of the triangles follows from $\alpha(x,u)\vdash f(x,\tx{dom}(u))\wedge f'(x,\tx{cod}(u))$. The horizontal composite is exactly the data of a 2-cell $\Sigma M([f])\Rightarrow\Sigma M([f'])$ in $\K$ which we define to be $\Sigma M([\alpha])$.
	\end{itemize}
	The fact that $\Sigma M$ is is well-defined and is a 2-functor is done as in the 1-dimensional setting, and we leave the details to the reader. To show that it is isoregular we proceed by steps. \begin{enumerate}
		\item {\em $\Sigma M$ preserves powers by $[1]$}. Fix an object $\{\bar x\colon \bar X|\ A( \bar x)\}$ of $\Syn(\TT)$; then the cylinder expressing the power of such object by $[1]$ is given by the 2-cell
		\begin{center}
			\begin{tikzpicture}[baseline=(current  bounding  box.south), scale=2]
				
				\node (a0) at (0,1) {$\{\bar u\colon   \bar X^{[1]}|\ A^{[1]}(  \bar u)\}$};
				\node (b0) at (2,1) {$\{\bar x\colon  \bar X|\ A( \bar x)\}$.};
				\node (e'0) at (1.1,1) {$\Downarrow [\iota]$};

				\path[font=\scriptsize]
				
				(a0) edge [bend left, ->] node [above] {$[\dom]$} (b0)
				(a0) edge [bend right, ->] node [below] {$[\cod]$} (b0);
			\end{tikzpicture}	
		\end{center}
		defined in the proof of \cref{prop:limits-in-SynT}. This is sent by $\Sigma M$ to
		\begin{center}
			\begin{tikzpicture}[baseline=(current  bounding  box.south), scale=2]
				
				\node (a0) at (0,1) {$\br{A^{[1]}}$};
				\node (b0) at (2,1) {$\br{A}$};
				\node (e'0) at (1.1,1) {$\Downarrow \Sigma M([\iota])$};

				\path[font=\scriptsize]
				
				(a0) edge [bend left, ->] node [above] {$\Sigma M([\dom])$} (b0)
				(a0) edge [bend right, ->] node [below] {$\Sigma M([\cod])$} (b0);
			\end{tikzpicture}	
		\end{center}
		where $\Sigma M([\iota])$ is, by definition of $\Sigma$, the transpose of the composite
		$$ \br{A^{[1]}}\xrightarrow{\Sigma M ([\bar\iota])} \br{A^{[1]}}\xrightarrow{\cong} \br{A}^{[1]}. $$
		But $[\bar\iota]=1_{\{u|\ A^{[1]}\}}$, so $\Sigma M ([\bar\iota])=1_{\br{A^{[1]}}}$ and $\Sigma M([\iota])$ is the transpose of the canonical isomorphism defining the power in $\K$; hence $\Sigma M([\iota])$ is a limiting cylinder.
		
		\item {\em $\Sigma M$ preserves binary products}. By \cref{prop:limits-in-SynT}, the product of $\{\bar x\colon \bar X|\ A(\bar x)\}$ with $\{\bar y\colon \bar Y|\ B(\bar y)\}$ in $\Syn(T)$ is $ \{\bar x\colon  \bar X,\bar y\colon \bar Y |\ A(\bar x)\wedge B(  \bar y)\} $. This is sent by $\Sigma M$ to the pullback
		\begin{center}
			\begin{tikzpicture}[baseline=(current  bounding  box.south), scale=2]
				
				\node (a0) at (0,0.8) {$\br{A\wedge B}$};
				\node (a0') at (0.2,0.6) {$\lrcorner$};
				\node (b0) at (1.5,0.8) {$\br{\bar A}\times\br{\bar Y}$};
				\node (c0) at (0,0) {$\br{X}\times \br{B}$};
				\node (d0) at (1.5,0) {$\br{\bar X}\times\br{\bar Y}$};
				
				\path[font=\scriptsize]
				
				(a0) edge [>->] node [above] {} (b0)
				(a0) edge [>->] node [left] {} (c0)
				(b0) edge [>->] node [right] {$m_A\times 1$} (d0)
				(c0) edge [>->] node [below] {$1\times m_B$} (d0);
			\end{tikzpicture}
		\end{center}
		which is isomorphic to $\br{A}\times\br{B}=\Sigma M(\{\bar X|\ A\})\times \Sigma M(\{\bar Y|\ B\})$.
		
		\item {\em $\Sigma M$ preserves equalisers}. By \cref{prop:limits-in-SynT}, the equaliser of $[f],[g]\colon \{  \bar x\colon \bar X|\ A(\bar x)\}\to  \{ \bar y\colon \bar Y|\ B(\bar y)\}$ is $ \{\bar x\colon \bar X|\ \exists \bar y\ f(\bar x,\bar y)\wedge g(\bar x,\bar y)\}$. Consider the diagram below.
		\begin{center}
			\begin{tikzpicture}[baseline=(current  bounding  box.south), scale=2, on top/.style={preaction={draw=white,-,line width=#1}}, on top/.default=6pt]
				
				\node (a0) at (0,1.2) {$\br{f\wedge g}$};
				\node (b0) at (1.3,1.2) {$\br{g}$};
				\node (c0) at (0,0) {$\br{\exists \bar y\ f\wedge g}$};
				\node (d0) at (1.3,0) {$\br{A}$};
				%\node (e0) at (0.15,0.85) {$\lrcorner$};
				
				\node (a0') at (0.7,0.6) {$\br{f}$};
				\node (c0') at (0.7,-0.6) {$\br{A}$};
				\node (d0') at (3,-0.1) {$\br{\bar X}\times\br{\bar Y}$};
				%\node (e0') at (0.9,0.35) {$\lrcorner$};
				
				\path[font=\scriptsize]

				(a0) edge [->] node [above] {} (b0)
				(a0) edge [->] node [left] {$P_1$} (c0)
				(b0) edge [->] node [right] {} (d0)
				(c0) edge [>->] node [below] {} (d0)
				
				(a0) edge [->] node [above] {} (a0')
				(b0) edge [bend left=10, >->] node [above] {} (d0')
				(c0) edge [>->] node [right] {} (c0')
				(d0) edge [bend left=5, ->] node [below] {$(m_B,\Sigma M([g]))$} (d0')
				
				(a0') edge [bend left=10, >->, on top] node [above] {} (d0')
				(a0') edge [->, on top] node [above] {} (c0')
				(c0') edge [bend right=10, ->] node [below] {$\ \ \ (m_A,\Sigma M([f]))$} (d0');
			\end{tikzpicture}	
		\end{center}
		Where the top square is a pullback and the vertical maps are all isomorphisms ($P_1$ since it is obtained by factorising a monomorphism). It follows that also the bottom face is a pullback. Therefore, by the standard argument involving pullbacks and products, $\Sigma M(\exists \bar y\ f\wedge g)=\br{\exists y\ f\wedge g}$ is isomorphic to the equaliser of $\Sigma M([f])$ and $\Sigma M([g])$.
		
		\item {\em $\Sigma M$ preserves fully faithful regular epimorphisms}. By \cref{prop:swe-synT} any fully faithful regular epimorphism in $\Syn(\TT)$ can be expressed as
		$$[f]\colon \{\bar x\colon \bar X|\ A(\bar x)\}\to  \{\bar y\colon   Y|\ \exists \bar x:X\ f(\bar x,\bar y)\}.$$
		By definition of $\Sigma M$ the image of $[f]$ is the composite $\pi_2\circ \pi_1^{-1}$ identified below.
		\begin{center}
			\begin{tikzpicture}[baseline=(current  bounding  box.south), scale=2]
				
				\node (10) at (0.9,0) {$\br{f}$};
				\node (20) at (2,0) {$\br{\bar X}\!\times \br{\bar Y}$};
				
				\node (11) at (0.9,0.8) {$\br{A}$};
				\node (21) at (2,0.8) {$\br{\bar X}$};
				
				\node (1-1) at (0.9,-0.8) {$\br{\exists x\ f(x,y)}$};
				\node (2-1) at (2,-0.8) {$\br{\bar Y}$};
				
				\path[font=\scriptsize]
				
				(10) edge [->] node [left] {$\pi_1$} (11)
				(10) edge [->] node [right] {$\cong$} (11)
				(10) edge [right hook ->>] node [left] {$\pi_2$} (1-1)
				(10) edge [>->] node [below] {} (20)
				
				(11) edge [>->] node [above] {} (21)
				(1-1) edge [>->] node [above] {} (2-1)
				(20) edge [->] node [below] {} (21)
				(20) edge [->] node [left] {} (2-1);
			\end{tikzpicture}	
		\end{center}
		Now observe that $\pi_2$ is a fully faithful regular epimorphism by how existential quantification is interpreted in $\K$. Thus $\Sigma M([f])$ is a fully faithful regular epimorphism too.
	\end{enumerate}
	This defines $\Sigma$ on objects; we need to show that this assignment extends on 1-cells and 2-cells.
	
	Consider a morphism $p\colon M\to N$ in $\Mod(\TT,\K)$, by definition this is just a morphism of $\LL$-structures and thus is determined by maps 
	$$ p_S\colon \br{S}_M\to \br{S}_N $$
	in $\K$ for any sort $S$, which preserve the interpretation of function and relation symbols. We define the components of the 2-natural transformation $\Sigma p\colon \Sigma M\Rightarrow\Sigma N$ by
	$$ (\Sigma p)_{\{\bar{X}|A\}}\defeq p_A\colon \Sigma M(\{\bar{X}|A\})\to \Sigma N(\{\bar{X}|A\}) $$
	for any $\{\bar{X}|A\}\in\Syn(\TT)$, where $p_A$ is given by \cref{lemma:morphisms.respect.formulas}. By a standard 2-dimensional argument, this is 2-natural if and only if it is natural in the 1-dimensional sense and for any object $\{\bar{X}|A\}$ the square
	\begin{center}
		\begin{tikzpicture}[baseline=(current  bounding  box.south), scale=2]
			
			\node (a0) at (0,0.8) {$\Sigma M(\{\bar{X}|A\})^{[1]}$};
			\node (b0) at (1.6,0.8) {$\Sigma M(\{\bar{X}|A\}^{[1]})$};
			\node (c0) at (0,0) {$\Sigma N(\{\bar{X}|A\})^{[1]}$};
			\node (d0) at (1.6,0) {$\Sigma N(\{\bar{X}|A\}^{[1]})$};
			
			\path[font=\scriptsize]
			
			(a0) edge [->] node [above] {$\cong$} (b0)
			(a0) edge [->] node [left] {$(\Sigma p_{\{\bar{X}|A\}})^{[1]}$} (c0)
			(b0) edge [->] node [right] {$\Sigma p_{\{\bar{X}|A\}^{[1]}}$} (d0)
			(c0) edge [->] node [below] {$\cong$} (d0);
		\end{tikzpicture}	
	\end{center}
	commutes. The commutativity of such square follows from the recursive definition of $p_{A}$ in~\cref{lemma:morphisms.respect.formulas}. As for ordinary naturality, consider a morphism $[f]\colon \{   \bar{X}| A\}\to  \{  \bar Y| B\}$ in $\Syn(\TT)$, then (using the definition of $\Sigma$) we need to prove that the front face of the diagram below commutes.
	\begin{center}
		\begin{tikzpicture}[baseline=(current  bounding  box.south), scale=2, on top/.style={preaction={draw=white,-,line width=#1}}, on top/.default=6pt]
			
			\node (a1) at (0,1) {$\br{A}_M$};
			\node (b1) at (1.5,1) {$\br{f}_M$};
			\node (c1) at (3.2,1) {$\br{B}_M$};
			\node (a0) at (0,0) {$\br{A}_N$};
			\node (b0) at (1.5,0) {$\br{f}_N$};
			\node (c0) at (3.2,0) {$\br{B}_N$};
			
			\node (a1') at (0.7,1.7) {$\br{\bar{X}}_M$};
			\node (b1') at (2.2,1.7) {$\br{\bar{X}}_M\!\times\br{\bar{Y}}_M$};
			\node (c1') at (3.9,1.7) {$\br{\bar{Y}}_M$};
			\node (a0') at (0.7,0.7) {$\br{\bar{X}}_N$};
			\node (b0') at (2.2,0.7) {$\br{\bar{X}}_N\!\times\br{\bar{Y}}_N$};
			\node (c0') at (3.9,0.7) {$\br{\bar{Y}}_N$};
			
			\path[font=\scriptsize]
			
			(a1') edge [<-] node [above] {} (b1')
			(b1') edge [->] node [above] {} (c1')
			(a0') edge [->] node [below] {} (b0')
			(b0') edge [->] node [below] {} (c0')
			(a1') edge [<-] node [left] {$p_{\bar{X}}$} (a0')
			(b1') edge [->] node [right] {$p_{\bar{X}}\!\times p_{\bar{Y}}$} (b0')
			(c1') edge [->] node [right] {$p_{\bar{Y}}$} (c0')
			
			(a1) edge [->, on top] node [above] {$\ \ \ \ \ \ \ \ \cong$} (b1)
			(b1) edge [->, on top] node [above] {} (c1)
			(a0) edge [->] node [below] {$\cong$} (b0)
			(b0) edge [->] node [below] {} (c0)
			(a1) edge [->] node [left] {$p_{A}$} (a0)
			(b1) edge [->, on top] node [left] {$p_{f}$} (b0)
			(c1) edge [->, on top] node [left] {$h_{B}$} (c0)
			
			(a1) edge [>->] node [above] {} (a1')
			(b1) edge [>->] node [above] {} (b1')
			(c1) edge [>->] node [above] {} (c1')
			(a0) edge [>->] node [above] {} (a0')
			(b0) edge [>->] node [above] {} (b0')
			(c0) edge [>->] node [above] {} (c0');
		\end{tikzpicture}	
	\end{center}
	The three vertical lateral squares commute by \cref{lemma:morphisms.respect.formulas} and the two squares in the back commute since they are obtained by taking product projections. Finally, since the slanted arrows are monomorphisms, it follows that also the two squares in the front commute. 
	
	Next we need to define the action of $\Sigma$ on 2-cells; this is done by using powers by $[1]$ which exist in $\Mod(\TT,\K)$ by \cref{lemma:powers of models}. Given any 2-cell $\eta\colon M\Rightarrow N$ in $\Mod(\TT,\K)$, consider its transpose $\eta^t\colon M\to N^{[1]}$; the image of this through $\Sigma$ defines a morphism $\Sigma(\eta^t)\colon \Sigma( M)\to \Sigma( N^{[1]} )$ in $[\Syn(\TT),\K]$. Now, since for any formula $A$ we have $\br{A}_{N^{[1]}}\cong (\br{A}_N)^{[1]}$, it is easy too see that $\Sigma( N^{[1]} )\cong \Sigma(N)^{[1]}$. Then we can define $\Sigma(\eta)$ to be the transpose of $\Sigma(\eta^t)$ composed with the isomorphism just mentioned.  
	
	This concludes the definition of $\Sigma$. In the next few steps we will prove that it is 2-functorial and an equivalence.
	
	{\em $\Sigma$ defines a faithful 2-functor}. To show this consider the triangle below
	\begin{center}
		\begin{tikzpicture}[baseline=(current  bounding  box.south), scale=2]
			
			\node (a0) at (0,0.8) {$\Mod(\TT,\K)$};
			\node (b0) at (1.8,0.8) {$\IsoReg(\Syn(\TT),\K)$};
			\node (c0) at (0.91,0) {$\prod_{S\in\S}\K$};
			
			\path[font=\scriptsize]
			
			(a0) edge [dashed, ->] node [above] {$\Sigma$} (b0)
			(a0) edge [->] node [left] {ev} (c0)
			(b0) edge [->] node [right] {ev'} (c0);
		\end{tikzpicture}	
	\end{center}
	where $\tx{ev}$ and $\tx{ev'}$ are defined by evaluating on the basic sorts. These are both easily seen to be faithful; moreover, $\Sigma$ (seen as an assignment on objects, morphisms, and 2-cells) makes the triangle commute. It now follows from the fact that $\tx{ev}$ and $\tx{ev'}$ are both 2-functorial and faithful that $\Sigma$ is 2-functorial and faithful as well. 
	
	{\em $\Sigma$ is a fully faithful 2-functor}. Consider $M,N\in\Mod(\TT,\K)$ and a 2-natural transformation $h\colon \Sigma M\to\Sigma N$ in $[\Syn(\TT),\K]$. By evaluating at the objects $\{S|\top\}$, induced by the basic sorts, we obtain a family of morphism
	$$ p_S\co \br{S}_M=\Sigma M(\{S|\top\})\xrightarrow{h_{\{S|\top\}}} \br{S}_N=\Sigma N(\{S|\top\}) $$
	in $\K$. We will show that $p\defeq(p_S)_{S\in\S}$ defines a morphism of $\LL$-structures $p\colon M\to N$; it will then follow by construction of $\Sigma$ and from \cref{lemma:morphisms.respect.formulas} that $\Sigma(p)=h$.
	
	Consider a function symbol $f\colon \bar{X} \to S^{\bb c}$; then we can consider the morphism
	$$ \{\bar x\co\bar{X}|\top\}\xrightarrow{[f(\bar x)=y]}\{y\co S^{\bb c}|\top\} $$
	in $\Syn(\TT)$. The corresponding naturality square induced by the fact that $h$ is a natural transformation shows that $p$ respects the interpretation of $f$. Similarly, for any relation symbol $R\rightarrowtail \bar{X}$, the naturality square corresponding to the morphism
	$$ \{ \bar x\co\bar{X}| R(x)\}\xrightarrow{[R(\bar x)\wedge \bar x=\bar y]}\{ \bar y\co \bar{X}|\top\} $$
	shows that $p$ respects the interpretation of $R$.
	
	This shows that $\Sigma$ is full (and faithful) on 1-morphisms, or equivalently that the ordinary functor $\Sigma_0\co \Mod(\TT,\K)_0\to \IsoReg(\Syn(\TT),\K)_0$ is fully faithful. Since $\Sigma$ preserves powers by $[1]$ (as observed before) this is enough to imply that it is fully faithful as a 2-functor: for any $M,N\in\Mod(\TT,\K)$ the action 
	$$\Sigma_{M,N}\colon \Mod(\TT,\K)(M,N)\to \IsoReg(\Syn(\TT),\K)(\Sigma M,\Sigma N)$$
	 is an isomorphism of categories if and only if $\Cat([1],\Sigma_{M,N})$ is a bijection (since $[1]$ is a strong generator in $\Cat$). But 
	$$ \Cat([1],\Sigma_{M,N})\cong\Cat([0],\Sigma_{M,N^{[1]}})\cong (\Sigma_0)_{M,N^{[1]}} $$
	and the latter is a bijection since $\Sigma_0$ is fully faithful. Thus $\Sigma$ is fully faithful as a 2-functor. 
	
	{\em $\Sigma$ is essentially surjective on objects}. Consider an isoregular 2-functor $F\colon\Syn(\TT)\to \K$ and define an $\LL$-structure $M$ as follows:\begin{itemize}
		\item for any basic sort $S\in\S$ we define $\br{S}\defeq F(\{S|\top\})$;
		\item for any function symbol $f\colon S_1^{\bb c_1},\cdots, S_n^{\bb c_n}\to S^{\bb c}$ in $\LL$, we define $\br{f}$ as the composite
		$$ \br{S_1}^{\bb c_1}\times\cdots\times  \br{S_n}^{\bb c_n}\cong F(\{S_1^{\bb c_1},\cdots,S_n^{\bb c_n}|\top\}) \xrightarrow{\ F([f(\bar x)=y])\ } F(\{S^{\bb c}|\top\})\cong\br{S}^{\bb c} $$
		where we used that $F$ preserves products and finite powers, and applied \cref{lemma:powers-in-Syn(T)}.
		\item for any relation symbol $R\rightarrowtail \bar{X}\defeq S_1^{\bb c_1},\cdots, S_n^{\bb c_n}$ in $\LL$ we define its interpretation as the subobject
		$$ \br{R}\defeq F(\{\bar{X}|R(\bar x)\} )\xrightarrow{ F([R(\bar x)\wedge \bar x=\bar y]) }  F(\{\bar{X}|\top(\bar y)\})\cong \br{S_1}^{\bb c_1}\times\cdots\times  \br{S_n}^{\bb c_n} $$
		where we used the same properties as above plus that $F$ preserves monomorphisms.
	\end{itemize}
	Now, it is easy to show by induction that for any term $(\bar x\co\bar{X})\ t(x)\co Y$ we have an isomorphism 
	\begin{center}
		\begin{tikzpicture}[baseline=(current  bounding  box.south), scale=2]
			
			\node (a0) at (0,0.8) {$\br{ \bar{X} }$};
			\node (b0) at (1.1,0.8) {$F(\{\bar{X}|\top(\bar x)\} )$};
			\node (c0) at (0,0) {$\br{Y}$};
			\node (d0) at (1.1,0) {$F(\{Y|\top(y)\} )$};
			
			\path[font=\scriptsize]
			
			(a0) edge [->] node [above] {$\cong$} (b0)
			(a0) edge [->] node [left] {$\br{t}$} (c0)
			(b0) edge [->] node [right] {$F([t(\bar x)=y])$} (d0)
			(c0) edge [->] node [below] {$\cong$} (d0);
		\end{tikzpicture}
	\end{center}
	in the arrow category $\K^{[1]}$. Similarly, if we have an isoregular formula $(\bar x\co\bar{X})\ A(x) \co \mathsf{prop}$ relative to $\TT$, then by induction we construct an isomorphism 
	\begin{center}
		\begin{tikzpicture}[baseline=(current  bounding  box.south), scale=2]
			
			\node (a0) at (0,0.8) {$\br{A}$};
			\node (b0) at (1.1,0.8) {$F(\{\bar{X}|A(\bar x)\} )$};
			\node (c0) at (0,0) {$\br{ \bar{X}}$};
			\node (d0) at (1.1,0) {$F(\{\bar{X}|\top(\bar y)\} )$};
			
			\path[font=\scriptsize]
			
			(a0) edge [->] node [above] {$\cong$} (b0)
			(a0) edge [>->] node [left] {} (c0)
			(b0) edge [>->] node [right] {$F([A(\bar x)\wedge \bar x=\bar y])$} (d0)
			(c0) edge [->] node [below] {$\cong$} (d0);
		\end{tikzpicture}
	\end{center}
	in $\K^{[1]}$. To show this we use the various lemmas of Section~\ref{section:SynT} which imply that the operations defining our formulas coincide with taking certain finite 2-limits or image-factorizations in $\Syn(\TT)$, and hence are preserved by $F$.
	
	Now, any axiom $(\bar x\co \bar{X})\ A\vdash B$ in $\TT$ induces the commutative triangle below left in $\Syn(\TT)$.
	\begin{center}
		\begin{tikzpicture}[baseline=(current  bounding  box.south), scale=2]
			
			\node (a0) at (0,0.8) {$\{\bar x\co\bar{X}|A(\bar x)\}$};
			\node (b0) at (2,0.8) {$\{\bar x'\co\bar{X}|B(\bar x')\}$};
			\node (c0) at (1,0) {$\{\bar x''\co\bar{X}|\top(\bar x'')\}$};
			
			\path[font=\scriptsize]
			
			(a0) edge [>->] node [above] {$[A(\bar x)\wedge \bar x=\bar x']$} (b0)
			(a0) edge [>->] node [left] {$[A(\bar x)\wedge \bar x=\bar x'']\ $} (c0)
			(b0) edge [>->] node [right] {$\ [B(\bar x')\wedge \bar x'=\bar x'']$} (c0);
		\end{tikzpicture}	
		\quad\quad
		\begin{tikzpicture}[baseline=(current  bounding  box.south), scale=2]
			
			\node (a0) at (0,0.8) {$\br{A}$};
			\node (b0) at (1.6,0.8) {$\br{B}$};
			\node (c0) at (0.8,0) {$\br{\bar{X}}$};
			
			\path[font=\scriptsize]
			
			(a0) edge [>->] node [above] {} (b0)
			(a0) edge [>->] node [left] {} (c0)
			(b0) edge [>->] node [right] {} (c0);
		\end{tikzpicture}	
	\end{center}
	By applying $F$ and using the isomorphisms above we obtain the commutative triangle above right, which says exactly that $M$ satisfies $(\bar x\co\bar{X})\ A\vdash B$. It follows that $M\in\Mod(\TT)$, and thanks to the isomorphisms above that $\Sigma M\cong F$.
\end{proof}

\section{Accessibility of 2-categories of models} 
\label{sec:accessibility}

\subsection{Basics} 

The notion of accessible 2-category that we use is a specialization of one that has been considered for enriched categories, for the first time, in~\cite{LR12:articolo}. Another notion of accessibility (for enriched categories) had been introduced before in~\cite{BQ96:articolo}, but  these coincide in the 2-categorical context by~\cite[Theorem~3.14]{LT21:articolo}.

\begin{defi} \leavevmode
\begin{itemize}
\item Let $\lambda$ be a regular cardinal. We say that a 2-category $\K$ is 
	\emph{$\lambda$-accessible} if it is equivalent to the free cocompletion of a small 2-category under $\lambda$-filtered colimits. We say that $\K$ is \emph{accessible} if it is $\lambda$-accessible for some $\lambda$.
	\item A functor $F\colon\K\to\L$ between accessible 2-categories is called \emph{accessible} if it preserves $\lambda$-filtered colimits for some $\lambda$.
	\end{itemize}
\end{defi}

Given a 2-category $\K$ with $\lambda$-filtered colimits, an object $X\in K$ is called 
\emph{$\lambda$-presentable} if $\K(X,-)\colon \K\to\Cat$ preserves $\lambda$-filtered colimits. Then, by~\cite[Proposition~3.15]{LT22:virtual}, a 2-category~$\K$ with $\lambda$-filtered colimits is $\lambda$-accessible if and only if there is a set $\G$ of $\lambda$-presentable objects such that every objects of $\K$ can be written as a $\lambda$-filtered colimit of objects in $\G$.  

The accessible 2-categories that we are interested in are those that admit all flexible limits; that is, have all small products, inserters, and equifiers. (Flexible limits also include splitting of idempotents, but these already exist in any accessible 2-category). Note that any 2-category with flexible limits has in particular also all (weighted) pseudolimits and (weighted) bilimits, see for instance~\cite{lack20102}.

Accessible 2-categories with flexible limits have been studied in~\cite{LR12:articolo,BLV:articolo,Bou2021:articolo}. 
We recall some of the key facts that will be used later.

\begin{prop}
	Any accessible 2-category with flexible limits also has all bicolimits. 
\end{prop}
\begin{proof}
	See for instance \cite[Example~8.4]{BLV:articolo}, or \cite[Theorem~9.4]{LR12:articolo}, where it is shown that any accessible 2-category has a specific kind of ``weak'' colimits, and these subsume bicolimits.
\end{proof}

\begin{theo}\label{theo:left-biadj-exist}
	Let $U\colon\K\to \L$ be an accessible 2-functor between accessible 2-categories. If $\K$ has flexible limits and $U$ preserves them, then $U$ has a left biadjoint.
\end{theo}
\begin{proof}
	This is shown in~\cite[Section~9.3]{BLV:articolo}.
\end{proof}

\subsection{Accessibility of models} 
We will now show that every 2-category of models of a isoregular theory is accessible with flexible limits.

\begin{theo}\label{theo.acc+flex}
	Let $\C$ and $\K$ be isoregular 2-categories, with $\C$ small. 
	\begin{enumerate}
		\item If $\K$ is accessible then $\IsoReg(\C,\K)$ is also accessible and the inclusion
		$$\IsoReg(\C,\K)\hookrightarrow [\C,\K]$$
		is an accessible 2-functor;
		\item If $\K$ has flexible limits and fully faithful regular epimorphisms are stable under them, then $\IsoReg(\C,\K)$ has flexible limits and the inclusion
		$$\IsoReg(\C,\K)\hookrightarrow [\C,\K]$$
		preserves them. 
	\end{enumerate} 
\end{theo}
\begin{proof}
	For part~(i), consider the following commutative square, which we shall prove is a pullback (for this we do not need to assume that $\K$ is accessible).
	\begin{center}
		\begin{tikzpicture}[baseline=(current  bounding  box.south), scale=2]
			
			\node (a0) at (0,0.9) {$\IsoReg(\C,\K)$};
			\node (b0) at (1.6,0.83) {$\prod_{q\in \tx{S}}\K^{\cong}$};
			\node (c0) at (0,0) {$\Lex(\C,\K)$};
			\node (d0) at (1.6,-0.07) {$\prod_{q\in\tx{S}}\K^{[1]}$};
			\node (e0) at (0.3,0.65) {$\lrcorner$};
			
			\path[font=\scriptsize]
			
			(a0) edge [->] node [above] {$(H_q)_{q\in\tx{S}}$} ([yshift=2]b0.west)
			(a0) edge [right hook->] node [left] {$J$} (c0)
			([yshift=2]b0.south) edge [right hook->] node [right] {$I$} ([yshift=-1.3]d0.north)
			(c0) edge [->] node [below] {$(G_q)_{q\in\tx{S}}$} ([yshift=2]d0.west);
		\end{tikzpicture}	
	\end{center}
	Above, $\tx{S}$ is the set of all fully faithful regular epimorphisms in $\C$, while $I$ and $J$ are full subcategory inclusions (and are also isofibrations). For $q\colon A\to B\in S$, we define the 2-functor $G_q$ as follows: given $F\in\Lex(\C,\K)$ we let $K_{F,q}$ be the coequaliser of the kernel pair of $Fq$ (this exists since $\K$ is isoregular and the kernel pair of $Fq$ is fully faithful); then $G_q(F)$ is defined as the morphism 
	$$ G_q(F)\colon K_{F,q}\longrightarrow FB $$
	induced by the universal property of the coequaliser. Then, the 2-functor $H_q$ can be seen as the restriction of $G_q$ along $J$; this is well defined since whenever $F$ is isoregular then $G_q(F)$ is an isomorphism by definition. 
	
	Now it is easy to see that the square is actually a pullback; indeed, a lex 2-functor $F$ is isoregular if and only if for any $q\in \tx S$ the morphism $Fq$ is a fully faithful regular epimorphism, if and only if $Fq$ is a regular epimorphism (since $Fq$ is automatically fully faithful), if and only if $G_q(F)$ is an isomorphism.
	
	If $\K$ is accessible, then by \cite[Corollary~5.13]{LT22:virtual} so is $\Lex(\C,\K)$ and its inclusion into $[\C,\K]$. Moreover, $\textstyle\prod_{q\in \tx{S}}\K^{\cong}$ and $\textstyle\prod_{q\in\tx{S}}\K^{[1]}$ are accessible too by \cite[Theorem~5.5]{LT22:virtual}. Finally, $I$ is an accessible 2-functor since it preserves all existing colimits (being induced by precomposition with the inclusion $[1]\to\ \cong$), and so are the $G_q$ since colimits commute with (filtered) colimits. It follows again by \cite[Theorem~5.5]{LT22:virtual} that $\IsoReg(\C,\K)$ is an accessible 2-category and $J$ is an accessible 2-functor. By post composing with the inclusion into $[\C,\K]$ we also obtain that
	$$\IsoReg(\C,\K)\hookrightarrow [\C,\K]$$
	is accessible.
	
For part~(ii), assume now that $\K$ has flexible limits and fully faithful regular epimorphisms are stable under them. Note first that since limits commute with themselves, $\Lex(\C,\K)$ has flexible limits too and they are computed pointwise. Then, if we take a flexible limit of isoregular 2-functors in $\Lex(\C,\K)$, then the resulting 2-functor $F$ is still isoregular since for any fully faithful regular epimorphism $q$ in $\C$, the morphism $Fq$ is a flexible limit of fully faithful regular epimorphisms in $\K$, and hence a fully faithful regular epimorphism itself by our assumptions. It follows that $\IsoReg(\C,\K)$ is closed in $\Lex(\C,\K)$ under flexible limits; hence the thesis follows.
\end{proof}

We deduce the following result from~\cref{theo.acc+flex} using \cref{theorem:models=functors}.

\begin{theo}\label{theo:accessibility-of-models}
	Let $\TT$ be an isoregular theory and $\K$ be an isoregular 2-category. Then:\begin{enumerate}
		\item if $\K$ is accessible then $\Mod(\TT,\K)$ is also accessible and so is the evaluation 2-functor
		$$\br{A}_{(-)} \colon \Mod(\TT,\K)\longrightarrow \K$$
		for any isoregular formula $(\bar x\co\bar X)\ A(\bar x)$;
		\item if $\K$ has flexible limits and fully faithful regular epimorphisms are stable under them, then $\Mod(\TT,\K)$ has flexible limits and the evaluation 2-functor
		$$\br{A}_{(-)} \colon \Mod(\TT,\K)\longrightarrow \K$$
		preserves them, for any isoregular formula $(\bar x\co\bar X)\ A(\bar x)$.
	\end{enumerate}
\end{theo}
\begin{proof}
	Given an isoregular formula  $(\bar x\co\bar X)\ A(\bar x)$, the evaluation 2-functor can be seen as the composite
	$$ \Mod(\TT,\K)\xrightarrow{\ \Sigma\ } \IsoReg(\C_\TT,\K)\hookrightarrow [\C_\TT,\K]\xrightarrow{ \tx{ev}_{\{\bar X|A\}} }\K\mathrlap{,} $$
	where $\Sigma$ is the equivalence of Theorem~\ref{theorem:models=functors}. Then the result follows immediately from Theorem~\ref{theo.acc+flex} plus the fact that $\tx{ev}_{\{\bar X|A\}}$ is continuous and cocontinuous.
\end{proof}

\begin{cor}\label{cor:accessibility-forgetful}
	Let $\TT$ be isoregular over $\LL$ and $\TT'$ be isoregular over $\LL'$ with $\LL\subseteq \LL'$ and $\TT\subseteq \TT'$. Let $\K$ be an isoregular 2-category and consider the induced forgetful 2-functor
	$$ U\colon \Mod(\TT',\K)\longrightarrow \Mod(\TT,\K).$$
	\begin{enumerate}
		\item if $\K$ is accessible then $U$ is an accessible 2-functor;
		\item if $\K$ has flexible limits and fully faithful regular epimorphisms are stable under them, then $U$ preserves flexible limits.
	\end{enumerate}
\end{cor}
\begin{proof}
	For any isoregular $\LL$-formula $(\bar x\co\bar X)\ A(\bar x)$, which can naturally be seen also as an isoregular $\LL'$-formula, the triangle
	\begin{center}
		\begin{tikzpicture}[baseline=(current  bounding  box.south), scale=2]
			
			\node (a0) at (0,0.8) {$\Mod(\TT',\K)$};
			\node (b0) at (1.5,0.8) {$\Mod(\TT,\K)$};
			\node (c0) at (0.75,0) {$\K$};
			
			\path[font=\scriptsize]
			
			(a0) edge [->] node [above] {$U$} (b0)
			(a0) edge [->] node [left] {$\br{A}_{(-)}$} (c0)
			(b0) edge [->] node [right] {$\ \ \br{A}_{(-)}$} (c0);
		\end{tikzpicture}	
	\end{center}
	commutes. Since the right vertical 2-functors are jointly conservative when we let $A$ vary among all isoregular $\LL$-formulas, the results follows from Theorem~\ref{theo:accessibility-of-models} above.
\end{proof}

\section{Applications}
\label{sec:applications}

\subsection{Preliminaries}
In this section we focus on models of isoregular theories in $\Cat$, providing several applications of our general results, both proving new fact and giving new proofs of known ones
in a uniform manner.  We begin by considering categories with finite limits and colimits (\cref{limits-colimits}), move on to categories satisfying a variety of exactness conditions, as studied in
categorical algebra (\cref{sec:exactness}), isofibrations and Grothendieck fibrations (\cref{sec:fibrations}) and conclude with categories of interest in categorical logic, such as comprehension categories and clans (\cref{sec:varia}).

Let us first make some general observations. If $\TT$ is an isoregular theory over a language $\LL$, and $M$ is an $\LL$-structure in $\Cat$ (or more generally in any isoregular $\K$ where the (strong epi, mono) factorization system exists), then the interpretation of any isoregular formula $(\bar X) \ A \co \prp$, relative to $\TT$, is well-defined as a subobject $\br A\rightarrowtail\br{\bar X}$ since to interpret existential quantification we can take the (strong epi, mono) factorisation of the relevant arrow. Note that such interpretation coincides with the one we introduced since the (fully faithful regular epimorphism, monomorphism)-factorisation, when it exists, coincides with the (strong epimorphism, monomorphism)-factorisation.

Because of this the statement of \cref{prop:1-dimensional-is-enough} below is well-formed. This result will be essential in recognising the 2-categories of models of isoregular theories, as it allows us test the validity of the axioms at the level of the underlying sets (of objects), as with ordinary satisfaction.

\begin{prop}\label{prop:1-dimensional-is-enough}
	Let $\TT$ be an isoregular theory on a language $\LL$, and $\H$ a full subcategory of $\Str(\LL,\Cat)$ closed under powers by $[1]$. Suppose that an $\LL$-structure $\M\defeq\br -$ lies in $\H$ if and only if for any axiom $(\bar X)\  A\vdash B$ in $\TT$ the inclusion
	$$ \tx{Ob}(\br{A})\subseteq \tx{Ob}(\br{B}), $$
	holds as subsets of $\tx{Ob}(\br{\bar X})$. Then $\H=\Mod(\TT,\Cat)$.
\end{prop}
\begin{proof}
	The hypotheses can be restates as saying that $\M$ is in $\H$ if and only if for any $(\bar X)\  A\vdash B$ in $\TT$ we have a commutative diagram as below left,
	\begin{center}
		\begin{tikzpicture}[baseline=(current  bounding  box.south), scale=2]
			
			\node (a0) at (0,0.8) {$\Cat(1,\br A)$};
			\node (b0) at (1.6,0.8) {$\Cat(1,\br B)$};
			\node (c0) at (0.8,0) {$\Cat(1,\br{\bar X})$};
			
			\path[font=\scriptsize]
			
			(a0) edge [>->] node [above] {} (b0)
			(a0) edge [>->] node [left] {$\Cat(1,m_A)$} (c0)
			(b0) edge [>->] node [right] {$\Cat(1,m_B)$} (c0);
		\end{tikzpicture}	
		\quad\quad\quad
		\begin{tikzpicture}[baseline=(current  bounding  box.south), scale=2]
			
			\node (a0) at (0,0.8) {$\Cat([1],\br A)$};
			\node (b0) at (1.6,0.8) {$\Cat([1],\br B)$};
			\node (c0) at (0.8,0) {$\Cat([1],\br{\bar X})$};
			
			\path[font=\scriptsize]
			
			(a0) edge [>->] node [above] {} (b0)
			(a0) edge [>->] node [left] {$\Cat([1],m_A)$} (c0)
			(b0) edge [>->] node [right] {$\Cat([1],m_B)$} (c0);
		\end{tikzpicture}	
	\end{center}
	where $m_A$ and $m_B$ denote the subobject inclusions. But since $\H$ is closed under powers by $[1]$, the structure $\M^{[1]}$ also satisfies the same property; thus, using the universal property of powers the commutativity of the triangle above left is equivalent to that of the triangle above right. 
	Since $[1]$ is a strong generator in $\Cat$, this last condition is equivalent to asking that $\br A\leq\br B$ as subobjects of $\br{\bar X}$, which says exactly that $\M\in\Mod(\TT,\Cat)$.
\end{proof}

Given this, it becomes quite easy to check when a given $\LL$-structure satisfies an entailment between isoregular formulas. Indeed, we are reduced to check that certain inclusions of sets hold, and at the level of object the interpretation of an isocartesian formula is what one expects it to be (remembering that we now have non-discrete arities):\begin{itemize}
	\item $\tx{Ob}(\br{S^{\bb b}})=\{f\co \bb b\to \br{S}|\ f \text{ is a functor} \}$; that is, the set of diagrams of shape $\bb b$ in $\br{S}$;
	\item $\tx{Ob}(\br{A^{[1]}})=\tx{Mor}(\br{A})$ is the set of morphisms in $\br{A}$;
	\item $\tx{Ob}(\br{A\wedge B})=\{ a\in \tx{Ob}(\br{\bar X}) |\  a\in \tx{Ob}(\br{A})\text{ and } a\in \tx{Ob}(\br{B})\}$;
	\item $\tx{Ob}(\br{(\exists y\co Y)A(\bar x,y)})=\{ a\in \tx{Ob}(\br{\bar X}) | \text{ there is } b\in \tx{Ob}(\br{Y})\text{ such that } (a,b)\in \tx{Ob}(\br{A})\}$;
\end{itemize}
For this reason, in the examples below we shall not dwell too much on trying to explain what it means for an $\LL$-structure to satisfy a given axiom (on object), as that simply amounts to translating into natural language what is written in symbols.

\subsection{Limits and colimits}\label{limits-colimits}

Fix a finite category $\bb b$. Denote by $\blim$ the 2-category of small categories with $\bb b$-limits, $\bb b$-limit preserving functors, and 2-natural transformations. 
 We describe an isoregular theory $\TT_{\bb b}$ whose models in $\Cat$ are exactly small categories with limits of shape $\bb b$; this follows the approach of~\cite[Example~4.18]{RT25ERegular}.

Let us set some notation first. We denote by $0*\bb b$ the category obtained by freely adding to $\bb b$ an initial object $0$; similarly, we let $[1]*\bb b$ be the category obtained by adding a further initial object to $0*\bb b$ (or equivalently, adding an ``initial arrow'' to $\bb b$). Finally, we let $\bb I*\bb b$ be the category obtained by adding an inverse to the new morphism of $[1]*\bb b$. To motivate the introduction of these notations, note that if $\C\in\Cat$, then an object of  $\C^{0*\bb b}$ is the same as a functor $\bb b\to \C$ together with a cone over it; while to give an object of $\C^{[1]*\bb b}$ is the same as giving a functor $\bb b\to \C$, two cones over it, and a morphism in $\C$ between the two cones; in $\C^{\bb I*\bb b}$ such morphism of cones is equipped with an inverse.

Now, define the language $\mathbb L_{\bb b}$ to have one sort $S$ and just a relation symbol $R_{\bb b}:S^{0*\bb b}$. The idea is that for each $\mathbb L$-structure $\br{-}$ in $\Cat$, we shall think of $\br{R_{\bb b}}$ as the category of limiting cones over diagrams of shape $\bb b$.

First, let us define the formula $(x,y\co S^{0*\bb b}, z\co S^{[1]*\bb b})\ A(x,y,z) $ as below
$$
(x,y\co S^{0*\bb b}, z\co S^{[1]*\bb b})\  R_{\bb b}(y)\wedge (z\cdot j_0=x) \wedge (z\cdot j_1=y) \mathrlap{,}
$$
where $j_0,j_1\colon 0*\bb b\to [1]*\bb b$ sending $\bb b$ to itself and picking out the domain and codomain of the freely added map respectively; below we will also use the inclusion $\iota\co \bb b\to 0*\bb b$.
We now define the theory $\TT_{\bb b}$ to consist of the following axioms:\begin{enumerate}
	\item $\J_{\mathsf{mono}}(A)$;
	\item $ (x,y\co S^{0*\bb b})\ R_{\bb b}(y),\ (x\cdot \iota=y\cdot \iota) \vdash (\exists z\co S^{[1]*\bb b})\ A(x,y,z)$;
	\item $ (z\co S^{\bb I*\bb b})\ R_{\bb b}(z\cdot j_0)\vdash R_{\bb b}(z\cdot j_1)$;
	\item $ (z\co S^{[1]\times 0*\bb b})\ R_{\bb b}(\dom(z)), R_{\bb b}(\cod(z))\vdash R_{\bb b}^{[1]}(z) $;
	\item $ (x\co S^{\bb b})\ \vdash (\exists y\co S^{0*\bb b})\ R_{\bb b}(y)\wedge (y\cdot \iota= x) $.
\end{enumerate}

Alternatively, one can axiomatise the same class of models with axioms~(i)--(iv) plus additional axioms that ensure that the existential quantifier in~(v) is well-formed. However, this is unnecessary since this can be derived from axioms~(i)--(iv), as we show in the proof of \cref{lemma:b-lim} below.

\begin{lemma}\label{lemma:b-lim}
	The theory $\TT_{{\bb b}}$ is isoregular and $\blim \cong \Mod(\TT_{\bb b}, \Cat)$.
%	$$ \Mod(\TT_{\bb b},\Cat)\cong\bb b\tx{-}\bo{lim}. $$		
\end{lemma}
\begin{proof}
	Let us begin by proving that $\TT_{\bb b}$ is isoregular. Note that the existential quantification in~(ii) is well-formed by~(i), so we only need to prove that the exists in (v) is well-formed; that is, that 
	$\J_{\mathsf{faithful}}(R_{\bb b}(y)\wedge (y\cdot \iota= x))$ and  $\J_{\mathsf{full}\text{-}\mathsf{id}}(R_{\bb b}(y)\wedge (y\cdot \iota= x))$ are derivable. 
	
	We first show that the faithfulness condition holds; the assumption is 
	$$ R_{\bb b}^{[1]}(v), R_{\bb b}^{[1]}(v'), (v\cdot \iota^{[1]}=u), (v'\cdot \iota^{[1]}=u),\dom(v)=\dom(v'), \cod(v)=\cod(v') $$
	in context $(u\co S^{[1]\times \bb b}, v,v'\co S^{[1]\times 0*\bb b})$, and we need to derive that $v=v'$. Notice that by transitivity of equality we can already deduce that $v\cdot \iota^{[1]}=v'\cdot \iota^{[1]}$, or equivalently $v\cdot [1]\times \iota=v'\cdot [1]\times \iota$. Now, let 
	$$\kappa\co [1]*\bb b\to [1]\times( 0*\bb b)$$ 
	be the functor obtained by identifying $\bb b$ with $\{1\}\times \bb b$ and $[1]$ with $[1]\times \{0\}$ (where we are using $[1]=\{0\to 1\}$).  By Rules~\ref{equ:action} and~\ref{equ:power-dom-cod-prop} we can easily derive 
	$$ A(\dom(v),\cod(v),v\cdot\kappa) \wedge A(\dom(v'),\cod(v'),v'\cdot\kappa). $$
	By axiom (i) and the faithfulness hypotheses we then deduce that $v\cdot\kappa=v'\cdot\kappa$. Now note that $\kappa$, $[1]\times \iota$, and the map $0*\bb b\to [1]\times 0*\bb b$ inducing $\dom$ are jointly epimorphic. Thus it follows from Rule~\ref{equ:epimorphic-family} that $v=v'$.
	
	Let us now show fullness on identities  whose assumption is 
	$$ (x\co S^{\bb b}, y,y'\co S^{0*\bb b})\quad R_{\bb b}(y), R_{\bb b}(y'), (y\cdot \iota=x), (y'\cdot \iota=x).$$
	By axiom (ii) we derive
	$$ (\exists z\co S^{[1]*\bb b})\quad z\cdot j_0=y \wedge z\cdot j_1=y'. $$
	Consider the two pushouts below.
	\begin{center}
		\begin{tikzpicture}[baseline=(current  bounding  box.south), scale=2]
			
			\node (a0) at (0,0.8) {$\bb b$};
			\node (b0) at (1.5,0.8) {$[1]\times \bb b$};
			\node (c0) at (0,0) {$0*\bb b $};
			\node (d0) at (1.5,0) {$0* ([1] \times\bb b) $};
			
			\path[font=\scriptsize]
			
			(a0) edge [right hook->] node [above] {$\{0\} \! \times \! 1_{\bb b}$} (b0)
			(a0) edge [right hook->] node [left] {$\iota$} (c0)
			(b0) edge [right hook->] node [right] {$\iota $} (d0)
			(c0) edge [right hook->] node [below] {$0\!*\!(\{0\}\!\times\! 1_{\bb b})$} (d0);
		\end{tikzpicture}	
		\qquad
		\begin{tikzpicture}[baseline=(current  bounding  box.south), scale=2]
			
			\node (a0) at (0,0.8) {$0*\bb b$};
			\node (b0) at (1.5,0.8) {$0* ([1]\times \bb b) $};
			\node (c0) at (0,0) {$[1]*\bb b $};
			\node (d0) at (1.5,0) {$[1]\times (0*\bb b) $};
			
			\path[font=\scriptsize]
			
			(a0) edge [right hook->] node [above] {$0\!*\!(\{1\}\!\times\! 1_{\bb b})$} (b0)
			(a0) edge [right hook->] node [left] {$j_0$} (c0)
			(b0) edge [right hook->] node [right] {$\rho$} (d0)
			(c0) edge [right hook->] node [below] {$\kappa$} (d0);
		\end{tikzpicture}	
	\end{center}
	By Rule~\ref{equ:union} applied to the pushout on the left we derive
	$$ (\exists v'\co S^{0* ([1] \times\bb b)} )\quad v'\cdot 0\!*\!(\{0\}\!\times\! 1_{\bb b})=y \wedge v\cdot \iota=\id_x. $$
	Then, by Rule~\ref{equ:union} applied to the pushout on the right, we obtain
	$$ (\exists v\co S^{[1]\times 0*\bb b})\quad v\cdot \kappa=z \wedge v\cdot \rho=v'$$
	using axiom (iv) and the previous hypotheses we derive 
	$$(\exists v\co S^{[1]\times 0*\bb b})\ R_{\bb b}^{[1]}(v)\wedge v\cdot\iota^{[1]}=\id_x\wedge \dom(v)=y\wedge \cod(v)=y'$$
	concluding the proof that the required judgements are derivable.
	
	It remains to prove that we have an isomorphism $\Mod(\TT_{\bb b},\Cat)\cong \blim$. First notice that $\blim$ can be identified with a full subcategory of $\Str(\LL_{\bb b},\Cat)$ by sending $\C\in \blim$ to the $\LL_{\bb b}$-structure with $\br S\defeq\C$ and $\br{R_{\bb b}}$ being the full subcategory of $\C^{0*\bb b}$ spanned by all the limiting cones. Note that the inclusion is full since a functor preserves $\bb b$-limits if and only if it preserves the limiting cones of diagrams of shape $\bb b$, and 2-cells in $\blim$ are just natural transformations. 
	
	It is easily seen that, under this identification, $\blim$ is closed in $\Str(\LL_{\bb b},\Cat)$ under powers by $[1]$; thus we can apply \cref{prop:1-dimensional-is-enough}. Therefore, it is enough to show that a structure $M$ satisfies the axioms of $\TT_{\bb b}$ on the underlying set of objects if an only if $M\in \blim$. It is easy to see that $M$ satisfies the axioms~(i) and~(ii) if every object of $\br{R_{\bb b}}$ is a limiting cone, it satisfies~(iii) if $\br{R_{\bb b}}$ is closed under isomorphism (and hence contains all limiting cones on a diagram), satisfies (iv) if any natural transformation between limit cones in an arrow in $\br{R}$, and finally satisfies (v) if every $\bb b$-shaped diagram admits a limiting cone in $\br{R_{\bb b}}$. Any object of $\blim$ clearly satisfies the conditions above. Conversely, if $M$ satisfies such conditions, then $\br S$ has $\bb b$-limits and $\br{R_{\bb b}}$ is the full subcategory of $\br{S}^{0*\bb b}$ spanned by all the limiting cones; thus $M$ lies in $\blim$.
\end{proof}

In a similar way we can construct a language $\LL_{{\bb b}}'$ and an isoregular theory $\TT'_{{\bb b}}$ such that 
\[
 \Mod(\TT'_{\bb b}, \Cat) \cong  \bb b\text{-}{\tx{colim}}
 \]
is isomorphic to the 2-category of small categories with $\bb b$-colimits, $\bb b$-colimit preserving functors, and 2-natural transformations.
The difference from the previous case is that, instead of freely adding an initial object (or initial arrow/isomorphism) to $\bb b$, we add a terminal object (or terminal arrow/isomorphism). So the language $\LL_{\bb b}'$ will still have one sort $S$, but a relation symbol $R_{\bb b}'\co S^{\bb b*0}$. The theory $\TT_{\bb b}'$ will have essentially the same axioms of $\TT_{\bb b}$, with properly modified arities. 

Now we can allow $\bb b$ to vary and define theories for the 2-categories: $\Lex$ of small categories with finite limits, finite-limit preserving functors, and natural transformations; $\Rex$ of small categories with finite colimits, finite-colimit preserving functors, and natural transformations; and $\RLex$ of small categories with finite limits and finite colimits, finite-limit and finite-colimit preserving functors, and natural transformations.

More precisely, we define languages 
$$\LL_{\tx{lex}}\defeq \bigcup_{\bb b\in\tx{finCat}} \LL_{\bb b}, \qquad \LL_{\tx{rex}}'\defeq \bigcup_{\bb b\in\tx{finCat}} \LL_{\bb b}', \qquad \LL_{\tx{rlex}}\defeq \LL_{\tx{lex}}\cup \LL_{\tx{rex}}$$
and corresponding isoregular theories
$$\TT_{\tx{lex}}\defeq \bigcup_{\bb b\in\tx{finCat}} \TT_{\bb b}, \qquad \TT_{\tx{rex}}'\defeq \bigcup_{\bb b\in\tx{finCat}} \TT_{\bb b}', \qquad \TT_{\tx{rlex}}\defeq \TT_{\tx{lex}}\cup \TT_{\tx{rex}}$$
respectively on $\LL_{\tx{lex}}$, $\LL_{\tx{rex}}$, and $\LL_{\tx{rlex}}$.

It follows immediately by Lemma~\ref{lemma:b-lim}, its dual, and how models are defined that the following holds. 

\begin{lemma} \label{thm:lex-isoreg}
	The theories $\TT_{\tx{lex}}$, $\TT_{\tx{rex}}$, and $\TT_{\tx{rlex}}$ are isoregular and \begin{enumerate}
		\item $ \Mod(\TT_{\tx{lex}},\Cat)\cong\Lex$,
		\item $ \Mod(\TT_{\tx{rex}},\Cat)\cong\Rex$.
		\item $ \Mod(\TT_{\tx{rlex}},\Cat)\cong\RLex$.
	\end{enumerate}
\end{lemma}

\begin{rmk} \label{thm:lex-rex-rlex} Since $\Cat$ is accessible, has flexible limits, and retract equivalences are stable under them, 
the combination of~\cref{theo:accessibility-of-models}  with~\cref{thm:lex-isoreg} gives a new proof that $\Lex$, $\Rex$, and $\RLex$ are accessible with flexible limits, which was shown in~\cite[Section~6.4]{Bou2021:articolo}. Additionally, the forgetful functors
	\begin{center}
		\begin{tikzpicture}[baseline=(current  bounding  box.south), scale=2]
			
			\node (a0) at (0,0.8) {$\RLex$};
			\node (b0) at (1,0.8) {$\Lex$};
			\node (c0) at (0,0) {$\Rex$};
			\node (d0) at (1,0) {$\Cat$};
			
			\path[font=\scriptsize]
			
			(a0) edge [->] node [above] {} (b0)
			(a0) edge [->] node [left] {} (c0)
			(b0) edge [->] node [right] {} (d0)
			(c0) edge [->] node [below] {} (d0);
		\end{tikzpicture}	
	\end{center}
	are accessible and preserve flexible limits, and therefore they all have left biadjoints  by \cref{theo:left-biadj-exist}. Note that the left biadjoint to $\RLex \to \Cat$
	provides a finite case of Joyal's bicompletion.
\end{rmk}

\subsection{Exactness conditions} \label{sec:exactness}
We now consider 2-categories of categories satisfying a variety of exactness conditions, namely regular, Mal'sev, exact, protomodular, and semiabelian categories
(see~\cite{Bor94:libro,borceux2004mal,GranM:intrc} for details).

We begin by considering regular categories and write $\Reg$ for the 2-category of regular categories, regular functors, and natural transformations.
Let $\LL_{\tx{reg}}$ be the language obtained by adding to $\LL_{\tx{lex}}$ one relation symbol $R_{\tx{ckp}}:S^{\bb p*0}$, where $\bb p$ is the free-living pair of arrows. This new relation will be used to encode the coequalisers of kernel pairs.
To define $\TT_{\tx{reg}}$ we extend $\TT_{\tx{lex}}$ by adding a few axioms. To begin with, we ask that every element of $R_{\tx{ckp}}$ is a coequaliser; this is done by taking the dual of axioms (i)-(iv) of the theory $\TT_{\bb b}$. In more detail, we first define the formula $(x,y\co S^{\bb p*0}, z\co S^{\bb p*[1]})\ A(x,y,z) $ as below
$$
(x,y\co S^{\bb p*0}, z\co S^{\bb p*[1]})\  R_{\tx{ckp}}(x)\wedge (z\cdot j_0=x) \wedge (z\cdot j_1=y)
$$
where $j_0,j_1\colon \bb p*0\to \bb p*[1]$ pick out the domain and codomain of the freely added map respectively. Then add to $\TT_{\tx{lex}}$ the axioms:\begin{enumerate}
	\item $\J_{\mathsf{mono}}(A)$;
	\item $ (x,y\co S^{\bb p*0})\ R_{\tx{ckp}}(x),\ (x\cdot \iota=y\cdot \iota) \vdash (\exists z\co S^{\bb p*[1]})\ A(x,y,z)$;
	\item $ (z\co S^{\bb p*\bb I})\ R_{\tx{ckp}}(z\cdot j_0)\vdash R_{\tx{ckp}}(z\cdot j_1)$;
	\item $ (z\co S^{[1]\times \bb p*0})\ R_{\tx{ckp}}(\dom(z)), R_{\tx{ckp}}(\cod(z))\vdash R_{\tx{ckp}}^{[1]}(z) $.
\end{enumerate}
The next step is to ask that every kernel pair has a coequaliser. To that end, we use that kernel pairs are just a specific kind of pullback. For simplicity denote by $R_{\tx{pb}}\defeq R_{\bb c}$, where $\bb c$ is the free cospan, the relation of $\LL_{\tx{reg}}$ containing the pullback squares. This has sort $S^{0*\bb c}$ which, since $0*\bb c\cong [1]\times [1]$, by substitution we can replace with $S^{[1]\times [1]}$. 
Consider the formula $(x\co S^{\bb p*0})\ A_{\tx{kp}}(x) $ as below
$$
 (\exists y\co S^{[1]\times [1]})\ R_{\tx{pb}}(y)\wedge (\prd(y)=\pr_q(x)) \wedge (\prr(y)=\pr_q(x))\wedge (\pru(y)=\pr_f(x)) \wedge (\prl(y)=\pr_g(x)),
$$
where $f,g$ denote the parallel arrows in $\bb p*0$ and $q$ the coequalising one. 

We now add the axiom stating that every kernel pair has a coequaliser:\begin{enumerate}
	\item[(v)]  $(x\co S^{\bb p*0})\ A_{\tx{kp}}(x)\vdash (\exists y:S^{\bb p*0})\ R_{\tx{ckp}}(y)\wedge (\pr_f(y)=\pr_f(x))\wedge (\pr_g(y)=\pr_g(x)). $
\end{enumerate}

Finally, we add an axiom saying that coequalisers of kernel pairs are stable under pullback. For that consider the finite category $\bb d$ generated by the commutative diagram below.
\begin{center}
	
	\begin{tikzpicture}[baseline=(current  bounding  box.south), scale=2, hookarrow/.style={{Hooks[right]}->}]
		
		\node (k) at (0.2,0) {$\bullet $};
		\node (a) at (1,0) {$\bullet$};
		\node (b) at (1.8,0) {$\bullet$};
		
		\node (l) at (0.2,-0.6) {$\bullet$};
		\node (c) at (1,-0.6) {$\bullet$};
		\node (e) at (1.8,-0.6) {$\bullet$};
		
		\path[font=\scriptsize]

		([yshift=1.5pt]k.east) edge [->] node [above] {} ([yshift=1.5pt]a.west)
		([yshift=-1.5pt]k.east) edge [->] node [below] {} ([yshift=-1.5pt]a.west)
		
		(a) edge [->] node [above] {} (b)
		
		([yshift=1.5pt]l.east) edge [->] node [above] {} ([yshift=1.5pt]c.west)
		([yshift=-1.5pt]l.east) edge [->] node [below] {} ([yshift=-1.5pt]c.west)
		
		(a) edge [->] node [left] {} (c)
		(c) edge [->] node [right] {} (e)
		(b) edge [->] node [left] {} (e);
	\end{tikzpicture}
\end{center}
Denote by $\iota_t,\iota_b\colon \bb p*0\to \bb d$ the inclusion of the top and bottom co-forks respectively, and by $\iota_{sq}\colon [1]\times [1]\to \bb d$ the inclusion of the square on the right. Then pullback stability of coequalisers of kernel pairs is expressed by the following axiom:\begin{enumerate}
	\item[(vi)]  $(x\co S^{\bb d})\ R_{\tx{pb}}(x\cdot \iota_{sq}),\ R_{\tx{ckp}}(x\cdot \iota_t),\ A_{\tx{kp}}(x\cdot \iota_b)  \vdash R_{\tx{ckp}}(x\cdot \iota_b). $
\end{enumerate}
% This concludes the definition of $\TT_{\tx{reg}}$. Now define $\Reg$ as the 2-category of small regular categories, regular functors, and natural transformations. 

\begin{lemma}  \label{thm:reg-isoregular-theory}
	The theory $\TT_{\tx{reg}}$ is isoregular and 
	$ \Mod(\TT_{\tx{reg}},\Cat)\cong\Reg$.
\end{lemma}
\begin{proof}
	We begin by showing that $\TT_{\tx{reg}}$ is isoregular. First, we already know that $\TT_{\tx{lex}}$ is, and the axioms (i)-(iv) are isoregular by the dual of Lemma~\ref{lemma:b-lim}. The existential quantification in $(x\co S^{\bb p*0})\ A_{\tx{kp}}(x) $ is well-formed since the variable $y\co S^{[1]\times [1]}$ is uniquely determined by its four components (Rule~\ref{equ:epimorphic-family}). 
	
	Finally, the fact that the existential quantification in axiom (v) is well-formed follows the same arguments that we used in Lemma~\ref{lemma:b-lim} to prove that axiom (v) of $\TT_{\bb b}$ is.
	
	As for the 2-category of models, we can identify $\Reg$ as a full subcategory of $\Str(\LL_{\tx{reg}},\Cat)$ by sending a regular category $\C$ to the $\LL_{\tx{lex}}$-structure corresponding to its underlying lex category and  interpreting $R_{\tx{ckp}}$ as the full subcategory spanned by all coequalisers of kernel pairs. It is easy to see that, under this identification, $\Reg$ is closed under powers by $[1]$, so we can apply Proposition~\ref{prop:1-dimensional-is-enough}. Given this, it is immediate from the axioms that a model of $\TT_{\tx{reg}}$ is the same as a regular category.
\end{proof}

\begin{rmk}
The combination of \cref{theo:accessibility-of-models} and \cref{thm:reg-isoregular-theory} gives a new proof that $\Reg$ is accessible with flexible limits, originally shown in~\cite[Section~6.5]{Bou2021:articolo}.
\end{rmk}

Recall that a \emph{Mal'cev category} is a lex category where every reflexive relation is an equivalence relation~\cite{borceux2004mal}. Define $\Mal$ to be the full subcategory of $\Lex$ spanned by the Mal'cev categories, and $\RMal$ to be the full subcategory of $\Reg$ spanned by the regular Mal'cev categories. 

To define a theory whose models are Mal'cev categories, we first construct a theory that identifies all internal relations (that is, jointly monic pairs of arrows) in a category. Starting from the language and theory for lex categories, add a relation symbol $R_{\tx{irel}}:S^{\bb p}$, where $\bb p$ is the free-living pair, which will contain the internal relations; we call the new language $\LL_{\tx{irel}}$. We expand $\TT_{\tx{lex}}$ by adding the axioms:\begin{enumerate}
	\item $(x\co S^{\bb p})\ R_{\tx{irel}}(x) \vdash(\exists y\co S^{[1]\times [1]}) R_{\tx{pb}}(y)\wedge (\prd(y)=\pr_f(x)) \wedge (\prr(y)=\pr_g(x))\wedge$\\ 
	\hspace*{37pt}$(\pru(y)=\id_{\dom (\pru(y))}) \wedge (\prl(y)=\id_{\dom( \prl(y))})$
	\item $(x\co S^{\bb p})\  R_{\tx{pb}}(y)\wedge (\prd(y)=\pr_f(x)) \wedge (\prr(y)=\pr_g(x))\wedge (\pru(y)=\id_{\dom (\pru(y))}) \wedge$\\
	\hspace*{37pt}$ (\prl(y)=\id_{\dom( \prl(y))}) \vdash R_{\tx{irel}}(x);$
	\item $ (z\co S^{[1]\times \bb p})\ R_{\tx{irel}}(\dom(z)), R_{\tx{irel}}(\cod(z))\vdash R_{\tx{irel}}^{[1]}(z) $.
\end{enumerate}
These three axioms together ensure that $R_{\tx{irel}}$ is the full (axiom (iii)) subcategory of $S^{\bb p}$ spanned by all the internal relations (axioms (i) and (ii)). At this point we still have $ \Mod(\TT_{\tx{irel}},\Cat)\cong\Lex $, since we are not adding any new properties.

To construct the theory $\TT_{\tx{Mal}}$ for Mal'cev categories we now need to introduce formulas expressing when a relation is reflexive, symmetric, or transitive. Within $\TT_{\tx{irel}}$, the formula for reflexivity $(x\co S^{\bb p})\ A_{\tx{refl}}(x)$ is defined by
$$
(x\co S^{\bb p})\ (\exists y\co S^{\bb p'})\ R_{\tx{irel}}(x)\wedge (y_f=x_f) \wedge (y_g=x_g),
$$
where $\bb p'$ is the free split pair, obtained by adding a common splitting to the two arrows in $\bb p$. Note that the $y\co S^{\bb p'}$ above is unique since $x$ is an internal relation. The formulas $A_{\tx{sym}}(x)$ and $A_{\tx{tran}}(x)$ are constructed in a similar way.

Thus, we can define $\TT_{\tx{mal}}$ by adding to $\TT_{\tx{irel}}$ the axiom
$$ (x\co S^{\bb p})\ A_{\tx{refl}}(x) \vdash A_{\tx{sym}}(x)\wedge A_{\tx{tran}}(x). $$ 
Similarly, we can consider the isoregular theory $\TT_{\tx{rmal}}\defeq\TT_{\tx{reg}}\cup \TT_{\tx{mal}}$ which models regular Mal'cev categories.

\begin{lemma} \leavevmode
\begin{enumerate}
\item The theory $\TT_{\tx{mal}}$  is isoregular and $\Mod(\TT_{\tx{mal}},\Cat)\cong\Mal.$
\item The theory $\TT_{\tx{rmal}}$  is isoregular and $\Mod(\TT_{\tx{rmal}},\Cat)\cong \RMal$
\end{enumerate}
\end{lemma}
\begin{proof} For part~(i), 
	the existential quantifications in the first and second axioms of $\TT_{\tx{irel}}$ are well-formed since $y\co S^{[1]\times [1]}$ is uniquely determined by its components. Finally, the existential quantification in $(x\co S^{\bb p})\ A_{\tx{refl}}(x)$ is well-formed since, using that $R_{\tx{irel}}(x)$ holds, the $y\co S^{\bb p'}$ is unique (use the formulas defining the universal property of pullbacks).
	Similar arguments apply for $A_{\tx{sym}}(x)$ and  $A_{\tx{tran}}(x)$.
	
	The isomorphism between the 2-category of models and $\Mal$ goes as usual. We can identify $\Mal$ as a full subcategory of $\LL_{\tx{irel}}$-structures by interpreting $R_{\tx{irel}}$ as the full subcategory of all internal relations. Then notice that we can apply Proposition~\ref{prop:1-dimensional-is-enough}.
	
	Part~(ii) is similar.
\end{proof}

A direct application of Corollary~\ref{cor:accessibility-forgetful} and Theorem~\ref{theo:left-biadj-exist} then gives the result below, ensuring the existence of free Mal'cev categories.

\begin{theo} \label{thm:mal} \label{thm:rmal}  The 2-categories $\Mal$ and $\RMal$ are accessible with flexible limits. Moreover, the forgetful 2-functors $\Mal\to \Lex$ and $\RMal \to \Reg$ are accessible and preserve flexible limits, and therefore they have  left biadjoints. 
\end{theo}

Let us write $\Ex$ for the 2-category of (Barr) exact categories, exact functors, and natural transformations. Extending $\TT_{\tx{reg}}$ and $\TT_{\tx{irel}}$, we can define an isoregular theory $\TT_{\tx{ex}}$ such that
\[
 \Mod(\TT_{\tx{ex}},\Cat)\cong\Ex \mathrlap{.} 
 \]
 To do this, start from the language $\LL_{\tx{reg}}\cup \LL_{\tx{irel}}$ and add to $\TT_{\tx{reg}}\cup \TT_{\tx{irel}}$ the axiom
$$ (x\co S^{\bb p})\ A_{\tx{refl}}(x), A_{\tx{sym}}(x), A_{\tx{tran}}(x)\vdash (\exists y\co S^{\bb p*0})\ A_{\tx{kp}}(y)\wedge (\pr_f(y)=\pr_f(x)) \wedge (\pr_g(y)=\pr_g(x) $$
saying that every equivalence relation arises as a (unique up to isomorphism) kernel pair.

The next notion from categorical algebra that we examine is protomodularity. Recall that a category with pullbacks is called {\em protomodular} if for any commutative diagram
	\begin{center}
	\begin{tikzpicture}[baseline=(current  bounding  box.south), scale=1.7]
		
		\node (a) at (0,0.8) {$\bullet$};
		\node (b) at (0.9,0.8) {$\bullet$};
		\node (c) at (1.8,0.8) {$\bullet$};
		
		\node (d) at (0,0) {$\bullet$};
		\node (e) at (0.9,0) {$\bullet$};
		\node (f) at (1.8,0) {$\bullet$};
		
		\node (g) at (0.4,0.4) {$(a)$};
		\node (h) at (1.35,0.4) {$(b)$};

		\path[font=\scriptsize]
		
		(a) edge [->] node [above] {} (b)
		(b) edge [->] node [left] {} (c)
		(d) edge [->] node [right] {} (e)
		(e) edge [->] node [below] {} (f)
		
		(a) edge [->] node [above] {} (d)
		(b) edge [->] node [right] {$t$} (e)
		(c) edge [->] node [above] {} (f)
		
		([xshift=-2pt]e.north) edge [bend left=25, ->] node [left] {$s$} ([xshift=-2pt]b.south);
	\end{tikzpicture}	
\end{center}
where $t\circ s=1$, if (a) and (a+b) are pullback squares, then so is (b) (see~\cite{janelidze2002semi}). We can express this property with an isoregular theory. We denote by $\Ptm$ the 2-category of protomodular categories, pullback-preserving functors, and natural transformations between them.

We start from the theory $\TT_{\tx{pb}}$ of categories with pullbacks. Denote by $\bb{q}$ the category generated by the diagram above, and by $\iota_a,\iota_b,\iota_{a+b}\co [1]\times [1]\to \bb q$ the inclusions of the three squares. Then it is enough to add to $\TT_{\tx{pb}}$ the axiom
$$ (x\co S^{\bb q})\quad R_{\tx{pb}}(x\cdot \iota_a), R_{\tx{pb}}(x\cdot \iota_{a+b})\vdash R_{\tx{pb}}(x\cdot \iota_b). $$
Call $\TT_{\tx{ptm}}$ the new theory we obtain; it is then straightforward to show that

\begin{lemma}
	The theory $\TT_{\tx{ptm}}$  is isoregular and 
	$ \Mod(\TT_{\tx{ptm}},\Cat)\cong\Ptm$.
\end{lemma}
 
As a consequence, we can establish the existence of free protomodular categories over categories of pullbacks. Below, we write $\Pb$ for the 2-category of small categories with pullbacks, pullback-preserving functors, and natural transformations.

\begin{theo} \label{thm:ptm}
	The 2-category $\Ptm$ is accessible with flexible limits. Moreover, the forgetful 2-functor $\Ptm\to \Pb$ is accessible and preserves flexible limits, and therefore has a left biadjoint. 
\end{theo}

\begin{rmk}
By expanding $\TT_{\tx{ptm}}$ with the axioms in $\TT_{\tx{lex}}$ (respectively, $\TT_{\tx{reg}}$ or $\TT_{\tx{ex}}$) we also obtain that the free lex (respectively, regular or exact) protomodular category on a lex (respectively, regular or exact) category exists.
\end{rmk} 

Recall that a category is called {\em semi-abelian} if it is exact, protomodular, has finite coproducts, and has a zero object (see again see~\cite{janelidze2002semi}). Denote by $\SmAb$
the 2-category of semi-abelian categories, exact and finite coproduct-preserving functors, and natural transformations between them.
To construct a theory for semi-abelian category then it is enough to define $\TT_{\tx{smab}}$ as the union of:\begin{itemize}
	\item the theory $\TT_{\tx{ex}}$ of exact categories (which contains in particular a relation $R_\emptyset\co S^{[0]}$ collecting all terminal objects),
	\item the theory $\TT_{\tx{ptm}}$ of protomodular categories,
	\item the theory $\TT_{\tx{fc}}$ of categories with finite coproducts (which contains a relation $R'_\emptyset\co S^{[0]}$ collecting all the initial objects);
	\item the single axiom
	$$ (x\co S^{[0]})\quad R_\emptyset(x)\vdash R'_\emptyset(x) $$
	expressing the fact that every terminal object is also initial.
\end{itemize}   

Given this it is straightforward to show
 
\begin{lemma}
	The theory $\TT_{\tx{smab}}$  is isoregular and 
	$\Mod(\TT_{\tx{ptm}},\Cat)\cong\SmAb$.
\end{lemma}

As a consequence also free semi-abelian categories exist; below $\Fc$ is the 2-category of small categories with finite coproducts, finite-coproduct preserving functors, and natural transformations.

\begin{theo} \label{thm:smab}
	The 2-category $\SmAb$ is accessible with flexible limits. Moreover, the forgetful functors $\SmAb \to \Ptm$, $\SmAb \to \Fc$, $\SmAb \to \Ex$
	are accessible and preserve flexible limits, and therefore they all have left biadjoints. 
\end{theo}

%\begin{center}
%		\begin{tikzpicture}[baseline=(current  bounding  box.south), scale=2]
%			
%			\node (a0) at (0,0.8) {$\SmAb$};
%			\node (b0) at (1,0.8) {$\Ex$};
%			\node (c0) at (0,0) {$\Ptm$};
%			\node (d0) at (0.85,0.12) {$\Fc$};
%			
%			\path[font=\scriptsize]
%			
%			(a0) edge [->] node [above] {} (b0)
%			(a0) edge [->] node [left] {} (c0)
%			(a0) edge [->] node [below] {} (d0);
%		\end{tikzpicture}	
%	\end{center}

\begin{rmk}
	Many more exactness conditions (such as extensive categories and pretoposes, as well as homological categories) can be expressed in this framework. We leave it to the reader to fill in the details of the corresponding isoregular theories.
\end{rmk}

\subsection{Fibrations} \label{sec:fibrations} \label{sec:isofib} We now consider isofibrations and Grothendieck fibrations. Recall that a functor $p\colon \E\to \B$ is an isofibration if and only if for any $x\in\E$ and any isomorphism $u\colon px\to y$ in $\B$ there exists an isomorphism $v\colon x\to x'$ in $\E$ with $pv=u$. We denote by $\IsoFib$ the full subcategory of $\Cat^{[1]}$ spanned by the isofibrations.

Given this, it is straightforward to write down the language and theory for isofibrations. We let $\LL_{\tx{isofib}}$ have two basic sorts $E,B$ and one function symbol $p\colon E\to B$; then $\TT_{\tx{isofib}}$ consists of the single axiom
$$ (x\co E, u\co B^{\bb I})\ \dom(u_0)=p(x) \vdash (\exists v\co E^{\bb I})\ \dom(v_0)=x \wedge p^{[1]}(v_0)=u_0. $$
As before, we denoted by $\bb I$ the category obtained by adding an inverse to the morphism in $[1]$; given $u\co S^{\bb I}$ we denote by $u_0\co S^{[1]}$ the restriction of $u$ along the inclusion. Similarly, we denote by $u_0^{-1}\co S^{[1]} $ the restriction along the inclusion $[1]\to\bb I$ picking the added inverse.

\begin{lemma} \label{thm:isofib-isoreg}
The theory $\TT_{\tx{isofib}}$  is isoregular and 
	$ \Mod(\TT_{\tx{isofib}},\Cat)\cong \IsoFib$.
	\end{lemma} 
	
	\begin{proof}
	We need to show that $\J_\mathsf{ffk}(\dom(v_0)=x \wedge p^{[1]}(v_0)=u_0)$ is derivable to guarantee that the existential quantification is well-formed. For the faithfulness part, we assume to have
	\begin{align*}
		\dom^{[1]}(\tilde v_0)=\tilde x,\ &(p^{[1]})^{[1]}(\tilde v_0)=\tilde u_0,\
		\dom^{[1]}(\tilde v_0')=\tilde x,\ (p^{[1]})^{[1]}(\tilde v_0')=\tilde u_0',\\ 
		&\dom(\tilde v)=\dom(\tilde v'),\ \cod(\tilde v)=\cod(\tilde v') 
	\end{align*}
	in context $(\tilde x\co E^{[1]}, \tilde u\co B^{[1]\times \bb I},\tilde v,\tilde v'\co E^{[1]\times \bb I} )$. Now note that we can describe the components of a $\tilde v\co E^{[1]\times \bb I}$ as below
	\begin{center}
		
		\begin{tikzpicture}[baseline=(current  bounding  box.south), scale=2, hookarrow/.style={{Hooks[right]}->}]
			
			\node (a) at (1,0) {$\bullet$};
			\node (b) at (1,-0.7) {$\bullet$};
			
			\node (c) at (1.9,0) {$\bullet$};
			\node (e) at (1.9,-0.7) {$\bullet$};
			
			\path[font=\scriptsize]

			(a) edge [->] node [left] {$\dom (\tilde v)$} (b)
			(a) edge [->] node [right] {$\cong$} (b)
			(a) edge [->] node [above] {$\dom^{[1]}(\tilde v_0)$} (c)
			(c) edge [->] node [right] {$\cod (\tilde v)$} (e)
			(c) edge [->] node [left] {$\cong$} (e)
			(b) edge [->] node [below] {$\cod^{[1]}(\tilde v_0)$} (e);
		\end{tikzpicture}
	\end{center}
	Since the inclusions $\iota_u,\iota_l,\iota_r\co [1]\to [1]\times \bb I$ (of the upper, left, and right arrows) are jointly epic, it follows from the hypothesis and Rule~\ref{equ:epimorphic-family} that $\tilde v=\tilde v'$.
	
	It remains to prove fullness on identities. The hypotheses are
	$$ (x\co E, u\co B^{\bb I}, v,v'\co E^{\bb I}  )\quad \dom(v_0)=x,\ p^{[1]}(v_0)=u_0,\ \dom(v'_0)=x,\ p^{[1]}(v'_0)=u_0 $$
	and we need to derive that 
	$$ (\exists \tilde v\co E^{[1]\times \bb I})\ \dom^{[1]}(\tilde v_0)=\id_x\wedge  (p^{[1]})^{[1]}(\tilde v_0)=\id_u\wedge \dom(\tilde v)=v\wedge \cod(\tilde v)=v'. $$
	First notice that since $\cod (v_0^{-1})=\dom (v_0)=\dom(v_0')$ we can derive 
	$$ (\exists z\co E^{[2]})\ \tx{fst}(z)=v_0^{-1} \wedge \tx{snd}(z)=v_0';  $$
	then, using Rule~\ref{equ:union} a few times, we deduce
	$$ (\exists \tilde v\co E^{[1]\times \bb I})\  \dom(\tilde v)= v\wedge \cod (\tilde v)=v'\wedge \dom^{[1]}(\tilde v_0)= \id_x \wedge \cod^{[1]}(\tilde v_0)=\comp(z).$$
	It is easy to see that such $\tilde v$ satisfies the formula we want to derive, so that $\TT_{\tx{isofib}}$ is isoregular.
	
	Once more, since isofibrations are stable under powers by $[1]$, it is easy to observe that 
	$$ \Mod(\TT_{\tx{isofib}},\Cat)\cong\IsoFib $$
	using Proposition~\ref{prop:1-dimensional-is-enough}.
\end{proof}

Let us write $\GFib$ for the locally full subcategory of $\Cat^{[1]}$ spanned by the Grothendieck fibrations, and those commutative squares whose top functor preserves the cartesian arrows. 
The language $\LL_{\tx{Gfib}}$ for the theory of Grothendieck fibrations consists of two basic sorts $E,B$, a function symbol $p\co E\to B$, and a relation symbol $\tx{Cart}\rightarrowtail E^{[1]}$ that will collect all cartesian arrows in $E$. Before defining the theory, to set notation consider the finite categories below:
\begin{center}
	
	\begin{tikzpicture}[baseline=(current  bounding  box.south), scale=2, hookarrow/.style={{Hooks[right]}->}]
		
		\node (a') at (-1,0) {$\bullet$};
		\node (b') at (-1,-0.7) {$\bullet$};
		\node (e') at (-0.2,-0.7) {$\bullet$};
		\node (f') at (-1.25,-0.35) {$\bb e=$};

		\node (a) at (1,0) {$\bullet$};
		\node (b) at (1,-0.7) {$\bullet$};
		\node (e) at (1.8,-0.7) {$\bullet$};
		\node (f) at (0.7,-0.35) {$[2]=$};
		
		\node (a'') at (3,0) {$\bullet$};
		\node (b'') at (3,-0.7) {$\bullet$};
		\node (e'') at (3.8,-0.7) {$\bullet$};
		\node (f'') at (2.7,-0.35) {$\bb f=$};
		
		\path[font=\scriptsize]
		
		(a') edge [->] node [right] {comp} (e')
		(b') edge [->] node [below] {snd} (e')
		
		(a) edge [->] node [right] {fst} (b)
		(a) edge [->] node [right] {comp} (e)
		(b) edge [->] node [below] {snd} (e)
		
		(a'') edge [->] node [right] {$\cong$} (b'')
		(a'') edge [->] node [right] {comp} (e'')
		(b'') edge [->] node [below] {snd} (e'');
	\end{tikzpicture}
\end{center}
and denote by $\iota_{\bb e}\co \bb e \to [2]$ the inclusion. 

Now we define $\TT_{\tx{Gfib}}$ to consist of the following axioms:
\begin{enumerate}[itemsep=0.1cm]
	\item \hspace{16pt} $(z,w\co E^{[2]})\quad \tx{Cart}(\tx{snd}(z)),\ z\cdot \iota_{\bb e}=w\cdot \iota_{\bb e},\ p^{[1]}(\tx{fst}(z))=p^{[1]}(\tx{fst}(w))\vdash z=w$; 
	\item $(x\co E^{\bb e}, z\co B^{[2]})\quad \tx{Cart}(\tx{snd}(x)),\ p^{[1]}(\tx{snd}(x))=\tx{snd}(z),\ p^{[1]}(\tx{comp}(x))=\tx{comp}(z)\quad$ \\
	\hspace*{75pt} $\vdash  (\exists w\co E^{[2]})\quad w\cdot\iota_{\bb e}= x\wedge p^{[1]}(\tx{fst}(w))=\tx{fst}(z)\wedge \tx{Cart}(\tx{snd}(w))$;
	\item \hspace{31pt} $ (y\co E^{\bb f})\quad \tx{Cart}(\tx{snd}(y)) \vdash \tx{Cart}(\tx{comp}(y)) $;
	\item \hspace{12pt} $ (u\co E^{[1]\times [1]})\quad \tx{Cart}(\dom(u)), \tx{Cart}(\cod(u))\vdash \tx{Cart}^{[1]}(u) $;
	\item \hspace{3pt} $ (y\co E, u\co B^{[1]})\quad \cod(u)=p(y) \vdash (\exists v\co E^{[1]})\ \tx{Cart}(v)\wedge \cod(v)=y \wedge p^{[1]}(v)=u$.
\end{enumerate}

\begin{lemma} \label{thm:gfib-isoreg}
The theory $\TT_{\tx{Gfib}}$  is isoregular and 
	$ \Mod(\TT_{\tx{Gfib}},\Cat)\cong\GFib. $
	\end{lemma}

\begin{proof}
	Axioms (i) and (ii) say that every map in $\tx{Cart}$ is cartesian (satisfies the unique lift property); the existential quantification in (ii) is well-formed because (i) proves uniqueness. Axiom (iii) says that $\tx{Cart}$ is closed under isomorphisms (hence its objects are all the cartesian maps), while (iv) imposes that $\tx{Cart}$ is a full subcategory of $E^{[1]}$. Finally, (v) asks that every map in $B$ with codomain of the form $p(y)$ has a cartesian lift. The fact that such lift is unique up to unique isomorphism (and hence that the existential quantification is well-formed) is proved as in the case of isofibrations.
	
	As usual, since we already know that Grothendieck fibrations are stable under powers by $[1]$, it is easy to conclude that we have the desired isomorphism of 2-categories.
\end{proof}

\begin{rmk}
Combining \cref{theo:accessibility-of-models} with \cref{thm:isofib-isoreg} and~\cref{thm:gfib-isoreg}, we obtain new proofs that $\IsoFib$ and $\GFib$ are accessible with flexible limits, originally shown in~\cite[Section~8]{Bou2021:articolo}.
\end{rmk}

\subsection{Categorical logic, adjoints, and multicategories} \label{sec:varia} We conclude by applying our results to various 2-categories of interest in categorical logic, 2-categories of adjoint functors, and multicategories.

Recall~\cite{JacobsB:catltt} that a {\em comprehension category} is the data of a commutative triangle
\begin{center}
	\begin{tikzpicture}[baseline=(current  bounding  box.south), scale=2]
		
		\node (a0) at (0,0.8) {$\E$};
		\node (b0) at (1,0.8) {$\B^{[1]}$};
		\node (c0) at (0.5,0.1) {$\B$};
		
		\path[font=\scriptsize]
		
		(a0) edge [->] node [above] {$\Sigma$} (b0)
		(a0) edge [->] node [left] {$p$} (c0)
		(b0) edge [->] node [right] {$\cod$} (c0);
	\end{tikzpicture}	
\end{center}
in $\Cat$, where $\cod$ is the codomain functor, $p$ is a Grothendieck fibration, and $\Sigma$ sends cartesian arrows to pullback squares. Comprehension categories assemble into a 2-category $\Cmp$ whose morphisms are of the form
\begin{center}
	\begin{tikzpicture}[baseline=(current  bounding  box.south), scale=2, on top/.style={preaction={draw=white,-,line width=#1}}, on top/.default=6pt]
		
		\node (a0) at (0,0.8) {$\E$};
		\node (b0) at (1,0.8) {$\B^{[1]}$};
		\node (c0) at (0.5,0.1) {$\B$};
		
		\node (a1) at (1,0.3) {$\F$};
		\node (b1) at (2,0.3) {$\C^{[1]}$};
		\node (c1) at (1.5,-0.4) {$\C$};
		
		\path[font=\scriptsize]
		
		(a0) edge [->] node [above] {$\Sigma$} (b0)
		(a0) edge [->] node [left] {$p$} (c0)
		(b0) edge [->] node [above] {} (c0)
		
		(a1) edge [->] node [above] {$\Delta$} (b1)
		(a1) edge [->] node [left] {$q$} (c1)
		(b1) edge [->] node [right] {$\cod$} (c1)
		
		(a0) edge [->, on top] node [above] {$\ F$} (a1)
		(b0) edge [->] node [above] {$\ \ \ G^{[1]}$} (b1)
		(c0) edge [->] node [below] {$G$} (c1);
	\end{tikzpicture}	
\end{center}
where $F$ preserves the cartesian arrows.
The 2-cells of $\Cmp$ are just 2-cells in $\Cat^{[2]}$ (the 2-category of commutative triangles in $\Cat$).

To define as isoregular theory describing comprehension categories, consider the language $\LL_{\tx{cmp}}$ consisting of two sorts $E,B$, two function symbols $p\co E\to B$ and $\Sigma\co E\to B^{[1]}$, and a relation symbol $\tx{Cart}\rightarrowtail E^{[1]}$. Then $\TT_{\tx{comp}}$ is defined by:\begin{enumerate}
	\item the axiom $ (x\co E)\ \cod(\Sigma(x))=p(x) ;$
	\item all of the axioms in of $\TT_{\tx{Gfib}}$, asserting that $p$ is a Grothendieck fibration and $\tx{Cart}$ collects all cartesian maps;
	\item axioms asserting that if $\tx{Cart}(u)$ holds, then $\Sigma^{[1]}(u)\co B^{[1]\times[1]}$ is a pullback square in $B$; this is done as in Section~\ref{limits-colimits}.
\end{enumerate}
Given what we have shown already about limits and Grothendieck fibrations, the next lemma follows easily.

\begin{lemma}
	The theory $\TT_{\tx{cmp}}$  is isoregular and 
	$ \Mod(\TT_{\tx{cmp}},\Cat)\cong\Cmp. $
\end{lemma}

\begin{theo} \label{thm:comp}
	The 2-category $\Cmp$ is accessible with flexible limits. Moreover, the forgetful 2-functor $\Cmp\to \GFib$ is accessible and preserves flexible limits, and therefore has a left biadjoint.
\end{theo}

\begin{rmk}
	In a slightly different context, the free comprehension category over a Grothendieck fibration has been constructed explicitly in~\cite[Section~3.3]{giusto2024fibrazioni}. What differs is that their 2-category of comprehension categories a lax version of our notion of morphism.
\end{rmk}

Recall~\cite{JoyalA:notct,TaylorP:recdic} that a {\em clan} is a category $\C$ together with a set of morphisms $\tx{Fib}\subseteq \C^{[1]}$ (which we see as a full subcategory) such that:\begin{itemize}
	\item $\C$ has a terminal object $1$;
	\item pullbacks along maps in $\tx{Fib}$ exist;
	\item every isomorphism and every map with codomain $1$ lies in $\tx{Fib}$;
	\item the elements of $\tx{Fib}$ are closed under composition and stable under pullback.
\end{itemize}
Clans assemble in a 2-category $\Clan$ whose morphisms are functors preserving the class of fibrations as well as the terminal object and pullbacks along fibrations; 2-cells are just natural transformations.

The language $\LL_{\tx{clan}}$ for clans consists of a single sort $S$, a relation symbol $\tx{Fib}\rightarrowtail S^{[1]}$ for the set of fibrations, a relation symbol $R_\emptyset\rightarrowtail S$ to collect the terminal objects, and a relation symbol $R_{\tx{pb}}\rightarrowtail S^{[1]\times [1]}$ collecting all the pullbacks along maps in $\tx{Fib}$. The isoregular theory $\TT_{\tx{clan}}$ is then formed by:\begin{enumerate}
	\item the five axioms of Section~\ref{limits-colimits} for $\bb b=\emptyset$, saying that $S$ has a terminal object and $R_\emptyset$ collects all of them;
	\item the first four axioms of Section~\ref{limits-colimits} for $\bb b$ the free cospan (with inclusion $\iota\co \bb b\to [1]\times [1]$), saying that $R_{\tx{pb}}$ consists of pullback squares, is closed under isomorphisms, and is a full subcategory of $S^{[1]\times [1]}$;
	\item the axiom (similar to (v) in Section~\ref{limits-colimits})
	$$ (x\co S^{\bb b})\  \tx{Fib}(\prr(x))\vdash (\exists y\co S^{[1]\times [1]})\ R_{\tx{pb}}(y)\wedge (y\cdot \iota= x) $$
	for $\bb b$ the free co-span, saying that pullbacks along maps in $\tx{Fib}$ exists and are in $R_{\tx{pb}}$;
	\item the axiom
	$$ (z\co S^{[1]\times [1]})\ \tx{Fib}(\dom(z)), \tx{Fib}(\cod(z))\vdash \tx{Fib}^{[1]}(z) $$
	saying that $\tx{Fib}$ defines a full subcategory;
	\item the axiom
	$$ (u\co S^{\bb I}) \vdash \tx{Fib}(u_0) $$
	saying that $\tx{Fib}$ contains all isomorphism (we follow the notation used in Section~\ref{sec:isofib});
	\item the axiom
	$$ (u\co S^{[2]})\  \tx{Fib}(\tx{fst}(u)), \tx{Fib}(\tx{snd}(u))\vdash \tx{Fib}(\comp(u)) $$
	saying that $\tx{Fib}$ is closed under composition;
	\item the axiom
	$$ (z\co S^{[1]\times [1]})\ R_{\tx{pb}}(z)\vdash \tx{Fib}(\prl(z))\wedge \tx{Fib}(\prr(z))  $$
	saying that the vertical legs of every square in $R_{\tx{pb}}$ lie in $\tx{Fib}$. Together with (iii) this implies that $\tx{Fib}$ is stable under pullbacks.
\end{enumerate}

\begin{lemma}
	The theory $\TT_{\tx{clan}}$  is isoregular and 
	$ \Mod(\TT_{\tx{clan}},\Cat)\cong\Clan. $
\end{lemma}
\begin{proof}
	The only existential quantifications in $\TT_{\tx{clan}}$ appear to give the universal property of a limit or to say that a limit exists, so they are well-formed by the same arguments of Lemma~\ref{lemma:b-lim}; thus $\TT_{\tx{clan}}$ is isoregular.
	
	Given this, a model $M$ of $\TT_{\tx{clan}}$ is determined by the category $\br{S}$ and the relation $\br{\tx{Fib}}$, which is a full subcategory of $\br{S}^{[1]}$ by (iv). Indeed, then $\br{R_{\emptyset}}$ is forced to be the full subcategory of all terminal objects in $\br{S}$, and $\br{R_{\tx{pb}}}$ the full subcategory of all pullback squares with vertical maps in $\br{\tx{Fib}}$. Then $\br{\tx{Fib}}$ satisfies the conditions required to form a clan by (v)-(vii). One concludes as usual by using that clans are stable under powers by $[1]$.
\end{proof}

\begin{theo} \label{thm:clan}
	The 2-category $\Clan$ is accessible with flexible limits. Moreover, the forgetful 2-functor $\Clan\to \Cat$ is accessible, preserves flexible limits, and therefore has a left biadjoint.
\end{theo}

	\begin{rmk}
		One could also give an isoregular theory for small categories with a terminal object and a natural number object, functors preserving them, and natural transformations.
	\end{rmk}

For an example of a slightly different nature, let us consider $\Radj$,  the locally full sub-2-category of $\Cat^{[1]}$ whose objects are right adjoint functors $R\co\C\to\D$, and morphisms are squares
\begin{center}
	\begin{tikzpicture}[baseline=(current  bounding  box.south), scale=2]
		
		\node (a0) at (0,0.7) {$\C$};
		\node (b0) at (0.8,0.7) {$\C'$};
		\node (c0) at (0,0) {$\D$};
		\node (d0) at (0.8,0) {$\D'$};
		
		\path[font=\scriptsize]
		
		(a0) edge [->] node [above] {$F$} (b0)
		(a0) edge [->] node [left] {$R$} (c0)
		(b0) edge [->] node [right] {$R'$} (d0)
		(c0) edge [->] node [below] {$G$} (d0);
	\end{tikzpicture}	
\end{center}
were $F$ preserves universal arrows; that is, if $\eta$ is a universal arrow from $c\in \C$ to $R$, then $F\eta$ is a universal arrow from $Fc$ to $R'$. This is equivalent to requiring that $F$ preserves the units of the adjunction up to isomorphism.

At this point the reader will probably guess how to construct an isoregular theory that models $\Radj$. The idea is that, beside having a function symbol $R\co C\to D$ between two sorts, one also adds a relation symbol $\tx{Univ}\rightarrowtail C^{[1]}$ to collect (via given axioms) all universal arrows; \ie all units of the adjunction. Then we add an axiom stating that for any object of $C$ there exists an element of $\tx{Univ}$ (which is unique up to isomorphism) with domain the given object. When taking models in $\Cat$ this is equivalent to asking the existence of a left adjoint to $R$; since morphisms need to respect the relation symbol, the 2-category of models will be the same as $\Radj$.

The same arguments apply for the 2-category $\Ladj$ whose objects are left adjoint functors $L\co\C\to\D$, and where the morphism preserve universal arrows from $L$ to $d\in\D$.

As our final example consider the 2-category $\Mult$ of multicategories, multifunctors, and multinatural transformations. To find an isoregular theory presenting $\Mult$ we see the data of a multicategory $\M$ as:\begin{itemize}
	\item a category $M_n$ of $n$-ary morphisms and multisquares between them, for any $n\geq 0$, where $0$-ary morphisms are the objects;
	\item unit $1\co M_0\to M_1$, source $s_n\co M_n\to(M_0)^n$, and target $t_n\co M_n\to M_0$ functors (for any $n> 0$);
	\item subcategories $\tx{Comp}_{n}^{m_1,\cdots,m_n}\rightarrowtail M_n\times M_{m_1}\times\cdots\times  M_{m_n}\times M_{m_1+\cdots m_n}$ to encode that\\ $\tx{Comp}_{n}^{m_1,\cdots,m_n}(f,g_1,\cdots,g_n,h)$ holds if and only if $f\circ (g_1,\cdots,g_n)$ is well-typed and coincides with $h$.
\end{itemize}
All of this data identifies a language $\LL$. The existence of composites is expressed via unique existential quantification and is subject to axioms expressing associativity and unitality. Finally one also needs to add axioms saying that morphisms in $M_n$ are just compatible $(n+1)$-tuples of objects from $M_1$. This suffices to decribe $\Mult$.

From this, we can add relation symbols $\tx{Rep_n}\rightarrowtail M_n$, for any $n\geq 1$, collecting (via specific axioms) all strong universal arrows $\eta:(x_1,\cdots, x_n)\to \hat x$ in the sense of~\cite[Definition~8.3]{hermida2000representable}. If we further add an axiom saying that for any $n$-tuple of objects $(x_1,\cdots, x_n)$ there is a (unique up to isomorphism) strong universal arrow with source $(x_1,\cdots, x_n)$, we obtain an isoregular theory that present the 2-category $\RMult$ of representable multicategories, multifunctors that preserve strong universal arrows, and multinatural transformations.

Since $\RMult$ is 2-equivalent to the 2-category $\MonCat$ of monoidal categories, strong monoidal functors, and monoidal natural transformations, we obtain a new proof that $\MonCat$ is accessible with flexible limits, as shown in~\cite[Section~6.1]{Bou2021:articolo}.

\section{Directions for future work}

The paper suggests several promising ideas, including possible generalisations of our results which we expect to be true, whose investigation we leave for future work.

\subsubsection*{Internal languages}

Given a small isoregular 2-category $\C$, we expect to be possible do define an isoregular theory $\TT_\C$ such that
$$\IsoReg(\C,\K)\simeq \Mod(\TT_\C,\K) \mathrlap{,} $$
for any isoregular 2-category $\K$. This would strengthen the view of isoregular 2-categories as the semantical 
counterpart of isoregular theories.

\subsubsection*{Barr's embedding}

Given a small isoregular 2-category $\C$, a natural question that arises is whether the evaluation 2-functor
$$ \tx{ev}\co \C\longrightarrow [\IsoReg(\C,\Cat),\Cat]$$
is fully faithful, generalising the ordinary embedding for regular categories proved by Barr~\cite{Bar86:articolo}. We believe that an adaptation of the arguments in~\cite{LT20:articolo} could work in this context. 

Such a theorem would be very important for at least two reasons. On one hand, it would imply a completeness theorem for isoregular logic relative to models in $\Cat$. On the other, it would mean that isoregular categories could be captured in the context of lex-colimits~\cite{GL12:articolo}, and all of the theory developed in that paper would apply.

\subsubsection*{Makkai's conceptual completeness}

Assuming to have proved an analogue of Barr's embedding as above, then is easy to see that the evaluation 2-functor restricts to 
$$ J\co \C\longrightarrow \tx{FlexFilt}(\IsoReg(\C,\Cat),\Cat)$$
where the codomain is the full subcategory of $[\IsoReg(\C,\Cat),\Cat]$ spanned by those 2-functors that preserve flexible limits and filtered colimits. Is it possible to add exactness condition on $\C$ to enforce that $J$ is a 2-equivalence? This would give a 2-dimensional version of Makkai's theorem for Barr-exact categories~\cite{Mak90:articolo}.

\subsubsection*{Infinitary case}

Let $\lambda$ be a regular cardinal. One can then define ``$\lambda$-isoregular'' logic by allowing arities from $\Cat_\lambda$ (the full subcategory of $\Cat$ spanned by the $\lambda$-presentable objects) and $\lambda$-small conjunctions. 
Semantically this corresponds to asking an isoregular 2-category also to have all $\lambda$-small products (and hence all $\lambda$-small limits), plus the requirement that fully faithful regular epimorphisms are stable under $\lambda$-small products.

Then all the results of the paper, and possibly those envisaged in this section, should extend to the infinitary setting without major efforts. As a consequence, one would be able to characterise accessible 2-categories with flexible limits exactly as the 2-categories of models of infinitary isoregular theories in $\Cat$.

\appendix

\section{Deduction rules for isoregular theories} 
\label{sec:deduction-rules}

As in 1-dimensional first-order logic, we assume the standard structural and equality rules (see for instance \cite[D1.3.1,(a)--(b)]{Joh02:libro}). We also assume two rules expressing unitality and associativity of term substitution, which we do not spell out. All the other rules of isoregular logic are listed below.

\subsection*{Deduction rules for conjunction} 
\label{sec:deduction-conjunction}

\begin{drule}   \label{equ:conj} 
\[
\begin{prooftree}
A_1, A_2 \vdash B 
\justifies
A_1 \land A_2 \vdash B
\end{prooftree} \qquad
\begin{prooftree}
 A_1, \ldots, A_n \vdash B_1  \quad
A_1, \ldots, A_n \vdash B_2 
 \justifies
A_1, \ldots, A_n \vdash B_1 \land B_2
\end{prooftree}
\]
\end{drule}

\subsection*{Deduction rules for existential quantifiers} 
\label{sec:deduction-existential}

\begin{drule}
\[
\begin{gathered} 
\begin{prooftree}
(\bar{X}) \quad (\exists x \co X) A \co \mathsf{prop}  \qquad
(\bar{X}) \quad s \co X \qquad 
(\bar{X}) \quad A_1, \ldots, A_n \vdash A[s/x]  
\justifies
(\bar{X}) \quad A_1, \ldots, A_n \vdash (\exists x \co X) A
\end{prooftree} \medskip \\ 
\begin{prooftree}
(\bar{X}, x \co X) \quad A(x) \vdash B
\justifies
(\bar{X})  \quad (\exists x \co X) A(x) \vdash B
\end{prooftree}
\end{gathered}
\]
\end{drule}
\begin{drule}[Frobenius rule] \label{equ:frobenius} 
\[
\begin{prooftree}
A_1, \ldots, A_n \vdash A \qquad 
A_1, \ldots, A_n \vdash (\exists y \co Y) B
\justifies
A_1, \ldots, A_n \vdash (\exists y\co Y)( A \land B)
\end{prooftree}
\]
\end{drule}

\subsection*{Terminal object rule} 
\label{sec:deduction-finite-products}

\begin{drule} \label{equ:initial} 
\[	
	\begin{prooftree}
		s \co S^{0} \quad t \co S^{0}
		\justifies
		s=t
	\end{prooftree}
\]
\end{drule}

\subsection*{Deduction rules for restriction} 
\label{sec:deduction-restriction}

\begin{drule}  \label{equ:action} For functors $f \co \bb d \to \bb c$ and $g \co \bb e \to \bb d$, 
\[
\begin{prooftree}
s \co S^{\bb c}
\justifies
(s \cdot f) \cdot g = s \cdot (f \circ g)
\end{prooftree} \qquad\qquad
\begin{prooftree}
s \co S^{\bb c}
\justifies
s \cdot 1_{\bb c}  = s 
\end{prooftree}
\]
\end{drule}
\begin{drule} \label{equ:epimorphic-family} For a jointly epimorphic family of functors
$(f_i \co \bb d \to \bb c)_{1 \leq i \leq n}$, 
\[
\begin{prooftree}
s \co S^{\bb c} \quad
t \co S^{\bb c} \qquad
s \cdot f_1 = t \cdot f_1 \quad
\ldots \quad
s \cdot f_n = t \cdot f_n 
\justifies
s = t
\end{prooftree}
\]
\end{drule}
\begin{drule} \label{equ:union} For every pushout diagram
\[
\begin{tikzcd}
\bb b \ar[r, "i_2"] \ar[d, "i_1"'] & \bb c_2 \ar[d, "j_2"] \\
\bb c_1 \ar[r, "j_1"'] & \bb d  \mathrlap{,}
\end{tikzcd}
\]
the rule
\[
\begin{prooftree}
s_1 \co S^{\bb c_1} \quad
s_2 \co S^{\bb c_2} \quad
s_1 \cdot i_1 = s_2 \cdot i_2 
\justifies
(\exists x \co S^{\bb d}) \big( x \cdot j_1 = s_1 \land x \cdot j_2 = s_2 \big)
\end{prooftree}
\]
The existential quantifier in the conclusion can be formed since  $j_1$, $j_2$ are jointly epimorphic.

\end{drule} 
\begin{drule} \label{equ:coeq} For $q\colon \bb b\to \bb c$ the coequaliser of $f,g\co \bb a\to \bb b$, the rule
\[
	\begin{prooftree}
		s \co S^{\bb b} \quad
		s \cdot f = s \cdot g 
		\justifies
		(\exists x \co S^{\bb c}) \big( x \cdot q = s \big)
	\end{prooftree}
\]
The existential quantifier in the conclusion can be formed since  $q$ is an epimorphism.
\end{drule}

\subsection*{Deduction rules for powers.}
\label{sec:deduction-powers}

\begin{drule} For a functor $f \co \bb d \to \bb c$
	\begin{equation*}
		\label{equ:powers-restriction} 
		\begin{prooftree}
			s \co S^{\bb c}
			\justifies
			(s\cdot f)^{[1]} = s^{[1]}\cdot ([1]\times f) (s) 
		\end{prooftree}
	\end{equation*}
\end{drule}

\begin{drule} 
	\begin{equation*}\label{equ:comp}
		\begin{prooftree}
			(\bar{x} \co \bar{X})\  
			\justifies
			(\bar{u} \co \bar{X}^{[1]})\  \var_{\bar X, i}^{[1]}(\bar u) = \var_{\bar X^{[1]}, i}(\bar u)
		\end{prooftree}\qquad
		\begin{prooftree}
			(\bar x\co \bar{X})\ s  \co X \qquad
			(\Delta, x \co X) \ t \co Y
			\justifies
			\big( t[s/x] \big)^{[1]} = t^{[1]} [ s^{[1]} / u ] 
		\end{prooftree}
	\end{equation*}
\end{drule}

\begin{drule} 
	\begin{equation*}
		\label{equ:naturality-dom-cod-id}
		\begin{gathered}
			\begin{prooftree}
				(\bar x\co \bar X)\ s(\bar x) \co S^{\bb c} 
				\justifies
				(\bar x\co \bar X)\ s^{[1]}(\id_{\bar x})=\id_{s(\bar x)}
			\end{prooftree} \medskip \\
			\begin{prooftree}
				(\bar x\co \bar X)\ s(\bar x) \co S^{\bb c} 
				\justifies
				(\bar u\co \bar X^{[1]})\ \dom (s^{[1]}(\bar u))=s(\dom(\bar u))
			\end{prooftree} \qquad
			\begin{prooftree}
				(\bar x\co \bar X)\ s(\bar x) \co S^{\bb c} 
				\justifies
				(\bar u\co \bar X^{[1]})\ \cod (s^{[1]}(\bar u))=s(\cod(\bar u))
			\end{prooftree}
		\end{gathered}
	\end{equation*}
\end{drule}

\begin{drule} 
\begin{equation*}
	\begin{prooftree}
		s \co X \qquad
		t \co X 
		\justifies
		(s = t)^{[1]} \vdash s^{[1]}=t^{[1]}
	\end{prooftree} \qquad
	\begin{prooftree}
		s \co X \qquad
		t \co X 
		\justifies
		s^{[1]}=t^{[1]} \vdash (s = t)^{[1]}  
	\end{prooftree}
\end{equation*}
\end{drule}
\begin{drule} 
\begin{equation*}
	\begin{prooftree}
		s_1 \co X_1 \quad
		\ldots \quad
		s_n \co X_n 
		\justifies
		R(s_1, \ldots, s_n)^{[1]} \vdash 	R^{[1]}(s^{[1]}_1, \ldots, s^{[1]}_n)
	\end{prooftree}
	\qquad
	\begin{prooftree}
		s_1 \co X_1 \quad
		\ldots \quad
		s_n \co X_n 
		\justifies
		R^{[1]}(s^{[1]}_1, \ldots, s^{[1]}_n) \vdash R(s_1, \ldots, s_n)^{[1]} 	
	\end{prooftree}
\end{equation*}
\end{drule}

\begin{drule} \label{equ:power-functoriality-formulas}
	\[
	\begin{prooftree}
		(\bar{x} \co \bar{X}) \quad A_1, \ldots, A_n \vdash A 
		\justifies
		(\bar{x} \co \bar{X})^{[1]} \quad A_1^{[1]}, \ldots, A_n^{[1]} \vdash A^{[1]}
	\end{prooftree}
	\]
\end{drule}
\begin{drule} \label{equ:power-functoriality-terms}
\[
\begin{prooftree}
s \co S^{\bb c}
\justifies
A(s)  \vdash  A^{[1]}(\id_s)
\end{prooftree}
\qquad
\begin{prooftree} 
s \co (S^{\bb c})^{[2]}
\justifies
A^{[1]}( \mathsf{fst}(s) ),  A^{[1]}( \mathsf{snd}(s) ) \vdash A^{[1]}( \mathsf{comp}(s) ) 
\end{prooftree}
\]
\end{drule}
\medskip 
\begin{drule} \label{equ:power-dom-cod-prop}
\[
\begin{prooftree}
s \co (S^{\bb c })^{[1]}
\justifies
A^{[1]}(s)  \vdash A( \dom(s) ) 
\end{prooftree} \qquad
\begin{prooftree}
s \co (S^{\bb c })^{[1]}
\justifies
A^{[1]}(s)  \vdash A( \cod(s) ) 
\end{prooftree}
\]
\end{drule}
\medskip
\begin{drule}\label{equ:power-squares}
\[
\begin{prooftree}
s \co (S^{\bb c})^{[1] \times [1]} 
\justifies
A^{[1]}( \pru(s) ), 
A^{[1]}( \prd(s) ), 
A^{[1]}( \prl(s) ), 
A^{[1]}( \prr(s) ) 
\vdash 
A^{[1] \times [1]}( s ) 
\end{prooftree}
\]
Here, $A^{[1] \times [1]}\defeq (A^{[1]})^{[1]}$ with the usual convention that the first $[1]$ on the left-hand side corresponds to the most external $[1]$ on the right-hand side.
\end{drule}
\begin{drule} 
\[
\begin{prooftree}
		(\bar{x} \co \bar{X}) \quad A_1, \ldots A_n \vdash (\exists x \co X) A(x) 
		\justifies
		(\bar{x} \co \bar{X})^{[1]}  \quad A_1^{[1]}, \ldots A_n^{[1]} \vdash (\exists u \co X^{[1]}) A^{[1]}(u)    \mathrlap{.}
		% \thickness=0.05em
	\end{prooftree}
\]
\end{drule}

\bibliography{biblio}

@article{UemuraT:genfst,
	author = {T. Uemura},
	journal = {Mathematical Structures in Computer Science},
	pages = {134--179},
	title = {A general framework for the semantics of type theory},
	volume = {33},
	year = {2023}}

@book{JacobsB:catltt,
	author = {B. Jacobs},
	publisher = {Elsevier},
	title = {Categorical logic and type theory},
	year = {1999}}

@phdthesis{TaylorP:recdic,
	author = {P. Taylor},
	school = {University of Cambridge},
	title = {Recursive domains, indexed category theory and polymorphism},
	year = {1987}}

@misc{JoyalA:notct,
	author = {A. Joyal},
	howpublished = {arXiv:1710.10238},
	title = {Notes on clans and tribes},
	year = {2017}}

@incollection{GranM:intrc,
	author = {M. Gran},
	booktitle = {New Perspectives in Algebra, Topology and Categories},
	editor = {M. M. Clementino and A. Facchini and M. Gran},
	pages = {113--145},
	publisher = {Springer},
	title = {An introduction to regular categories},
	year = {2021}}

@article{giusto2024fibrazioni,
	author = {Giusto, A.},
	journal = {Master Thesis, Universit\`a degli {S}tudi di {G}enova},
	title = {Fibrations with comprehensions and their completions},
	year = {2024}}

@article{AczelP:gentti,
	author = {P. Aczel and N. Gambino},
	journal = {Journal of Symbolic Logic},
	number = {1},
	pages = {67--103},
	title = {The generalised type-theoretic intepretation of constructive set theory},
	volume = {71},
	year = {2006}}

@incollection{nordstrom-petersson-smith:ml,
	author = {B. Nordstr\"om and K. Petersson and J. Smith},
	booktitle = {Handbook of Logic in Computer Science},
	editor = {S. Abramsky and D. M. Gabbay and T. S. E. Maibaum},
	pages = {1-37},
	publisher = {Oxford University Press},
	title = {Martin-{L}\"of type theory},
	volume = {5},
	year = {2001}}

@article{MakkaiM:gensfc-II,
	author = {M. Makkai},
	journal = {Journal of Pure and Applied Algebra},
	number = {2},
	pages = {197--212},
	title = {Generalised sketches as a framework for completeness theorem. Part~{II}},
	volume = {115},
	year = {1997}}

@article{MakkaiM:gensfc-I,
	author = {M. Makkai},
	journal = {Journal of Pure and Applied Algebra},
	number = {1},
	pages = {49--79},
	title = {Generalised sketches as a framework for completeness theorem. Part {I}},
	volume = {115},
	year = {1997}}

@article{StreetR:twodst,
	author = {R. Street},
	journal = {Journal of Pure and Applied Algebra},
	pages = {251--270},
	title = {Two-dimensional sheaf theory},
	volume = {23},
	year = {1982}}

@article{VoevodskyV:explfm,
	author = {V. Voevodsky},
	journal = {Mathematical Structures in Computer Science},
	number = {5},
	pages = {1278--1294},
	title = {An experimental library of formalised {M}athematics based on the univalent foundations},
	volume = {25},
	year = {2015}}

@book{HoTT-Book,
	address = {Institute for Advanced Study},
	author = {The {Univalent Foundations Program}},
	publisher = {\url{https://homotopytypetheory.org/book}},
	title = {Homotopy Type Theory: Univalent Foundations of Mathematics},
	year = {2013}}

@article{LackS:homtam,
	author = {S. Lack},
	journal = {Journal of Homotopy and Related Structures},
	number = {2},
	pages = {229--260},
	title = {Homotopy-theoretic aspects of $2$-monads},
	volume = {2},
	year = {2007}}

@article{KellyGM:mongcl,
	author = {G. M. Kelly and I. Le Creurer},
	journal = {Cahiers de Topologie et G\'{e}om\'{e}trie Diff\'{e}rentielle},
	pages = {179--191},
	title = {On the monadicity over graphs of categories with limits},
	volume = {38},
	year = {1997}}

@article{KellyGM:monccc,
	author = {G. M. Kelly and S. Lack},
	journal = {Theory and Applications of Categories},
	number = {7},
	pages = {148--170},
	title = {On the monadicity of categories with chosen colimits},
	volume = {7},
	year = {2000}}

@article{DiLibertiI:biabp2c,
	author = {I. Di Liberti and A. Osmond},
	journal = {Applied Categorical Structures},
	number = {3},
	pages = {1--64},
	title = {Bi-accessible and bipresentable 2-categories},
	volume = {33},
	year = {2025}}

@article{KockA:monwsa,
	author = {A. Kock},
	journal = {Journal of Pure and Applied Algebra},
	pages = {41 -- 53},
	title = {Monads whose structures are adjoint to units},
	volume = {104},
	year = {1993}}

@article{KellyG:prolt,
	author = {G. M. Kelly and S. Lack},
	journal = {Theory and Applications of Categories},
	number = {9},
	pages = {213 -- 250},
	title = {On property-like structures},
	volume = {3},
	year = {1997}}

@article{BlackwellR:twodmt,
	author = {R. Blackwell and G. M. Kelly and A. J. Power},
	journal = {Journal of Pure and Applied Algebra},
	pages = {1--41},
	title = {Two-dimensional monad theory},
	volume = {59},
	year = {1989}}

@article{KellyGM:eleo2c,
	author = {G. M. Kelly},
	journal = {Bulletin of the Australian Mathematical Society},
	pages = {301--317},
	title = {Elementary observations on 2-categorical limits},
	volume = {39},
	year = {1989}}

@book{Joh02:libro,
	author = {Johnstone, P.},
	publisher = {Oxford Logic Guides},
	title = {Sketches of an Elephant: A Topos Theory Compendium},
	year = {2002}}

@book{AR94:libro,
	author = {Ad\'{a}mek, J. and Rosick\'{y}, J.},
	doi = {10.1017/CBO9780511600579},
	isbn = {0-521-42261-2},
	mrclass = {18Axx (18-02)},
	mrnumber = {1294136},
	mrreviewer = {J. R. Isbell},
	pages = {xiv+316},
	publisher = {Cambridge University Press, Cambridge},
	series = {London Mathematical Society Lecture Note Series},
	title = {Locally presentable and accessible categories},
	url = {https://doi.org/10.1017/CBO9780511600579},
	volume = {189},
	year = {1994}}

@book{MP89:libro,
	author = {Makkai, M. and Par\'e, R.},
	publisher = {Springer},
	title = {Accessible Categories: The Foundations of Categorical Model Theory},
	volume = {104},
	year = {1989}}

@book{Bor94:libro,
	author = {Borceux, F.},
	publisher = {Cambridge University Press},
	title = {Handbook of Categorical Algebra: Volume 2, Categories and Structures},
	year = {1994}}

@book{Lur09:libro,
	author = {Lurie, J.},
	doi = {10.1515/9781400830558},
	isbn = {978-0-691-14049-0; 0-691-14049-9},
	mrclass = {18-02 (18B25 18E35 18G30 18G55 55U40)},
	mrnumber = {2522659},
	mrreviewer = {Mark Hovey},
	pages = {xviii+925},
	publisher = {Princeton University Press, Princeton, NJ},
	series = {Annals of Mathematics Studies},
	title = {Higher topos theory},
	url = {https://doi.org/10.1515/9781400830558},
	volume = {170},
	year = {2009}}

@incollection{lack20102,
	author = {Lack, S.},
	booktitle = {Towards higher categories},
	pages = {105--191},
	publisher = {Springer},
	title = {A 2-categories companion},
	year = {2010}}

@book{Law63:articolo,
	author = {Lawvere, F. W.},
	publisher = {Ph.D. thesis Columbia University},
	title = {Functorial Semantics of Algebraic Theories},
	year = {1963}}

@book{ARValgebraic,
	author = {Ad{\'a}mek, J. and Rosick{\`y}, J. and Vitale, E. M.},
	publisher = {Cambridge University Press},
	title = {Algebraic theories: a categorical introduction to general algebra},
	volume = {184},
	year = {2010}}

@article{Mak90:articolo,
	author = {Makkai, M.},
	journal = {Annals of Pure and Applied Logic 47},
	pages = {225-268},
	title = {A Theorem on {B}arr-Exact Categories, with an Infinitary Generalization},
	year = {1990}}

@article{Bar86:articolo,
	author = {Barr, M.},
	journal = {Journal of Pure and Applied Algebra},
	pages = {113-137},
	title = {Representation of Categories},
	volume = {41},
	year = {1986}}

@misc{Lur18:articolo,
	author = {Lurie, J.},
	howpublished = {Available from \textnormal{\url{https://www.math.ias.edu/~lurie/papers/Conceptual.pdf}}},
	title = {Ultracategories},
	year = {2018}}

@article{LR12:articolo,
	author = {Lack, S. and Rosick\'{y}, J.},
	journal = {Journal of Pure and Applied Algebra},
	pages = {1807-1822},
	title = {Enriched weakness},
	volume = {216},
	year = {2012}}

@article{GL12:articolo,
	author = {Garner, R. and Lack, S.},
	doi = {10.1016/j.jpaa.2012.01.003},
	issn = {0022-4049},
	journal = {Journal of Pure and Applied Algebra},
	mrclass = {18A35 (18B25 18D05 18E10)},
	mrnumber = {2890508},
	mrreviewer = {Josep Elgueta},
	number = {6},
	pages = {1372--1396},
	title = {Lex colimits},
	url = {https://doi.org/10.1016/j.jpaa.2012.01.003},
	volume = {216},
	year = {2012}}

@article{Mak87:articolo,
	author = {Makkai, M.},
	journal = {Advances in Mathematics},
	pages = {97-170},
	title = {Stone Duality for First Order Logic},
	volume = {65},
	year = {1987}}

@article{LT20:articolo,
	author = {Lack, S. and Tendas, G.},
	doi = {https://doi.org/10.1016/j.jpaa.2019.106268},
	issn = {0022-4049},
	journal = {Journal of Pure and Applied Algebra},
	number = {6},
	pages = {106268},
	title = {Enriched regular theories},
	url = {http://www.sciencedirect.com/science/article/pii/S0022404919302816},
	volume = {224},
	year = {2020}}

@article{BLV:articolo,
	author = {Bourke, J. and Lack, S. and Vok{\v{r}}{\'\i}nek, L.},
	journal = {Advances in Mathematics},
	pages = {108812},
	publisher = {Elsevier},
	title = {Adjoint functor theorems for homotopically enriched categories},
	volume = {412},
	year = {2023}}

@article{BQ96:articolo,
	author = {Borceux, F. and Quinteiro, C.},
	journal = {Bulletin of the Australian Mathematical Society},
	number = {3},
	pages = {489--501},
	publisher = {Cambridge University Press},
	title = {Enriched accessible categories},
	volume = {54},
	year = {1996}}

@article{BG14:articolo,
	author = {Bourke, J. and Garner, R.},
	issn = {0022-4049},
	journal = {Journal of Pure and Applied Algebra},
	number = {7},
	pages = {1346--1371},
	title = {Two-dimensional regularity and exactness},
	volume = {218},
	year = {2014}}

@article{BKPS89:articolo,
	author = {Bird, G.J. and Kelly, G.M. and Power, A.J. and Street, R.H.},
	doi = {https://doi.org/10.1016/0022-4049(89)90065-0},
	issn = {0022-4049},
	journal = {Journal of Pure and Applied Algebra},
	number = {1},
	pages = {1 - 27},
	title = {Flexible limits for 2-categories},
	volume = {61},
	year = {1989}}

@article{Bou2021:articolo,
	author = {J. Bourke},
	doi = {https://doi.org/10.1016/j.jpaa.2020.106519},
	issn = {0022-4049},
	journal = {Journal of Pure and Applied Algebra},
	number = {3},
	pages = {106519},
	title = {Accessible aspects of 2-category theory},
	url = {https://www.sciencedirect.com/science/article/pii/S0022404920302206},
	volume = {225},
	year = {2021}}

@article{LT21:articolo,
	author = {Lack, S. and Tendas, G.},
	journal = {Advances in Mathematics},
	pages = {108381},
	publisher = {Elsevier},
	title = {Flat vs. filtered colimits in the enriched context},
	volume = {404},
	year = {2022}}

@article{LT22:virtual,
	author = {Lack, S. and Tendas, G.},
	doi = {https://doi.org/10.1016/j.jpaa.2022.107196},
	journal = {Journal of Pure and Applied Algebra},
	number = {2},
	pages = {107196},
	title = {Virtual concepts in the theory of accessible categories},
	volume = {227},
	year = {2023}}

@article{RT23EUA,
	author = {Rosick{\`y}, J. and Tendas, G.},
	journal = {Selecta Mathematica},
	number = {2},
	pages = {21},
	publisher = {Springer},
	title = {Towards enriched universal algebra},
	volume = {32},
	year = {2026}}

@article{RT25ERegular,
	author = {Rosick{\'y}, J. and Tendas, G.},
	journal = {The Journal of Symbolic Logic},
	title = {Enriched concepts of regular logic},
	year = {2025}}

@article{janelidze2002semi,
	author = {Janelidze, G. and M{\'a}rki, L. and Tholen, W.},
	journal = {Journal of Pure and Applied Algebra},
	number = {2-3},
	pages = {367--386},
	publisher = {Elsevier},
	title = {Semi-abelian categories},
	volume = {168},
	year = {2002}}

@book{borceux2004mal,
	author = {Borceux, F. and Bourn, D.},
	publisher = {Springer Science \& Business Media},
	title = {Mal'cev, protomodular, homological and semi-abelian categories},
	volume = {566},
	year = {2004}}

@article{hermida2000representable,
	author = {Hermida, C.},
	journal = {Advances in Mathematics},
	number = {2},
	pages = {164--225},
	publisher = {Elsevier},
	title = {Representable multicategories},
	volume = {151},
	year = {2000}}
\bibliographystyle{alpha}

\end{document}